\documentclass[twocolumn]{article}          

\usepackage[utf8]{inputenc}
\usepackage{epsf}
\usepackage{epsfig,subfigure}

\usepackage{graphicx}

\usepackage{pifont} 
\usepackage{amsfonts}
\usepackage{amsmath}
\usepackage{amssymb}
\usepackage{amsthm}
\usepackage{xcolor}

\newtheorem{lemma}{Lemma}
\newtheorem{theorem}{Theorem}
\newtheorem{definition}{Definition}
\newtheorem{remark}{Remark}
\newtheorem{corollary}{Corollary}
\newtheorem{proposition}{Proposition}

\usepackage{hyperref}

\DeclareMathOperator\Vor{Vor}

\DeclareMathOperator\Diff{\mathbf{D}}
\DeclareMathOperator\Met{\mathbf{M}}
\DeclareMathOperator\diver{div}
\DeclareMathOperator\interp{I}

\DeclareMathOperator\DDiff{\mathbf{d}}

\def\R{{\mathbb R}}
\def\Z{{\mathbb Z}}
\def\sR{\R}

\newcommand\Id{{\rm Id}}

\def\sm{\setminus}
\def\ve{\varepsilon}

\newcommand\trans{\mathrm T}

\newcommand\cE{{\mathcal E}}
\newcommand\cO{{\mathcal O}}
\newcommand\cT{{\mathcal T}}
\newcommand\cQ{{\mathcal Q}}

\newcommand\cF{{\mathcal F}}
\newcommand\cG{{\mathcal G}}

\newcommand\stext[1]{\ \text{ #1 } \ }

\definecolor{BrickRed}{cmyk}{0, .89, .94, .28} 
\newcommand\alert[1]{\textcolor{BrickRed}{  #1  }}

\def\be{\begin{equation}}
\def\ee{\end{equation}}

\def\<{\langle}
\def\>{\rangle}

\def\orth{\mu}
\newcommand{\iref}[1]{\eqref{#1}}

\usepackage{esint}
\usepackage{enumerate}
\def\DD{{\bf D}}
\def\vp{\varphi}
\def\kappaD{{\boldsymbol \kappa}}

\def\DDMax{\overline \DDiff}
\def\DDMin{\underline \DDiff}

\def\WG{A-NN}

\def\QED{}
\begin{document}

\title{
Sparse Non-Negative Stencils for Anisotropic Diffusion
\thanks{
This work was partly supported by ANR grant MESANGE ANR-08-BLAN-0198.
}
}


\author{
J\'er\^ome Fehrenbach%
\footnote{
Institut de Math\'ematiques de Toulouse,
              Universit\'e Paul Sabatier, 
              31062 TOULOUSE CEDEX 9, France%
}%
\and Jean-Marie Mirebeau%
 \footnote{%
 CNRS, Laboratory CEREMADE, UMR 7534,
           University Paris Dauphine,
           Place du Mar\'echal De Lattre De Tassigny
            75775 PARIS CEDEX 16, France
 }
}

\maketitle

\begin{abstract}
We introduce a new discretization scheme for Anisotropic Diffusion, AD-LBR, on two and three dimensional cartesian grids.
The main features of this scheme is that it is non-negative and has sparse stencils, of cardinality bounded by $6$ in 2D, by $12$ in 3D, despite allowing diffusion tensors of arbitrary anisotropy. The radius of these stencils is not a-priori bounded however, and can be quite large for pronounced anisotropies.
Our scheme also has good spectral properties, which permits larger time steps and avoids e.g.\ chessboard artifacts.

AD-LBR relies on Lattice Basis Reduction, a tool from discrete mathematics which has recently shown its relevance for the discretization on grids of strongly aniso\-tropic Partial Differential Equations \cite{M12}.
We prove that AD-LBR is in 2D asymptotically equivalent to a finite element discretization on an anisotropic Delaunay triangulation, a procedure more involved and computationally expensive.
Our scheme thus benefits from the theoretical guarantees of this procedure, for a fraction of its cost. 
Numerical experiments in 2D and 3D illustrate our results.
\end{abstract}

\paragraph{keywords : } Anisotropic Diffusion, Non-Negative Numerical Scheme, Lattice Basis Reduction.


We consider throughout this paper a bounded smooth domain $\Omega \subset \R^d$, where $d\in \{2,3\}$ denotes the dimension, equipped with a continuous diffusion tensor $\Diff$. 
We do not impose any bound on the diffusion tensor anisotropy, and we are in fact interested in pronounced, non axis-aligned anisotropies. 
Anisotropic diffusion is here understood in the sense of \cite{W98}: the diffusion tensor $\Diff(z)$, at a point $z\in \Omega$, is a symmetric positive definite matrix whose eigenvalues may have different orders of magnitude. Our results are not relevant for isotropic diffusion with a variable scalar coefficient, as in the pioneering work of Perona and Malik \cite{PM90}. 

We address the discretization of the following energy $\cE$, defined for $u \in H^1(\Omega)$:
\be
\label{def:E}
\cE(u) := \int_\Omega \|\nabla u(z)\|^2_{\Diff(z)}dz.
\ee
We denote $\|e\|_M := \sqrt{\<e,M e\>}$, for any $e \in \R^d$, and any $M$ in the set $S_d^+$ of symmetric positive definite $d \times d$ matrices.
Gradient descent for the energy \iref{def:E} has the form of a parabolic PDE: 
\be
\label{pde}
\partial_t u = \diver( \Diff \nabla u).
\ee
This equation, Anisotropic Diffusion, is with its variants at the foundation of powerful image processing techniques. 
Some variants include curvature terms  \cite{OR90}, or diffusion-reaction terms \cite{CG93}. 
Time varying and solution dependent diffusion tensors can also be considered. 
A general exposition can be found in \cite{W98}, where various choices for the definition of the diffusion tensor $\DD$ from the image $u$, adapted to various applications, are proposed and discussed.

Our contribution in the discretization of the energy \iref{def:E} results in improved numerical solutions of \iref{pde}, in terms of accuracy and stability, for a minor increase in complexity. 
This extends to applications, such as Coherence Enhancing Diffusion and Edge Preserving Diffusion \cite{W98}, see the numerical experiments in \S \ref{sec:num}, which involve solving \iref{pde} using a solution dependent diffusion tensor $\DD = \DD(u)$.
For that purpose, one fixes a time step $\Delta T$, and solves for each integer $n \geq 0$ the linear diffusion equation $\partial_t u = \diver(\Diff_n \nabla u)$ on the interval $[n \Delta T, (n+1) \Delta T]$, with $\Diff_n := \Diff(u(n \Delta T))$.
In these applications, the diffusion tensor $\DD(u)$ is typically defined in terms of the structure tensor \cite{W98} of $u$, in such way that diffusion is pronounced within image homogeneous regions, and \emph{tangentially} along image edges, but not across edges.

In two dimensions, AD-LBR strictly speaking is not the first non-negative scheme for anisotropic diffusion: the proof of Theorem 6 in \cite{W98} implicitly defines an alternative 6-point non-negative scheme. 
This alternative scheme does however lack many of the qualities of AD-LBR: it leads to axis aligned artifacts, spectral aberrations, stencils of larger radius, reduced numerical accuracy, and does not extend to 3D. A detailed description and comparison is presented in \S \ref{sec:Schemes}.

Consider a scale parameter $h>0$, and a sampling $\Omega_h$ of the domain $\Omega$ on the cartesian grid $\Z^d$, rescaled by $h$: 
with obvious notations 
\begin{equation*}
\Omega_h := \Omega \cap  h \Z^d.
\end{equation*}
We introduce a novel discretization of the energy \iref{def:E}, referred to as AD-LBR (Anisotropic Diffusion using Lattice Basis Reduction). It is a sum of squared differences of a discrete map $u\in L^2(\Omega_h)$
\be
\label{def:E*}
\cE_h(u) := h^{d-2} \sum_{z\in \Omega_h} \sum_{e \in V(z)} \gamma_z(e)\, |u(z+h e)-u(z)|^2
\ee
The stencils $V(z)\subset \Z^d$, $z\in \Omega_h$, are symmetric and have cardinality at most $6$ in 2D, $12$ in 3D. The coefficients $\gamma_z(e)\geq 0$ are non-negative.
They are constructed using a classical tool from discrete mathematics, Lattice Basis Reduction, which allows to cheaply build efficient stencils for grid discretizations of Partial Differential Equations (PDEs) involving strongly anisotropic diffusion tensors or Riemannian metrics.
This approach has been applied to anisotropic static Hamilton-Jacobi PDEs in \cite{M12}, resulting in a new  numerical scheme: Fast Marching using Lattice Basis Reduction (FM-LBR). Substantial improvements were obtained in comparison with earlier methods, in terms of both accuracy and complexity. 

The paper is organized as follows. 
We describe the stencils of the two dimensional AD-LBR in \S\ref{sec:intro}, and state our main 2D result: the asymptotic equivalence of AD-LBR with a finite element discretization on an Anisotropic Delaunay Triangulation. 
Section \S\ref{sec:3D} provides additional details on the two dimensional stencils of AD-LBR, and describes the three dimensional ones.
The more technical \S\ref{sec:FE} details the proof of the 2D equivalence result stated in \S\ref{sec:intro}. Two and three dimensional numerical experiments are presented in \S \ref{sec:num}, including qualitative and quantitative comparisons with five other numerical schemes.

\section{Description of the scheme, and main results}
\label{sec:intro}

Our numerical scheme, Anisotropic Diffusion using Lattice Basis Reduction (AD-LBR), involves the construction of stencils whose geometry is tailored after the local diffusion tensor.
Its essential feature is non-negativity: the discrete energy $\cE_h(u)$ is written as a sum \iref{def:E*} of squared differences of values of $u$, with non-negative weights $\gamma_z(e) \geq 0$. 
This discretization is consistent if for each $z\in \Omega_h$, and any smooth $u$, 
\be
\label{approxEnergy}
h^d \|\nabla u(z) \|_{\Diff(z)}^2 = h^{d-2} \sum_{e \in V(z)}   \gamma_z(e)\, \<\nabla u(z), h e\>  ^2.
\ee
Indeed, the left hand side approximates the contribution of the ``voxel'' $z+[-h/2,h/2]^d$ to the integral \iref{def:E}, while the right hand side is obtained by inserting the first order approximation $u(z+he) \simeq u(z) + \<\nabla u(z), h e\>$ in \iref{def:E*}.
The identity \iref{approxEnergy} is in turn equivalent to 
\be
\label{DSum}
\Diff(z) = \sum_{e \in V(z)} \gamma_z(e) \, e e^\trans.
\ee
The next lemma shows how to obtain such a decomposition in 2D. 
We denote by $u^\perp := (-b,a)$ the rotation of a vector $u =(a,b)\in \R^2$ by $\pi/2$, in such way that for all $v \in \R^2$:
\begin{equation*}
\<u^\perp, v\> = \det(u,v).
\end{equation*}

\begin{lemma}
\label{lem:Sum}
Let $e_0,e_1,e_2 \in \R^2$ be such that $e_0+e_1+e_2=0$, and $|\det(e_1,e_2)|=1$.
Then for any $D \in S_2^+$, with the convention $e_{3+i} := e_i$:
\be
\label{DSum0}
D = - \sum_{0 \leq i \leq 2} \<e_{i+1}^\perp,D e_{i+2}^\perp\> \, e_i e_i^\trans.
\ee
\end{lemma}

\begin{proof}
Note that $1 
= |\det(e_2,e_0)| = |\det(e_0, e_1)|$.
Denoting by $D'$ the right hand side of \eqref{DSum0}, we obtain 
\begin{align*}
\<e_1^\perp, D' e_1^\perp\> &= -\<e_0^\perp, D e_1^\perp\> \<e_2, e_1^\perp\>^2 -\<e_1^\perp, D e_2^\perp\> \<e_0, e_1^\perp\>^2\\
&= -\<e_0^\perp+e_2^\perp, D e_1^\perp\> =  \<e_1^\perp, D e_1^\perp\>.
\end{align*}
Thus $\|e_1^\perp\|_{D'} = \|e_1^\perp\|_D$. Likewise $\|e_2^\perp\|_{D'} = \|e_2^\perp\|_D$, and $\|e_1^\perp+e_2^\perp\|_{D'}  = \|e_0^\perp\|_{D'} = \|e_0^\perp\|_D = \|e_1^\perp+e_2^\perp\|_D$.
Since $(e_1^\perp, e_2^\perp)$ is a basis of $\R^2$, the result follows.
\QED
\end{proof}

\begin{figure}
\centering
\includegraphics[width=2cm]{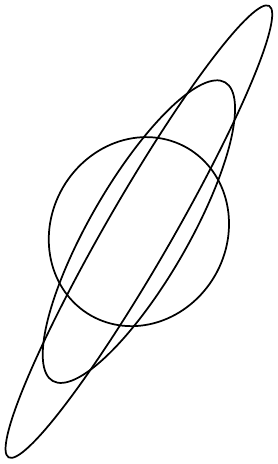}
\hspace{1cm}
\includegraphics[width=4cm]{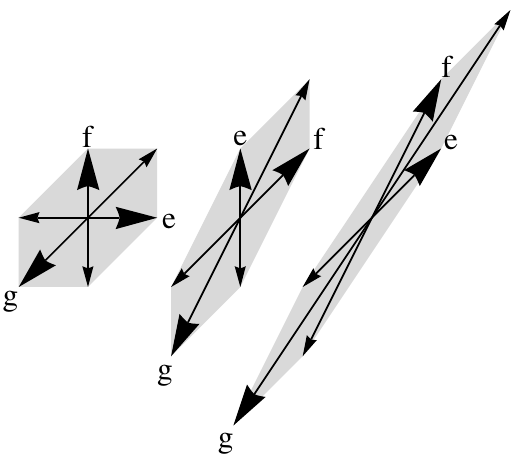}
\caption{
Right: the stencils associated to three matrices $M$ of anisotropy ratios $\kappa(M)$ equal to $1.1$, $3.5$, $8$ respectively. The ellipses $\{z \in \R^2;\, \|z\|_M=1\}$ are shown left; their principal axis is aligned with $(\cos(\pi/3) , \sin(\pi/3))$. More stencils are shown in \cite{M12}.
}
\label{fig:Stencil2d}
\end{figure}

The diffusion tensor $\Diff$ is meant to measure gradients, as in \iref{def:E}. 
In order to measure angles between vectors, we introduce a Riemannian
metric
\footnote{%
The Laplace Beltrami operator associated to $\Met$ does \emph{not} coincide with $\diver(\Diff \nabla \cdot)$, unless $\Diff$ is identically of determinant $1$. This is not an issue for our application.
}
 $\Met$ on the domain $\Omega$, which is proportional to the inverse of $\Diff$: for all $z \in \Omega$
\begin{equation}
\label{def:M}
\Met(z) := \DDiff(z) \Diff(z)^{-1}, \, \text{ where} \, \DDiff(z) := \det(\Diff(z))^\frac 1 d.
\end{equation}
The normalizing factor $\DDiff(z)$ was chosen so as to normalize the metric determinant: $\det(\Met(z)) = 1$. This normalization reflects the fact that the construction of our stencil $V(z)$ depends on the preferred direction of diffusion, and on the amount of anisotropy, whereas the absolute strength of diffusion is irrelevant.
In dimension $d=2$, one easily checks that for any $z\in \Omega$ and any $e,f\in \R^2$, one has 
\begin{equation}
\label{eq:DM}
\<e^\perp, \Diff(z) f^\perp\> = \DDiff(z) \<e, \Met(z) f\>. 
\end{equation}

The AD-LBR is based on decompositions \eqref{DSum}, given by the previous lemma, with a family of vectors $(e_i)_{i=0}^2$ chosen so that the scalar products appearing in \eqref{DSum0} are non-positive.
The adequate concept is that of $M$-obtuse superbase of $\Z^d$ \cite{CS92}.

\begin{definition}
\begin{itemize}
\item
A basis of $\Z^d$ is a family $(e_i)_{i=1}^d$ of elements of $\Z^d$ such that $|\det(e_1, \cdots, e_d)| = 1$. 
\item
A superbase of $\Z^d$ is a family $(e_i)_{i=0}^{d}$ such that $e_0+\cdots+e_d = 0$, and 
$(e_i)_{i=1}^d$ is a basis of $\Z^d$.
\end{itemize}
\end{definition}

\begin{definition}
Let $M \in S_d^+$.
A family $(e_i)_{i\in I}$ of vectors in $\R^d$ is said to be $M$-obtuse if $\<e_i, M e_j\> \leq 0$ for all distinct $i,j \in I$.
\end{definition}

In dimension $d \leq 3$, there exists for each $M \in S_d^+$ at least one $M$-obtuse superbase of $\Z^d$ \cite{CS92}. 
The practical construction of such superbases is discussed in \S \ref{sec:3D},
and based on lattice basis reduction algorithms described in \cite{L1773,S01,NS09} (hence the name of our numerical scheme). This construction has a logarithmic numerical cost $\cO(\ln \kappa(M))$ in the anisotropy ratio of the matrix $M$:
\begin{equation}
\label{def:Kappa}
\kappa(M) := \max_{|u|=|v|=1} \frac{\|u\|_M}{\|v\|_M} = \sqrt{\|M\| \|M^{-1}\|}.
\end{equation}
The AD-LBR energy $\cE_h : L^2(\Omega_h) \to \R_+$, see \eqref{def:E*}, is in two dimensions written in terms of the following stencils and coefficients. Let $z \in \Omega$, and let $e_0, e_1, e_2$ be an $\Met(z)$-obtuse superbase of $\Z^2$. We set 
\be
\label{def:Vz}
V(z) := \{e_0,e_1,e_2,\, -e_0,-e_1,-e_2\},\\
\ee
and for $0 \leq i \leq 2$, with the convention $e_{i+3} := e_i$,
\be
\label{def:GammaZ}
\gamma_z(\pm e_i) := -\frac 1 2 \<e^\perp_{i+1}, \Diff(z) e^\perp_{i+2}\>.
\ee
Lemma \ref{lem:Sum} implies the announced decomposition \eqref{DSum}, and the weights $\gamma_z$ are non-negative in view of \eqref{eq:DM}. These weights $\gamma_z : \Z^2 \to \R_+$, extended by $0$ outside $V(z)$, do not depend on the choice of $\Met(z)$-obtuse superbase $(e_0,e_1,e_2)$, see Lemma \ref{lem:indepBasis}. 
Stencils of the three dimensional AD-LBR are described in \S \ref{sec:3D}, and involve a construction of Selling\footnote{%
The authors would like to thank Professor P.\ Q.\ Nguyen for pointing out this $12$ points 3D stencil, which is simpler and sparser than the $14$ points stencil proposed by the authors in an earlier version of the manuscript.%
} \cite{S1874}.
The above description of the stencils $V(z)$ is suitable for periodic, reflected, and Dirichlet boundary conditions (extending $u$ by zero outside $\Omega_h$ in the latter case). In the case of Neumann boundary conditions, a slight modification is in order:
\begin{equation*}
V(z; \, h) := \{e \in V(z); \, z+ h e \in \Omega_h\}.
\end{equation*}

We have so far established three strongpoints of the AD-LBR:
\begin{description}
\item[Non-negativity.]
Off diagonal coefficients of the symmetric semi-definite $N \times N$ matrix, $N = \#(\Omega_h)$, associated to the energy $\cE_h$ are non-positive, while diagonal coefficients are positive.
\item[Sparsity.] Stencil cardinality is uniformly bounded, without restriction on the anisotropy ratio $\kappa(\Diff(z))$ of the diffusion tensor.
\item[Complexity.] The construction of the stencil $V(z)$, and of the associated coefficients $\gamma_z$, has a logarithmic cost  $\cO(\ln \kappa(\Diff(z)))$ in the anisotropy ratio of the diffusion tensor. 
\end{description}
The next result, Theorem \ref{th:FE}, restricted to the two dimensional case, establishes that AD-LBR is asymptotically equivalent to a more involved and computationally intensive procedure: a finite element discretization of the energy \iref{def:E}, on an \emph{Anisotropic Delaunay Triangulation} (ADT, see \cite{LS03} and below) of the domain $\Omega$. 
Under the assumptions of Theorem \ref{th:FE}, AD-LBR benefits from two additional guarantees, that we state informally and without proof.

\begin{description}
\item[No chessboard artifacts.] Some numerical schemes for anisotropic diffusion suffer from chessboard artifacts, in the sense that periodic artifacts develop at the pixel level. 
Such artifacts cannot develop in finite element discretizations, since they would lead to high frequency oscillations of the finite element interpolant, and therefore to an increase of the energy \iref{def:EFE}. 
The asymptotic equivalence of the AD-LBR with a finite element discretization also rules out these defects.

\item[Spectral correctness.] The $n$-th smallest eigenvalue $\lambda_n(h)$ of the symmetric matrix associated to $h^{-d}\cE_h$ \iref{def:E*}, converges as $h\to 0$ towards the $n$-th smallest eigenvalue $\lambda_n$ of the continuous operator $-\diver(\Diff \nabla)$, for any given integer $n \geq 0$. See Figure \ref{fig:Eigen} page \pageref{fig:Eigen} for an illustration.
This follows from a similar property of the finite element energy $\cE'_h$ \iref{def:EFE}, and from the asymptotic equivalence \iref{EnEpn}.
\end{description}

Our convergence result, Theorem \ref{th:FE} below, is specialized to the case of a square periodic domain, which covers reflecting boundary conditions frequently used in image processing. 
Since the grid discretization must be compatible with the boundary conditions, any scale parameter $h$ appearing in the rest of the paper is assumed to be the inverse of a positive integer:
\begin{equation*}
h \in \{1/n; \, n \geq 1\}.
\end{equation*}


\begin{theorem}
\label{th:FE}
Let $\Omega$ be the unit square $[0,1[^2$, equipped with periodic boundary conditions. Let $\Diff : \overline \Omega \to S_2^+$ be a (periodic) diffusion tensor with Lipschitz regularity, and let $\Met$ be the Riemannian metric defined by \eqref{def:M}. 
When $h$ is sufficiently small, the periodic Riemannian domain $(\Omega,\Met)$ admits an Anisotropic Delaunay Triangulation $\cT_h$, with collection of vertices $\Omega_h := \Omega \cap h \Z^2$. For $u\in L^2(\Omega_h)$, define
\be
\label{def:EFE}
\cE'_h(u) := \int_\Omega \|\nabla  (\interp _{\cT_h} u)(z) \|^2_{\Diff(z)} dz,
\ee
where $\interp_\cT$ denotes the piecewise linear interpolation operator on a triangulation $\cT$.
Then for some constant $c=c(\Diff)$, independent of $u$ and $h$, 
\be
\label{EnEpn}
(1- c h) \cE_h(u) \leq \cE'_h(u) \leq (1+ c h) \cE_h(u).
\ee
\end{theorem}

\begin{figure}
\centering
\includegraphics[width=4cm]{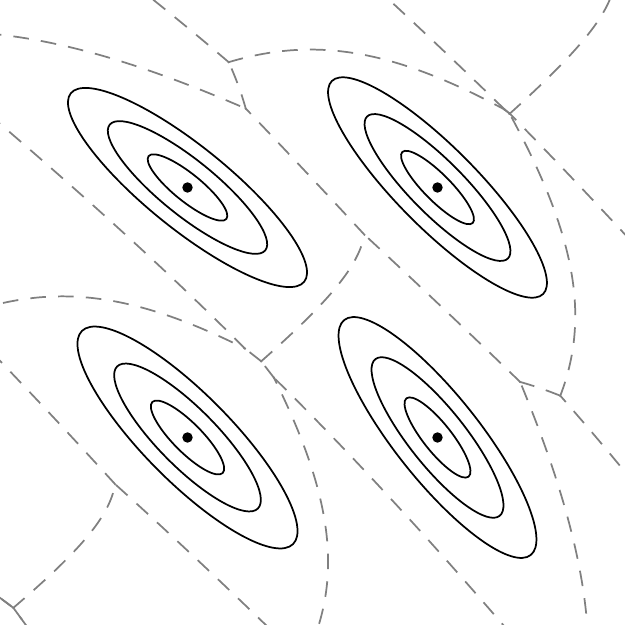}
\hspace{0.1cm}
\includegraphics[width=4cm]{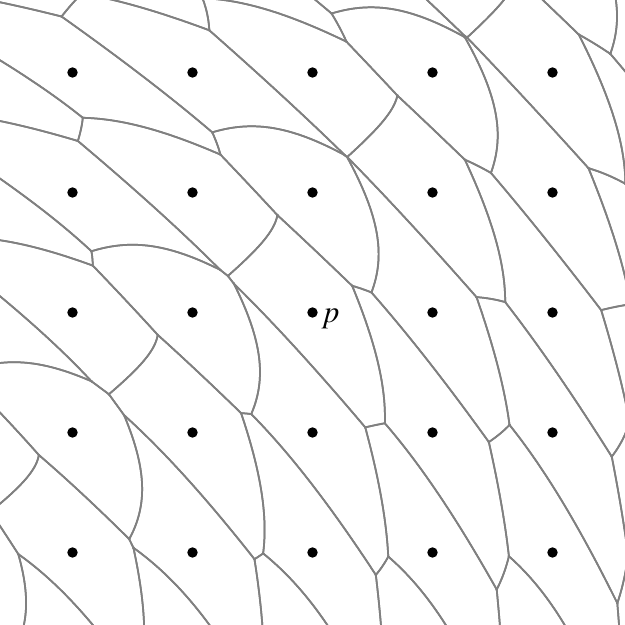}
\caption{
The distance $\delta_p(q)$, from a grid point $p$ to $q\in \R^2$, is defined in terms of the local metric $\Met(p)$, see \eqref{def:DeltaPQ}. The level lines $\{q \in \R^2; \, \delta_p(q) = r\}$ are ellipses (left). The collection of points $q\in \R^2$ closer to $p$ than to any other grid point is the Voronoi region of $p$ (left: the boundaries of Voronoi regions are shown dashed). The Voronoi diagram (right) is the collection of all Voronoi regions.
}
\label{fig:Vor}
\end{figure}

\begin{figure}
\centering
\includegraphics[width=4cm]{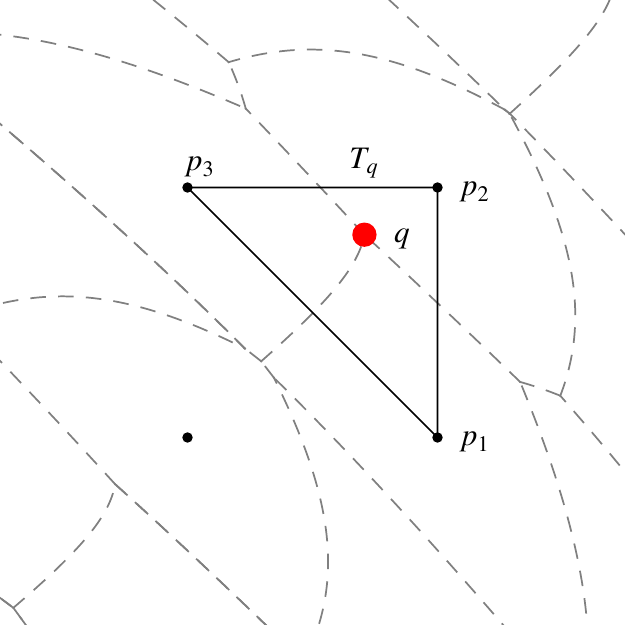}
\hspace{0.1cm}
\includegraphics[width=4cm]{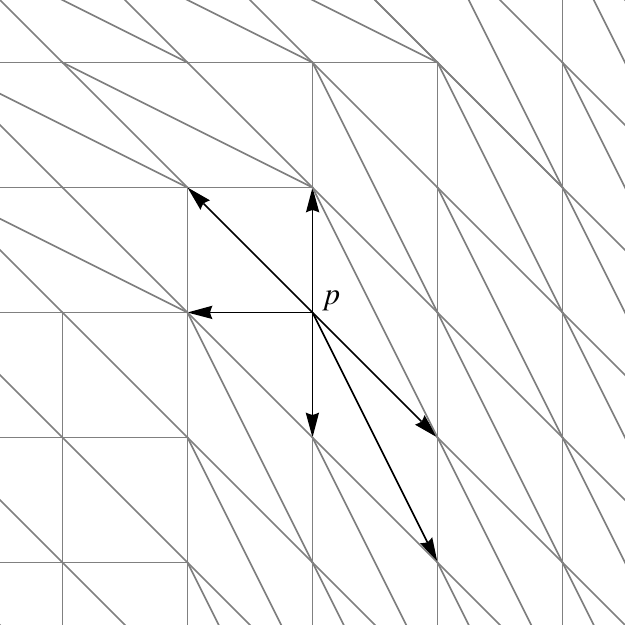}
\caption{
A Voronoi vertex $q$ is a point where the Voronoi regions of at least three grid points $(p_i)_{i=1}^k$ intersect (left: Voronoi region boundaries are shown dashed); here $k=3$ (values $k>3$ are non-generic). 
The dual Voronoi cell $T_q$, generically a triangle, is the convex envelope of the grid points $\{p_i; \, 1 \leq i \leq k\}$ (left). The collection of dual Voronoi cells $T_q$ defines a polygonization $\cQ_h$, generically a triangulation, and the Anisotropic Delaunay Triangulation $\cT_h$ is obtained by arbitrarily triangulating (if necessary) the elements of $\cQ_h$. Here $\cT_h = \cQ_h$ (right). The stencil $V_h(p)$, of a vertex $p$ of $\cT_h$, see \eqref{def:Vhp}, is represented by arrows (right).
}
\label{fig:Del}
\end{figure}

Let us mention that the finite element discretization on an ADT is a more general procedure than AD-LBR, since it does not require the domain $\Omega$ to be sampled on a grid. This flexibility can be used to locally increase the density of vertices, in places where solution $u$ is expected to be less regular, or to insert vertices exactly on $\partial \Omega$ for a better discretization of boundary conditions. (Such refinements are however generally incompatible with image processing since the unknowns, the pixel values, lie by construction on a fixed and given cartesian grid.) Here and as often, the performance of AD-LBR is at the cost of its specialization.




The proof of Theorem \ref{th:FE} is postponed to \S \ref{sec:FE}, but for the sake of concreteness, we describe here the concept of Anisotropic Delaunay Triangulation (ADT) \cite{LS03}. In the rest of this introduction, and in \S \ref{sec:FE}, we assume as in Theorem \ref{th:FE} that the diffusion tensor $\Diff$ is defined on the square $[0,1]^2$ and satisfies periodic boundary conditions. We extend it, as well as the metric $\Met$, to the whole plane $\R^2$ by periodicity. 

We specialize the concept of ADT \cite{LS03}, to the domain $\R^2$ and the collection of vertices $h \Z^2$.
For that purpose, we introduce some notations.
For all $p,q \in \R^2$,  we denote by $\delta_p(q)$ the distance from $p$ to $q$, as measured by the metric at the point $p$: 
\begin{equation}
\label{def:DeltaPQ}
\delta_p(q) := \|q-p\|_{\Met(p)}.
\end{equation}
We denote by $\Delta_h(q)$ the least distance from a point $q\in \R^2$, to the grid $h\Z^2$: 
\be
\label{def:Delta}
\Delta_h(q) := \min_{p \in h\Z^2} \delta_p(q).
\ee
We introduce the Voronoi cell $\Vor_h(p)$ of a grid point $p \in h\Z^2$, which is the collection of points $q\in \R^2$ closer to $p$ than to any other grid point: 
\begin{equation}
\label{def:Vor}
\Vor_h(p) := \{q\in \R^2; \, \delta_p(q) = \Delta_h(q)\}.
\end{equation}
The collection of Voronoi cells is referred to as the Voronoi diagram, see Figure \ref{fig:Vor}. 
A Voronoi vertex is a point $q\in \R^2$ at which at least three distinct Voronoi regions intersect: $(\Vor_h(p_i))_{i=1}^k$, $k \geq 3$, $p_i \in h \Z^2$. We attach to $q$ a \emph{dual} Voronoi cell $T_q$, defined as the convex hull of the points $(p_i)_{i=1}^k$, see Figure \ref{fig:Del}. 

The geometric dual $\cQ_h$, of the Voronoi diagram, is defined as the collection of all dual Voronoi cells $T_q$. 
Note that, generically on the metric $\Met$, no more than three Voronoi regions can intersect at any point in $\R^2$, thus the elements of $\cQ_h$ are generically triangles. 
If $h$ is small enough, we show in \S \ref{sec:FE} (using the Dual Triangulation Theorem in \cite{LS03}) that $T_q$ is a strictly convex polygon, of vertices $(p_i)_{i=1}^k$ with the above notations, and that $\cQ_h$ is a polygonization (generically a triangulation) of $\R^2$, with vertices $h\Z^2$. 

Since the metric $\Met$ and the vertices $h \Z^2$ are periodic (recall that $h = 1/n$ for some integer $n \geq 1$), arbitrarily triangulating the elements of $\cQ_h$, respecting periodicity, yields a periodic triangulation $\cT_h$. 
\begin{definition}[ADT, 
Labelle and Shewchuk \cite{LS03}]
\label{def:ADT}
The triangulation $\cT_h$ obtained by the above construction is referred to as an ADT of the domain $\R^2$, with collection of vertices $h \Z^2$, and underlying Riemannian metric $\Met$.
Since $\cT_h$ is $\Z^2$-periodic, we also regard it as an ADT of the periodic unit square $\Omega$. 
\end{definition}

We establish in \S3.1 the existence of the ADT $\cT_h$. Incidentally, we show in Lemma \ref{lem:DescribeT} (iii) page \pageref{lem:DescribeT} that the angles of the elements of $\cT_h$, measured with respect to the local metric $\Met$, are asymptotically acute. This geometrical property (which holds thanks to our special choice of triangulation vertices, on a grid) is linked to the non-negativity of AD-LBR: indeed, it is known that finite elements discretizations such as \eqref{def:EFE} yield non-negative numerical schemes, and the discrete maximum principle, if the mesh satisfies a non-obtuse angle condition, see Lemma 3.1 in \cite{H10}.

Subsection \S3.2 is devoted to the study of $M$-obtuse superbases of $\Z^2$, and their cousins $M$-reduced bases of $\Z^2$, on which the AD-LBR relies: we discuss their characterization, uniqueness and stability properties.
We study in \S3.3 the 
finite element stencils, defined for $p\in h \Z^2$ by
\begin{equation}
\label{def:Vhp}
V_h(p) := \{e \in \Z^2; [p,p+h e]\text{ is an edge of } \cT_h\},
\end{equation}
see Figure \ref{fig:Del} (right).
We show that $V_h(p)$ coincides with the AD-LBR stencil $V(p)$, unless the lattice $\Z^2$ admits a basis \emph{almost orthogonal} with respect to the scalar product associated to $\Met(p)$, see Lemma \ref{lem:EqualUnlessAlmostOrthogonal}. 
This is tied to the fact that orthogonal grids admit several (usual) Delaunay triangulations.
Overcoming this technical difficulty, we conclude the proof of Theorem \ref{th:FE}.

Note that the construction of the ADT $\cT_h$ is not easy to parallelize, 
in particular when anisotropy is pronounced since the Voronoi regions of far away points interact.  The construction of $\cT_h$ also involves solving polynomial equations of degree four, because Voronoi regions boundaries are conics, and Voronoi vertices must be identified at their intersections. In contrast, the AD-LBR stencils are independent of each other, and the numerical cost of their construction only grows logarithmically with the metric anisotropy.

\section{
Construction of obtuse superbases, and three dimensional stencils
}
\label{sec:3D}

\begin{figure}
\centering
\begin{tabular}{cc}
\includegraphics[width=4cm]{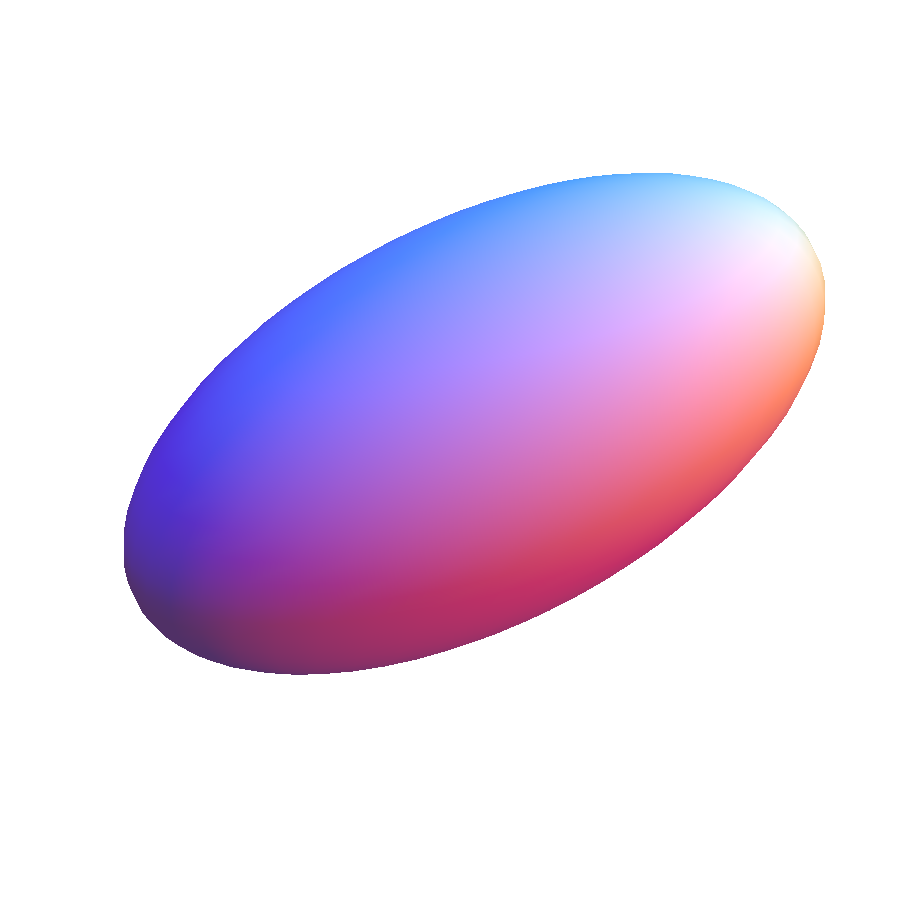}
&
\includegraphics[width=4cm]{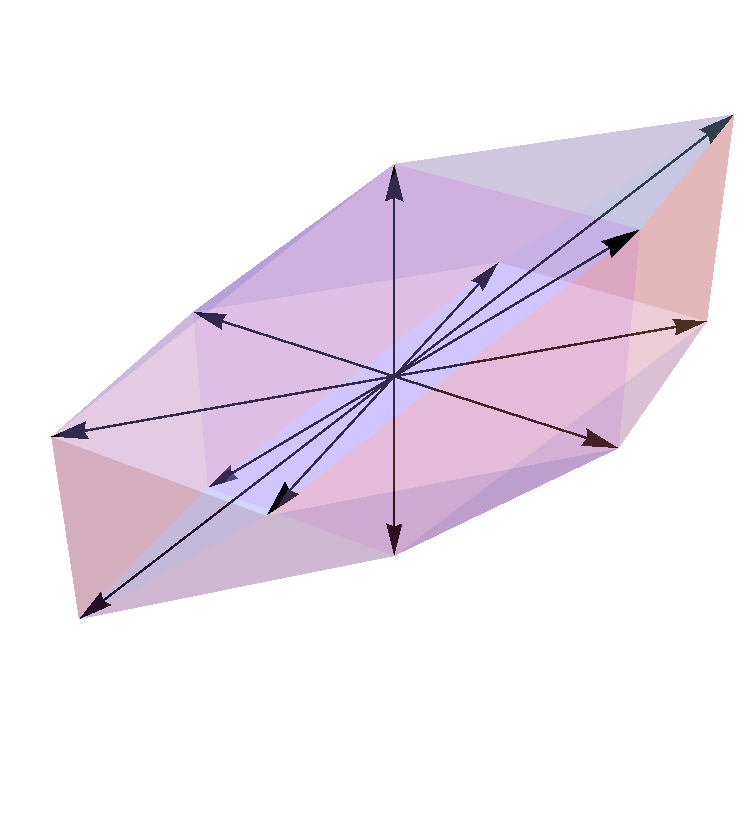}
\vspace{-0.5cm}
\\
{\raise 0.3cm \hbox{
\includegraphics[width=3.5cm]{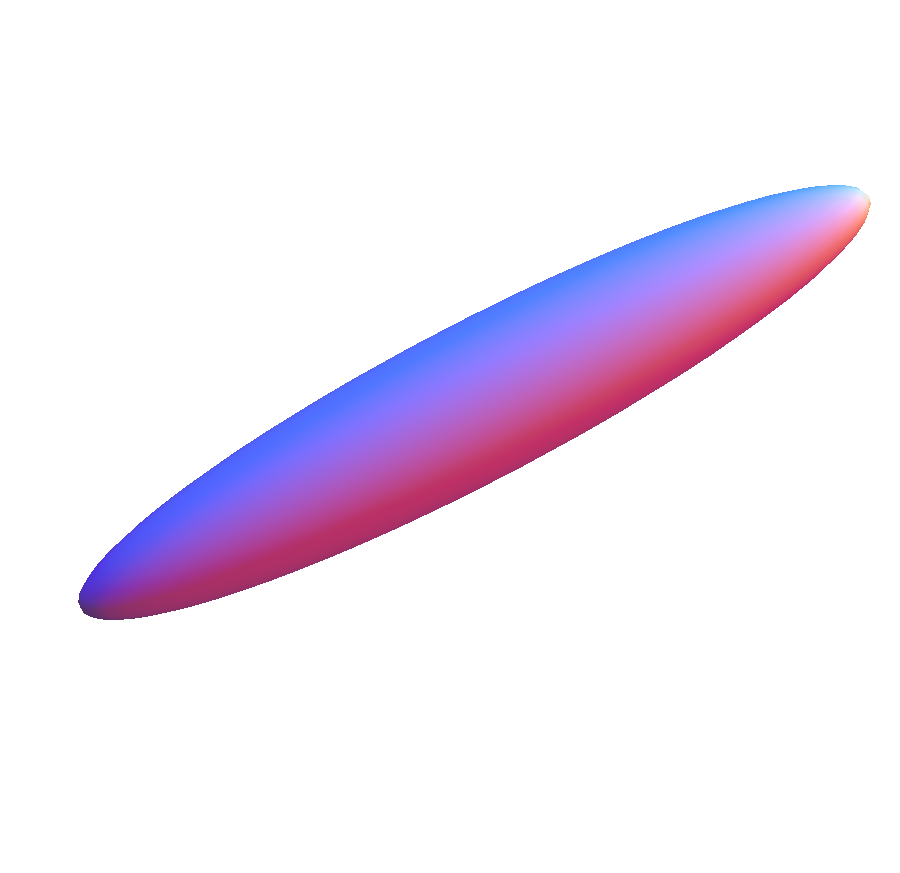}
}}
&
\hspace{-1.2cm}
\includegraphics[width=4.5cm]{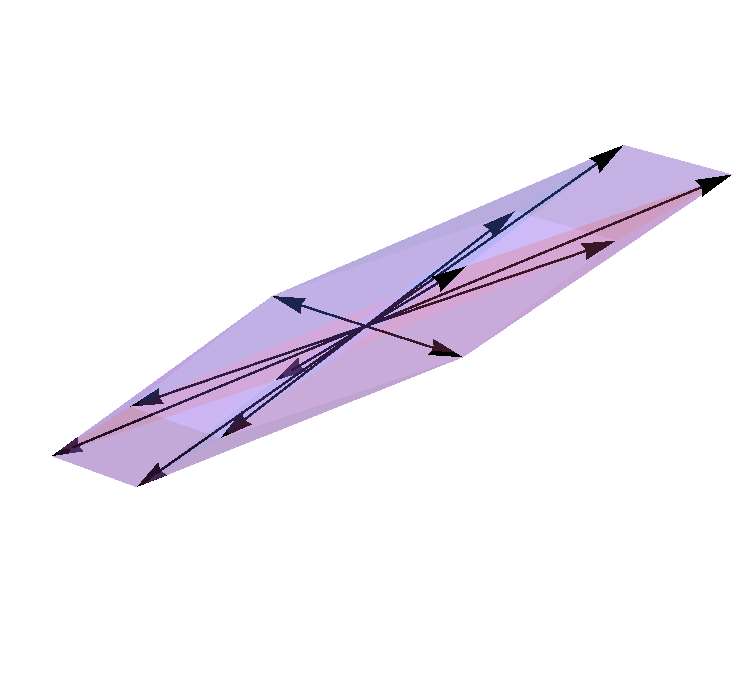}
\end{tabular}
\caption{
Right: stencil of the AD-LBR, for a symmetric matrix of eigenvector $M$ of anisotropy ratio $\kappa(M)$ equal to $2$ (top) or $6$ (bottom). The anisotropy is of ``needle'' type: the two largest eigenvalues of $M$ are equal, and the needle orientation is given by the vector $(4,2,3)$. The ellipsoid $\{z\in \R^3;\, \|z\| \leq 1\}$ is shown left.
}
\end{figure}

Algorithms for the construction of privileged bases of lattices, consisting of short and almost orthogonal vectors, have attracted an important research effort from the mathematical community, over a long period of time. The first such algorithm dates back to Lagrange \cite{L1773}, and is restricted to two dimensional lattices. 
Methods for high dimensional lattices, such as the LLL algorithm \cite{LLL82}, are of key importance for integer programming and cryptography \cite{NS01}.
AD-LBR is based on the original algorithm of Lagrange \cite{L1773}, and on its recent extension to three dimensional lattices \cite{S01,NS09}. These methods output a basis of $\Z^d$ reduced in the sense of Minkowski, which in dimension $d \leq 4$ is equivalent to the following definition.


We denote by $e_1\Z+ \cdots+e_k\Z$ the sub-lattice of $\Z^d$ generated by vectors $e_1, \cdots, e_k\in \Z^d$. This sub-lattice equals $\{0\}$ by convention if $k=0$.

\begin{definition}
\label{def:ReducedBasis}
An $M$-reduced basis of $\Z^d$, where $d\leq 4$ and $M \in S_d^+$, is a basis $(e_1, \cdots, e_d)$ of $\Z^d$ such that 
\begin{equation}
\label{BasisMinimality}
\|e_i\|_M  = \min \{ \|e\|_M; \, e \in \Z^d \sm (e_1 \Z+\cdots+e_{i-1} \Z)\}.
\end{equation}
\end{definition}

For each $d \leq 4$, and each $M\in S_d^+$, there exists at least one $M$-reduced basis \cite{NS09}. In contrast, there exists $M\in S_5^+$ for which no basis of $\Z^d$ satisfies \eqref{BasisMinimality}.
The norms of the elements $(e_i)_{i=1}^d$ of an $M$-reduced basis,
\begin{equation}
\label{def:MinkowskiMinima}
\lambda_i(M) := \|e_i\|_M,
\end{equation}
are called the Minkowski minima, and are independent of the choice of $M$-reduced basis.
In particular, $e_1$ is the shortest vector of $\Z^d$, with respect to the norm $\|\cdot\|_M$, and $e_2$ is the shortest linearly independent vector. 

\begin{lemma}
\label{lem:MinkowskiBounds}
For any $M \in S_d^+$, $1 \leq i \leq d$,
\begin{equation*}
\|M^{-\frac 1 2} \|^{-\frac 1 2} \leq \lambda_i(M) \leq \|M\|^\frac 1 2.
\end{equation*}
\end{lemma}

\begin{proof}
Note that $\|M^{-\frac 1 2}\|^{-\frac 1 2} \|e\| \leq \|e\|_M \leq \|M\|^\frac 1 2 \|e\|$, for any $e \in \R^2$. In addition: (i) any $e \in \Z^2\sm \{0\}$ satisfies $\|e\| \geq 1$, and (ii)
the set $\Z^d \sm (e_1 \Z+ \cdots+e_{i-1} \Z)$ appearing in \eqref{BasisMinimality} always contains at least one element $e$ of the canonical basis of $\R^d$, so that $\|e\| \leq 1$. The announced result easily follows.
\QED
\end{proof}


We emphasize that obtaining an $M$-reduced basis, i.e.\ solving the minimization problems \iref{BasisMinimality}, is both simple and cheap numerically. In dimension $d=2$, this is the object of Lagrange's algorithm \cite{L1773} (later rediscovered by Gauss and often erroneously called Gauss's algorithm, see \cite{NS09}): initialize $(e,f)$ as the canonical 
 basis of $\Z^2$, and 
 \begin{align}
\label{GaussAlgo}
&\text{\bf Do } (e,\ f) := (f, \ e - \mathrm{Round}(\<e, M f\>/\|f\|_M^2 ) \,f),\\
\nonumber
&\text{\bf while }\|e\|_M > \|f\|_M.
\end{align}
This algorithm can be regarded as a two dimensional geometrical generalization of greatest common divisor computation. It can be extended to higher dimension and, in dimension up to four, outputs an $M$-reduced basis after at most $\cO(\ln \kappa(M))$ iterations \cite{NS09}, each consisting of $\cO(1)$ operations among reals. 


The elements of an $M$-reduced basis are heuristically never very far from being orthogonal, as illustrated by the following lemma.
\begin{lemma}
\label{lem:IneqScal}
Let $M \in S_d^+$, $d \leq 4$, and let $(e_1, \cdots, e_d)$ be an $M$-reduced basis. Then for any $i,j \in \{1, \cdots, d\}$,
\begin{equation}
\label{IneqScal}
2 |\<e_i , M e_j \>| \leq \|e_i\|_M^2.
\end{equation}
\end{lemma}

\begin{proof}
Since $\|e_k\|_M$ is an increasing function of $k \in \{1,\cdots, d\}$, we may assume that $i<j$. If follows from \eqref{BasisMinimality} that $\|e_j\|_M \leq \|e_j+e_i\|_M$, and $\|e_j\|_M \leq \|e_j - e_i\|_M$. Squaring these inequalities, and developing the scalar products, we obtain the announced result.
\QED
\end{proof}

\begin{corollary}
\label{corol:ObtuseSuperBaseFromReducedBasis2}
Let $M \in S_2^+$, and let $(e,f)$ be an $M$-reduced basis such that $\<e,M f\> \leq 0$. Then $(e,f,g)$ is an $M$-obtuse superbase of $\Z^2$, with $g := -e-f$. In addition
\begin{equation}
\label{NegativeScalProd}
\<e, M g\> \leq -\|e\|_M^2/2, \quad \<f, M g\> \leq -\|f\|_M^2/2.
\end{equation}

\end{corollary}

\begin{proof}
The previous Lemma implies
$
\<e, M (e+f)\> \geq \|e\|_M^2-|\<e,M f\>| \geq \frac 1 2 \|e\|_M^2.
$
Likewise $\<f, M (e+f)\> \geq \frac 1 2 \|f\|_M^2$. The result follows. 
\QED
\end{proof}
The practical construction of the two dimensional AD-LBR stencil at a point $z\in \Omega$ amounts to (i) compute an $\Met(z)$-reduced basis $(e,f)$ using Lagrange's algorithm \eqref{GaussAlgo}, (ii) replace $f$ with $-f$, if necessary, so that $\<e, M f\> \leq 0$, and (iii) define the stencil $V(z)$ and the weights $\gamma_z$ in terms of the $M$-obtuse superbase $(e,f,g)$ of $\Z^2$, where $g=-e-f$, as described in \eqref{def:Vz} and \eqref{def:GammaZ}.

The rest of this section is devoted to the description of the three dimensional AD-LBR stencils.
In constrast with the two dimensional case, the construction of the 3D stencil $V(z)$ at a point $z\in \Omega$, involves a $\Diff(z)$-obtuse basis, instead of an $\Met(z)$-obtuse basis.
\begin{proposition}
\label{prop:MakeObtuse}
Let $D\in S_3^+$, and let $(e_1,e_2,e_3)$ be a $D$-reduced basis. Let $b_i := \ve_i e_{\sigma(i)}$, for all $1\leq i \leq 3$, where the signs $\ve_1, \ve_2, \ve_3 \in \{-1,1\}$, and the permutation $\sigma$ of $\{1,2,3\}$ are chosen so that 
\begin{equation}
\label{SignedOrderedIneq}
|\<b_1, D b_2\>| \leq \min \{ - \<b_1, D b_3\>, \, -\<b_2, D b_3\>\}.
\end{equation}
Then the following is a $D$-obtuse superbase:
\begin{equation}
\label{3DSuperBase}
\left\{
\begin{array}{cl}
(b_1,b_2,b_3,-b_1-b_2-b_3) & \text{if } \<b_1, D b_2\> \leq 0,\\
(-b_1,b_2,b_1+b_3,-b_2-b_3) & \text{otherwise.}
\end{array}
\right.
\end{equation}
\end{proposition}

\begin{proof}
To achieve \eqref{SignedOrderedIneq}, one can choose $\sigma$ such that $b'_i := e_{\sigma(i)}$ satisfies
$|\<b'_1, D b'_2\>| \leq |\<b'_1, D b'_3\>| \leq |\<b'_2, D b'_3\>|$. Then choose the signs $(\ve_i)_{i=1}^3$ such that $b_i := \ve_i b'_i$ satisfies $\<b_1, D b_3 \> \leq 0$ and $\<b_2, D b_3 \> \leq 0$.

The two families of vectors appearing in \eqref{3DSuperBase} are clearly superbases. 
We thus only need to show that they are $D$-obtuse; in other words that $\<e,D f\> \leq 0$ for any two distinct elements $e,f$ of these families. Note that for all distinct $i,j\in \{1,2,3\}$, using \eqref{IneqScal},
\begin{equation*}
2 |\<b_i, D b_j\>| \leq \|b_i\|_D^2.
\end{equation*}

In the case where $\<b_1,D b_2\> \leq 0$, the pairwise scalar products between $b_1, b_2, b_3$ are non-positive by construction. In addition 
\begin{align*}
& 2 \<b_1+b_2+b_3, D b_1\> \\
\geq & (\|b_1\|_D^2 - 2 |\<b_1, D b_2\>|)
 + (\|b_1\|^2_D - 2 |\<b_1, D b_3\>|) \geq 0.
\end{align*}
Likewise $ \<b_1+b_2+b_3, D b_i\> \geq 0$ for all $i \in \{1,2,3\}$, which concludes the proof.

We next turn to the second case, where $\<b_1,D b_2\> \geq 0$. 
Enumerating all scalar products we obtain
\begin{align*}
 \<  b_1, D(b_1+b_3)\> &\geq \|b_1\|_D^2 - |\<b_1,D b_3\>| \geq 0,\\
 \<  b_1,D(-b_2-b_3)\> &= -\<b_1, Db_2\> - \<b_1, Db_3\> \geq 0,\\
- \<  b_2, D(b_1+b_3)\> &= -\<b_2,Db_1\> - \<b_2,Db_3\> \geq 0,\\
 \<  b_2, D(b_2+b_3)\> &\geq \|b_2\|^2_D - |\<b_2,D b_3\>| \geq 0,
\end{align*}
and finally
\begin{align*}
&2\<b_1+b_3,D(b_2+b_3)\> \geq
 2 \<b_1,Db_2\> \\
 &+ (\|b_3\|^2 - 2 |\<b_1,Db_3\>|)  + (\|b_3\|^2 - 2 |\<b_2,Db_3\>|) \geq 0.
\end{align*}
This concludes the proof. 
\QED
\end{proof}

In view of the previous Proposition, obtaining a $D$-obtuse superbase of $\Z^3$ has numerical cost $\cO(\ln \kappa(D))$. Indeed a $D$-reduced basis needs to be computed in a preliminary step, after what Proposition \ref{prop:MakeObtuse} is applied for a negligible $\cO(1)$ cost. An alternative method for the construction of $D$-obtuse superbases of $\Z^3$ is presented in \cite{CS92} and in appendix B of \cite{BK10}, but its numerical complexity is not known to the authors.

The three dimensional AD-LBR is defined by the following stencils and coefficients. 
Let $z\in \Omega$, let
$D := \Diff(z)$, and let $(e_i)_{i=0}^3$ be a $D$-obtuse superbase of $\Z^3$. 
We set 
\begin{equation*}
V(z) := \{ e_k \times e_l; \ k,l \in \{0, 1,2, 3\}, \ k \neq l\},
\end{equation*}
and if $\{i,j,k,l\} = \{0,1,2,3\}$, $i\neq j$ and $k \neq l$, then
\begin{equation*}
\gamma_z(e_k \times e_l) := -\frac 1 2 \<e_i, D e_j\>.
\end{equation*}
As announced, $\#(V(z)) = 12$, and the weights $\gamma_z$ are non-negative. The proof of the scheme consistency \eqref{DSum}, due to Selling \cite{S1874}, is reproduced in the next lemma for completeness. 
A generalization, appearing in Appendix B of \cite{BK10}, allows in arbitrary dimension to build a non-negative decomposition of the form \eqref{DSum3D} from a $D$-obtuse superbase of $\Z^d$. However the non existence of such a superbase, for some matrices $D \in S_4^+$, forbids a straightforward extension of AD-LBR to higher dimension. 

\begin{lemma}[Selling \cite{S1874}]
Let $(e_i)_{i=0}^3$ be a superbase of $\Z^3$. 
For all $i,j,k,l$ such that $\{i,j,k,l\} = \{0,1,2,3\}$, $i<j$, and $k<l$, let 
$c_{ij} := e_k\times e_l$. Then, for any $D \in S_3^+$:
\begin{equation}
\label{DSum3D}
D = -\sum_{0 \leq i<j \leq 3} \<e_i, D e_j\> c_{ij} c_{ij}^\trans. 
\end{equation}
\end{lemma}
\begin{proof}
Let $i,j,k,l$ be as in the definition of $c_{ij}$. Then
\begin{equation*}
\<e_i, c_{ij}\> = \<e_i,e_k \times e_l\> = \det(e_i,e_k,e_l) \in \{-1,1\},
\end{equation*}
since $(e_i,e_k,e_l)$ is a basis of $\Z^3$. Also
\begin{align*}
\<e_j, c_{ij}\> &= \<-e_i-e_k-e_l, \, e_k \times e_l\>\\
 &= - \<e_i, \ e_k \times e_l\> = -\<e_i, c_{ij}\>.
\end{align*}
In addition, clearly, $\<e_k, c_{ij}\> = \<e_l, c_{ij}\> = 0.$
Denoting by $D'$ the right hand side of \eqref{DSum3D}, we obtain as a result
\begin{align*}
\<e_0, D' e_0\> &= - \<e_0, D e_1\> - \<e_0, D e_2\> - \<e_0, D e_3\> \\
&=  \<e_0,D(- e_1-e_2-e_3)\> =  \<e_0, D e_0\>.\\
\<e_0, D' e_1\> &= - \<e_0, D e_1\> \<e_0, c_{01}\> \<e_1, c_{01}\> = \<e_0, D e_1\>.
\end{align*}
Likewise $\<e_i, D' e_j\> = \<e_i, D e_j\>$ for all $i,j \in \{1,2,3,4\}$. It follows as announced that $D = D'$.
\QED
\end{proof}

\section{Equivalence to a finite element discretization}
\label{sec:FE}

This section is devoted to the proof of Theorem \ref{th:FE}: the asymptotic equivalence of AD-LBR with a finite element discretization on an Anisotropic Delaunay Triangulation (ADT). We use the notations of \S \ref{sec:intro}. 
The existence of the ADT $\cT_h$ is established in the first subsection, for $h$ sufficiently small, as well as a few of its properties. 
The second subsection is devoted to the study of $M$-reduced bases. Theorem \ref{th:FE} is proved in the third subsection, by comparing the stencils of the AD-LBR and of the finite element discretization. 

We denote by $\kappaD$ the maximum anisotropy ratio \eqref{def:Kappa} of the diffusion tensor 
\begin{equation}
\label{def:KappaD}
\kappaD := \max_{z \in \Omega} \kappa(\Diff(z)).
\end{equation}
Observing that $\kappa(\Diff(z)) = \kappa(\Met(z))$, and recalling that $\det(\Met(z))=1$, one easily checks that 
\begin{equation}
\label{KMK}
\kappaD^{-\frac 1 2} \|e\| \leq \|e\|_{\Met(z)} \leq \kappaD^\frac 1 2 \|e\|,
\end{equation}
for all $z\in \Omega$ and all $e \in \R^2$.

\subsection{Existence of an ADT} 
\label{subsec:ExistsADT}

Our first lemma provides an uniform bound on the size of the Voronoi regions, see Figure \ref{fig:Del}, involved in the construction of the ADT. 

\begin{lemma}
\label{lem:Diam}
\begin{enumerate}[(i)]
\item For all $r \in \R^2$, one has $\Delta_h(r) \leq \kappaD^\frac 1 2  h$.
\item If $p,q\in h \Z^2$, and $r \in \Vor_h(p) \cap \Vor_h(q)$, then ${\|p-r\| \leq \kappaD h}$ and $\|p-q\| \leq 2 \kappaD h$. 
\end{enumerate}
\end{lemma}

\begin{proof}
Point (i). Rounding the coordinates of $r$ to a nearest multiple of $h$, we obtain a point $p \in h\Z^2$ such that $\|p-r\| \leq h$. Recalling \eqref{KMK} we obtain $\delta_p(r) \leq \kappaD^\frac 1 2 h$, and therefore $\Delta_h(r) \leq \kappaD^\frac 1 2 h$ in view of \eqref{def:Delta}. 

Point (ii).  We have $\kappaD^{-\frac 1 2 } \|p-r\|  \leq \delta_p(r) = \Delta_h(r) \leq \kappaD^\frac 1 2 h$. Thus $\|p-r\| \leq \kappaD h$, and likewise $\|q-r\| \leq \kappaD h$. 
Finally, by the triangle inequality, $\|p-q\| \leq \|p-r\|+\|q-r\| \leq 2 \kappaD h$. 
\QED\end{proof}

Following the notations of \cite{LS03}, we denote by $\tau(p,q)$, $p,q \in \R^2$, the smallest constant $\tau \geq 1$ such that 
\begin{equation*}
\tau^{-1} \delta_p(r) \leq \delta_q(r) \leq \tau \delta_p(r), \quad \text{for all } r \in \R^2.
\end{equation*}
Equivalently, in the sense of symmetric matrices, 
\begin{equation}
\label{def:tauM}
\tau^{-2} \Met(p) \leq \Met(q) \leq \tau^2 \Met(p).
\end{equation}
We also define a quantity $\tau_h\geq 1$, closely related to the modulus of continuity of the metric $\Met$:
\begin{equation}
\label{def:TauH}
\tau_h := \max\{ \tau(p,q); \, \|p-q\| \leq 2 \kappaD h\}. 
\end{equation}
One has $\tau_h \to 1$ as $h\to 0$, for any continuous metric $\Met$ (indeed $\Met$ is periodic and therefore uniformly continuous). If $\Met$ is Lipschitz, as assumed in Theorem \ref{th:FE}, then $\tau_h = 1+\cO(h)$.

We show in the next lemma the existence of an ADT, by applying the main result of \cite{LS03}, under the assumption that $\tau_h$ is sufficiently small. 
More precisely, we assume in the rest of this subsection that 
\begin{equation}
\label{assumHADT}
\tau_h < \sqrt{1+\kappaD^{-2}}.
\end{equation}
\begin{lemma} 
\label{lem:Shew}
\begin{enumerate}[(i)]
\item If $p,q\in h\Z^2$, $p\neq q$, and $r\in \Vor_h(p) \cap \Vor_h(q)$, then 
$
\delta_p(r) < \delta_p(q)/\sqrt{\tau(p,q)^2-1}.
$
\item The geometric dual $\cQ_h$ of the Voronoi diagram is, as announced in \S \ref{sec:intro}, a polygonization of $\R^2$ into strictly convex polygons, with vertices $h \Z^2$. 
\end{enumerate}
\end{lemma}

\begin{proof}
Point (i). We may assume that $\tau(p,q)>1$, otherwise there is nothing to prove. Point (ii) of Lemma \ref{lem:Diam} implies that $\|p-q\| \leq 2 \kappaD h$, thus 
\begin{equation*}
\sqrt{\tau(p,q)^2-1} \leq \sqrt{\tau_h^2-1} < \kappaD^{-1}. 
\end{equation*} 
On the other hand $\delta_p(q) \geq \kappaD^{-\frac 1 2} \|q-p\| \geq \kappaD^{-\frac 1 2} h$, and $\delta_p(r) \leq \Delta_h(r) \leq \kappaD^\frac 1 2 h$. The announced inequality follows.

Point (ii). We apply Theorem 7 (Dual Triangulation Theorem) in \cite{LS03}. Since the domain $\R^2$ has no boundary, it suffices to check that all the Voronoi arcs and vertices are \emph{wedged}, see \cite{LS03}. This condition means that for any $p,q \in h \Z^2$ such that $p \neq q$, and any $r \in \Vor_h(p)\cap\Vor_h(q)$, one has $(r-q) \Met(q) (p-q) >0$, and likewise exchanging the roles of $p$ and $q$. Heuristically, it expresses the acuteness of some angles measured in the local metric. Lemma 5 in \cite{LS03} shows that this condition follows from point (i) of this lemma, which concludes the proof.
\QED\end{proof} 

We recall that $\cT_h$ is the triangulation obtained by arbitrarily triangulating the polygonization $Q_h$ of the previous lemma, respecting periodicity, see Definition \ref{def:ADT}. Generically $Q_h$ is already a triangulation, hence $\cT_h = Q_h$, see \S 1.  
The Voronoi regions $\Vor_h$, and the triangulation $\cT_h$, are illustrated in Figures \ref{fig:Vor} and \ref{fig:Del}. 

The next lemma provides estimates of the diameter, the area, and the angles of the elements of $\cT_h$.
These geometrical properties also have an interpretation in the context of lattices: (ii) shows that the edges of any triangle $T \in \cT_h$ define a superbase $(e,f,g)$ of $\Z^2$, and (iii) that this superbase is \emph{almost} $\Met(z)$-obtuse, for any $z\in T$. 


Note that the vertices $p,q,r$ of any triangle $T \in \cT_h$ satisfy by construction 
\begin{equation}
\label{IntersectVor}
\Vor_h(p) \cap \Vor_h(q) \cap \Vor_h(r) \neq \emptyset.
\end{equation}

\begin{lemma}
\label{lem:DescribeT}
Denote by $h e, h f, h g$ the edges of a triangle $T\in \cT_h$, where $e,f,g \in \Z^2$ are oriented so that $e+f+g=0$.
Then 
\begin{enumerate}[(i)]
\item
$\max \{\|e\|, \|f\|, \|g\| \} \leq 2 \kappaD$.
\item
$|\det(e,f)| = 1$, thus $|T| = h^2/2$.
\item
$\<e,\Met(z) f\> \leq \theta_h$, for any $z\in T$, where $\theta_h \to 0$ as $h \to 0$. (Explicitly: $ \theta_h = \kappaD (3+9 \tau_{2h}^2) (\tau_{2h}^2-1)$) 
\end{enumerate}
\end{lemma}

\begin{proof}
Point (i). We denote by $p,q,r$ the vertices of $T$, ordered in such way that $p+h e = q$, $q+h f = r$, $r+h g = p$.
The announced estimate follows from \iref{IntersectVor}, 
and from point (ii) of Lemma \ref{lem:Diam}.

Point (ii).
Since $\cT_h$ is a conforming triangulation, the intersection of $T$ with the collection $h \Z^2$ of all vertices of $\cT_h$ consists of only three points: the vertices $p,q,r$ of $T$. Thus the triangle of vertices $-e, 0, f$, homothetic to $T$, contains no point of integer coordinates but its vertices. This implies that $(e,f)$ is a basis of $\Z^2$, hence $|\det(e,f)| = 1$, as announced. 

Point (iii).
The pairwise distances between $p,q,r$ are bounded by $2 \kappaD h$, see point (i), and since $z \in T$ so are the pairwise distances between $p,q,r,z$. Defining $s := p-q+r \in h \Z^2$, and observing that $\|s-p\| = \|r-q\| \leq 2 \kappaD h$, we find that the pairwise distances between $p,q,r,z,s$ are bounded by $4 \kappaD h$.


Let $x\in \Vor_h(p) \cap \Vor_h(q) \cap \Vor_h(r)$.
We have $\delta_p(x) = \delta_q(x) = \delta_r(x) = \Delta_h(x)  \leq \delta_s(x)$, thus 
\be
\label{deltaSum}
\delta_s(x)^2 \geq \delta_p(x)^2-\delta_q(x)^2 + \delta_r(x)^2.
\ee
(For intuition: in a classical Delaunay triangulation, $x$ would be the circumcenter of $T$, and \eqref{deltaSum} would state that $s$ is outside the circumcircle of $T$.)
Denoting $M := \Met(z)$, and $\delta := \Delta_h(x)$, we obtain
\begin{align}
\nonumber
| \delta_p(x)^2 -  \| x-p\|^2_M | &\leq \delta_p(x)^2 (\tau(p,z)^2 - 1)\\
\label{deltaNormDiff}
& \leq \delta^2 (\tau_{2h}^2 - 1),
\end{align}
using Lemma \ref{lem:Diam}, and likewise for $q,r$. 
We also have 
\begin{align*}
\delta_s(x) &= \|p-q+r-x\|_{\Met(s)}\\
& \leq \|p-x\|_{\Met(s)}+\|q-x\|_{\Met(s)} + \|r-x\|_{\Met(s)}\\
& \leq 3 \delta \tau_{2h}.
\end{align*}
Thus, proceeding as in \iref{deltaNormDiff},
\begin{equation*}
|\delta_s(x) - \|s-x\|^2_M | \leq \delta_s(x)^2 (\tau_{2 h}^2-1) \leq 9 \delta^2  \tau^2_{2h} (\tau_{2 h}^2-1) .
\end{equation*}
Inserting in \iref{deltaSum} these estimates of $\delta_\star(x)$, $\star\in \{p,q,r,s\}$, and using the fact that $\delta \leq \kappaD^\frac 1 2 h$, see Lemma \ref{lem:Diam}, we obtain after expansion the announced estimate of $\<e, M f\>$.
\QED\end{proof}

We next rewrite the finite element energy $\cE'_h$ \eqref{def:EFE} in a form similar to that of the AD-LBR energy $\cE_h$ \eqref{def:E*}. 
Let $\vp_p^h : \R^2 \to \R$ be the piecewise linear function on $\cT_h$ such that $\vp_p^h(p)=1$, and $\vp_p^h(q)=0$ for any vertex $q\in h \Z^2$ distinct from $p$.
This is the classical ``hat function'' encountered in finite element analysis.
For all $p \in h\Z^2$, $e\in \Z^2\sm \{0\}$, let
\begin{equation}
\label{def:GammaH}
\gamma^h_p(e) := - \frac 1 2 \int_{\sR^2} \<\nabla \vp_p^h(z), \Diff(z) \nabla \vp_{p+he}^h(z)\> dz
\end{equation}
Clearly, $\gamma_p^h(e) = 0$ if $[p,p+he]$ is not an edge of $\cT_h$, in other words if $e$ does not belong to the stencil $V_h(p)$, defined in \eqref{def:Vhp}.
We express in the next lemma the finite element energy $\cE'_h$ \eqref{def:EFE} in terms of the stencils $V_h$ and of the (potentially negative) weights $\gamma^h_p$.

\begin{lemma}
\label{lem:ESumGammaH}
For any $u \in L^2(\Omega_h)$, extended by periodicity to $h\Z^2$, one has 
\be
\label{EpGammap}
\cE'_h(u) = \sum_{p \in \Omega_h} \sum_{e \in V_h(p)} \gamma_p^h(e) |u(p+he)-u(p)|^2.
\ee
\end{lemma}

\begin{proof}
For any triangle $T\in \cT_h$, and any $p,q \in h\Z^2$, we denote 
\begin{equation*}
s_T(p,q) := \int_T \<\nabla \vp_p^h(z),\, \Diff(z)\nabla \vp_q^h(z)\> \,dz.
\end{equation*}
Clearly $s_T(p,q) = 0$ if $q$ or $p$ is not a vertex of $T$. 
The coefficient $\gamma_p^h(e)$, $e\in \Z^2$, is thus given by the following sum with at most two non-zero terms:
\be
\label{gammaSumST}
\gamma_p^h(e) = -\frac 1 2 \sum_{T \in \cT_h} s_T(p,p+he).
\ee

Let $p,q,r \in h\Z^2$ be the vertices of a triangle $T \in \cT_h$. Since the sum $\vp_p^h+\vp_q^h+\vp_r^h$ is constant on $T$, equal to $1$, it has a null gradient on $T$, and therefore 
\begin{equation*}
s_T(p,p)+s_T(p,q)+s_T(p,r)=0.
\end{equation*}
Using this relation, and the two similar ones obtained by a cyclic permutation of $p,q,r$, we obtain
\begin{align*}
&\int_T \| \nabla (\interp_{\cT_h} u)(z)\|_{\Diff(z)}^2 dz \\
=&\, u(p)^2 s_T(p,p) + u(q)^2 s_T(q,q)+u(r)^2 s_T(r,r)\\
 & + 2 u(p) u(q) s_T(p,q) + 2 u(q) u(r) s_T(q,r) \\
 &+ 2 u(r) u(p) s_T(r,p),\\
=& - s_T(p,q) (u(p)-u(q))^2 - s_T(q,r) (u(q)-u(r))^2\\
& - s_T(r,p) (u(r)-u(p))^2.
\end{align*}
Summing this expression over all $T \in \cT_h$, and combining it with \iref{gammaSumST}, we obtain \iref{EpGammap}, which concludes the proof.
\QED\end{proof}

Finally, we provide an approximation of the coefficients $\gamma_p^h$ which will be easily compared with the AD-LBR weights $\gamma_p$ \eqref{def:GammaZ}.

\begin{lemma}
\label{lem:GammaHApprox}
Consider an edge $[p, p+h e]$ of $\cT_h$, shared by the two distinct triangles $T,T' \in \cT_h$. 
Let $h f, h g$ (resp.\ $h f', h g'$) be the two other vector edges of $T$ (resp.\ $T'$), oriented so that $e+f+g = 0$ (resp.\ $e+f'+g' = 0$). Then
\begin{equation*}
\left|\gamma_p^h(e) + \frac 1 4\left(\<f^\perp,\Diff(p) \, g^\perp\> + \<f'^\perp, \Diff(p) \, g'^\perp\>\right) \right| \leq \ve_h,
\end{equation*}
where $\ve_h := 2 \kappaD^2 \max \{ \|\Diff(x) -\Diff(y)\|;\, \|x-y\| \leq 2 \kappaD h\}$.
\end{lemma}

\begin{proof} 
We assume, up to exchanging $f$ and $g$, that $[p, p-h g]$ is an edge of $T$. Let $\alpha := \det(e,f) \in \{-1,1\}$, see point (ii) of Lemma \ref{lem:DescribeT}; note that $\alpha = \det(f,g) = \det(g,e)$.
Let $\gamma$ be the constant value of $\nabla \vp_p^h$ on $T$. Then $\<\gamma, h e\> = -1$ and $\<\gamma, h g\> = 1$.  These two independent linear identities are also satisfied by $\alpha f^\perp/h$, 
hence $\nabla \vp_p^h = \gamma = \alpha f^\perp/h$ on $T$.
%

Denoting $q := p+hg$, we obtain likewise $\nabla \vp_q^h = \alpha g^\perp/h$ on $T$.
Hence recalling that $|T| = h^2/2$: 
\begin{align*}
\int_T \< \nabla \vp_p^h, \Diff(p) \nabla \vp_q^h\> &= \frac {h^2} 2 \left\<\frac{\alpha f^\perp} h,\Diff(p)  \frac{\alpha g^\perp} h\right\>\\
&= \frac 1 2 \left\<f^\perp,\Diff(p)   g^\perp\right\>
\end{align*}
Therefore, using point (i) of Lemma \ref{lem:DescribeT} in the last step,
\begin{align*}
&\left|\int_T \< \nabla \vp_p^h, \Diff(z) \nabla \vp_q^h\> dz - \frac 1 2 \<f^\perp,\Diff(p) g^\perp\> \right| \\
&= \left|\int_T \< \nabla \vp_p^h,( \Diff(z)-\Diff(p)) \nabla \vp_q^h\> dz\right| \\
& \leq \frac {h^2} 2 \frac {2\kappaD} h  \frac {2\kappaD} h \max\{\|\Diff(z) - \Diff(p)\|;\, z \in T\} \leq \ve_h.
\end{align*}
Proceeding likewise on $T'$, and recalling \eqref{def:GammaH} (or \eqref{gammaSumST}), we conclude the proof. 
\QED
\end{proof}


\subsection{Some properties of $M$-reduced bases}

We establish some technical properties of $M$-reduced bases, thanks to which we will be able to compare in \S \ref{subsec:CompareStencils} the ``geometric'' construction of the ADT finite element stencils $V_h$, with the lattice based construction of the AD-LBR stencils $V$.

\begin{lemma}
\label{lem:AlmostAcuteNeigh}
Let $M \in S_2^+$, let $e_1,\cdots, e_n \in \Z^2$, $n > 2$, and let $\ve \in \{-1,1\}$. Assume that
for all $1 \leq i \leq n$, with the convention $e_{n+1} := e_1$:
\begin{align}
\label{curveCircles} 
\det(e_i, e_{i+1}) &= \ve,\\ 
\label{almostPositiveProducts}
\<e_i, M e_{i+1}\> &> -\frac 1 2 \min \left\{\|e_i\|_M^2, \|e_{i+1}\|_M^2 \right\}.
\end{align}
Then any $M$-reduced basis $(e,f)$ satisfies 
\begin{equation*}
\{e,f\} \subset \{e_1, \cdots, e_n\}. 
\end{equation*}
\end{lemma}

\begin{proof}
Let $z\in \Z^2\sm \{e_1, \cdots, e_n\}$. Our objective is to show that $z$ cannot be an element of an $M$-reduced basis, and we may therefore assume that $z$ has co-prime coordinates.

It follows from \eqref{curveCircles} that the closed polygonal line of consecutive vertices $e_1, \cdots, e_n$, circles at least once around the origin, see Figure \ref{fig:RBSec}. Hence $z = \alpha e_i+ \beta e_{i+1}$, for some $1 \leq i \leq n$ and some $\alpha, \beta \geq 0$. Since $(e_i, e_{i+1})$ is a basis of $\Z^2$ (indeed $|\det(e_i,e_{i+1})| = 1$), the coefficients $\alpha$ and $\beta$ are integers. Since $z\notin \{e_i,e_{i+1}\}$, $\alpha+\beta \geq 2$. Since $z$ has co-prime coordinates, $\alpha\beta \neq 0$.

Assuming without loss of generality that $\|e_i\|_M \geq \|e_{i+1}\|_M$, we obtain using \eqref{almostPositiveProducts}:
\begin{align*}
&\|z\|_M^2 = \alpha^2 \|e_i\|_M^2 + \beta^2 \|e_{i+1}\|_M^2 +2\alpha \beta \<e_i , M e_{i+1}\>\\
&> \alpha^2 \|e_i\|_M^2 + \beta^2 \|e_{i+1}\|_M^2 -\alpha \beta \min \{\|e_i\|_M^2, \|e_{i+1}\|_M^2\}\\
&\geq \|e_i\|^2_M + (\alpha^2+\beta^2-1-\alpha\beta) \|e_{i+1}\|_M^2.
\end{align*}
Observing that $\alpha^2+\beta^2-1-\alpha\beta \geq 0$ for all $\alpha, \beta \in [1, \infty[$, we obtain $\|z\|_M > \|e_i\|_M$. 
Since $e_i$ and $e_{i+1}$ are linearly independent, we have $\|e_i\|_M \geq \lambda_2(M)$.
Finally $\|z\|_M > \lambda_2(M)$, hence $z$ cannot be an element of an $M$-reduced basis, which concludes the proof. 
\QED
\end{proof}

The next corollary reverses the construction, presented in Corollary \ref{corol:ObtuseSuperBaseFromReducedBasis2}, of an $M$-obtuse superbase from an $M$-reduced basis.

\begin{corollary}
\label{corol:ReducedBasisInObtuseSuperBase}
Let $M \in S_2^+$ and let $(e,f,g)$ be an $M$-obtuse superbase of $\Z^2$, ordered so that $\|e\|_M \leq \|f\|_M \leq \|g\|_M$. Then $(e,f)$ is an $M$-reduced basis. 
\end{corollary}

\begin{proof}
The family $(e,-g,f,-e,g,-f)$ satisfies by construction the conditions of the previous lemma. 
Hence any $\Met(z)$-reduced basis $(e',f')$ of $\Z^2$ satisfies $\{e',f'\}\subset \{e,f,g,-e,-f,-g\}$. Observing that $e'$ and $f'$ are linearly independent, that $\|e'\|_M \leq \|f'\|_M$, and that $\|e\|_M \leq \|f\|_M \leq \|g\|_M$, we obtain that $\|e\|_M \leq \|e'\|_M$ and $\|f\|_M \leq \|f'\|_M$. 
Recalling that $M$-reduced bases are defined by the minimality of their $\|\cdot\|_M$-norms, see Definition \ref{def:ReducedBasis}, we obtain as announced that $(e,f)$ is an $M$-reduced basis.
\QED
\end{proof}

The previous lemma shows that for any $z \in \Omega$, there exists an $\Met(z)$-reduced basis $(e,f)$ such that 
\begin{equation}
\label{VReducedBasis}
V(z) = \{e,f,e+f,-e,-f,-e-f\}.
\end{equation}

\begin{figure}
\centering
\includegraphics[width=3cm]{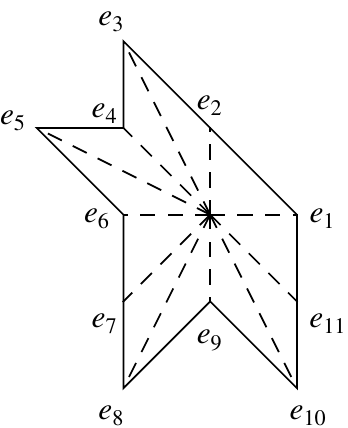}
\hspace{0.8cm}
\includegraphics[width=1.5cm]{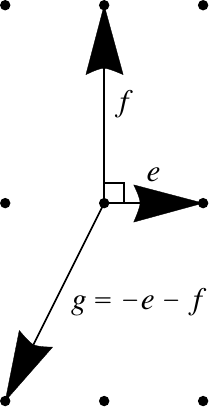}
\hspace{0.4cm}
\includegraphics[width=1.5cm]{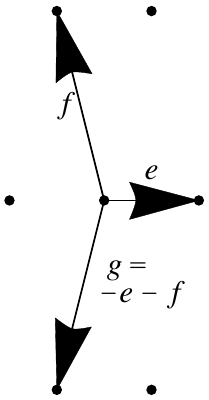}
\caption{(left) A family $e_1, \cdots, e_{11}$ satisfying condition \eqref{curveCircles} of Lemma \ref{lem:AlmostAcuteNeigh}: the closed polygonal line of vertices $(e_1, \cdots, e_{11})$ circles (at least) once around the origin, and the triangles $\widehat {(0,e_i,e_{i+1})}$ have area $1/2$.
(Center and right) The lattice $\Z^2$, and an $M$-reduced basis $(e,f)$, shown after a linear change of coordinates by $A$, such that $A^\trans A = M\in S_2^+$. Case $\mu(M) = 0$ (center), and case $2\mu(M) = \lambda_1(M)^2$ (right).
}
\label{fig:RBSec}
\end{figure}

Given $M \in S_2^+$, and an $M$-reduced basis $(e,f)$ of $\Z^2$, we denote $\orth(M) := |\<e,M f\>|$. This value can be expressed in terms of the Minkowski minima \eqref{def:MinkowskiMinima} and thus does not depend on the particular choice of $M$-reduced basis.
Indeed, recalling the identity
\begin{equation*}
\<e,M f\>^2+\det (M) \det (e,f)^2 = \|e\|_M^2 \|f\|_M^2,
\end{equation*}
we obtain 
\begin{equation}
\label{def:orth}
\mu(M) = |\<e,M f\>| = \sqrt{\lambda_1(M)^2 \lambda_2(M)^2 - \det(M)}.
\end{equation}
In addition one has 
\begin{equation}
\label{ineq:MuLambda}
0 \leq 2 \mu(M) \leq \lambda_1(M)^2,
\end{equation}
where the right hand side follows from Lemma \ref{lem:IneqScal}.
A vanishing value, $\mu(M) = 0$, indicates that the lattice $\Z^2$ admits an $M$-orthogonal basis. 
In contrast, when the upper bound is met, $2 \mu(M) = \lambda_1(M)^2$, one has $\|f\|_M = \|f+\ve e\|_M$ for $\ve := -{\rm sign}\<e,M f\>$, hence the reduced basis $(e,f)$ is not unique even up to sign changes. See Figure \ref{fig:RBSec}.

We next show that the stencils of the AD-LBR do not depend on the choices of reduced bases, as was announced in the introduction.

\begin{lemma}
\label{lem:indepBasis}
The weights $\gamma_z : \Z^2 \to \R_+$ used in the AD-LBR at a point $z \in \Omega$ (defined on $V(z)$ by \eqref{def:GammaZ} and extended to $\Z^2$ by $0$), do not depend on the choice of $\Met(z)$-obtuse superbase of $\Z^2$.
\end{lemma}

\begin{proof}
We denote $M := \Met(z)$ and $D := \Diff(z)$. Let $(e,f,g)$ and $(e',f',g')$ be two $M$-obtuse superbases, and let $V,V'$ and $\gamma, \gamma' : \Z^2 \to \R_+$ be the corresponding AD-LBR stencils and weights defined by \eqref{def:Vz} and \eqref{def:GammaZ}. 
We may assume, using Corollary \ref{corol:ReducedBasisInObtuseSuperBase} and up to reordering, that $(e,f)$ and $(e', f')$ are $M$-reduced bases.

Corollary \ref{corol:ObtuseSuperBaseFromReducedBasis2} states that the scalar products $\<e, M g\>$, $\<f, M g\>$, $\<e', M g'\>$ and $\<f', M g'\>$ are (strictly) negative. On the other hand 
\begin{equation}
\<e,M f\> = \<e', M f'\> = - \mu(M) \leq 0.
\end{equation}


Applying Lemma \ref{lem:AlmostAcuteNeigh} to the family 
\begin{equation*}
(e',-g',f',-e',g',-f')
\end{equation*}
we obtain that
\be
\label{basisInBasis}
\{e,f\} \subset \{e',f',g',  -e', -f', -g'\}. 
\ee
If $\mu(M)\neq 0$, then $\<e,M f\>$ and $\<e',M f'\>$ are negative, and not merely non-positive, thus $\{e,f\} \subset \{e',f',g'\}$, or $\{e,f\} \subset \{-e',-f',-g'\}$. Since $e+f+g = 0 = e'+f'+g'$, it follows that $\{e,f,g\} = \{e',f',g'\}$, or $\{e,f,g\} = \{-e',-f',-g'\}$. The stencils $V,V'$ are thus identical, see \eqref{def:Vz}, and so are the weights $\gamma, \gamma'$.

If $\mu(M) = 0$, then the stencils $V,V'$ may not be identical. Observe however that 
$\<e^\perp, D f^\perp\> = 0 = \<e'^\perp, D f'^\perp\>$, using \eqref{eq:DM}. Hence using the weights expression \eqref{def:GammaZ}: 
\begin{align}
\label{gammaOrth}
\gamma(\pm g) &= - \<e^\perp, D f^\perp\>/2= 0,\\
\nonumber
\gamma(\pm e) &= \|f^\perp\|_D^2/2,&
\gamma(\pm f)  &= \|e^\perp\|_D^2/ 2 ,
\end{align} 
and likewise for $\gamma', e', f', g'$.
Note also that $\|g'\|^2_M = \|e'\|_M^2 +\|f'\|_M^2 > \lambda_2(M)^2$, hence $e$ and $f$ are different from $g'$ and $-g'$. It follows from \eqref{basisInBasis} that $\{e,f\}=\{\ve_1 e', \ve_2 f'\}$ for some $\ve_1, \ve_2 \in \{-1,1\}$. This implies $\gamma=\gamma'$ in view of \iref{gammaOrth}, and concludes the proof. 
\QED\end{proof}

The next lemma establishes weak uniqueness and stability properties for $M$-reduced bases, 
in the case of a strict inequality $2 \mu(M) < \lambda_1(M)^2.$ 

\begin{lemma}
\label{lem:uniqueness}
Consider $M,M'\in S_2^+$, an $M$-reduced basis $(e,f)$, and an $M'$-reduced basis $(e',f')$. 
Let $\tau \geq 1$ be such that $\tau^{-2} M \leq M' \leq \tau^2 M$, in the sense of symmetric matrices.
Assume either:
\begin{enumerate}[(i)]
\item $2\mu(M) < \lambda_1(M)^2$, and $\tau=1$ (i.e. $M' = M$).
\item $4\mu(M) \leq \lambda_1(M)^2$, and $\tau^4\leq  1+\frac 1 3 {\kappa(M)^{-2}}$. 
\end{enumerate}
Then 
$
\{e',f'\}\subset \{e,f,-e,-f\}.
$
\end{lemma}

\begin{proof}
Denoting $\alpha := 2\mu(M)/\lambda_1(M)^2$, we obtain:
\begin{align*}
&4\<e,M' f\> = \|e+f\|_{M'}^2 - \|e-f\|_{M'}^2\\
 &\leq \tau^2 \|e+f\|_M^2 - \tau^{-2}\|e-f\|_M^2\\
 &= (\tau^2 -\tau^{-2}) (\|e\|_M^2+\|f\|_M^2)
+ 2 (\tau^2+\tau^{-2}) \<e,M f\> \\
 &\leq ( (\tau^2-\tau^{-2})(1+\kappa(M)^2) + \alpha (\tau^2+ \tau^{-2}) ) \|e\|_M^2 \\ 
 &\leq ( (\tau^4-1) (1+\kappa(M)^2) +\alpha (\tau^4+1) ) \|e\|_{M'}^2. 
\end{align*}
In the fourth line we used Lemma \ref{lem:MinkowskiBounds}, which implies that $\|f\|_M = \lambda_2(M) \leq \kappa(M) \lambda_1(M) = \kappa(M) \|e\|_M$, and Lemma \ref{lem:IneqScal} to bound $2 \<e, M f\>$.
Replacing $\alpha$ and $\tau$ with their assumed upper bounds, we obtain $2\<e,M' f\> < \|e\|_{M'}^2$. Proceeding likewise, we obtain $2 |\<e, M' f\>| <  \min \{\|e\|_{M'}^2, \|f\|_{M'}^2\}$.
We may therefore apply Lemma \ref{lem:AlmostAcuteNeigh} to $M'$ and $(e,f,-e,-f)$, which implies  $\{e',f'\} \subset \{e,f,-e,-f\}$ as announced. 
\QED\end{proof}

\subsection{
Comparison of the stencils
}
\label{subsec:CompareStencils}

We assume in this subsection that the scale parameter $h$ is sufficiently small. Our assumption is stronger than the one used in \S \ref{subsec:ExistsADT}, see  \eqref{assumHADT}, hence in particular there exists an Anisotropic Delaunay Triangulation $\cT_h$. More precisely we assume that 
\be
\label{assum:h}
\tau_h \leq \sqrt[4]{1+1/(3 \kappaD^2)} \stext{ and } \theta_h \leq  \theta_0 := 1/(4 \kappaD). 
\ee
See \eqref{def:KappaD}, \eqref{def:TauH}, and Lemma \ref{lem:DescribeT} for the definition of $\kappaD$, $\tau_h$ and $\theta_h$ respectively.
For Lipschitz metrics, $\tau_h = 1+\cO(h)$ and $\theta_h = \cO(h)$.

Our objective is to compare the stencils $V(p)$, $V_h(p)$, of the AD-LBR \eqref{def:Vz} and of the ADT finite element discretization \eqref{def:Vhp} respectively, at a point $p \in h \Z^2$. The next lemma shows that they are equal \emph{unless}  the lattice $\Z^2$ is almost orthogonal with respect to the local metric; a property quantified via $\mu(\Met(p))$, see \eqref{def:orth}.

\begin{lemma}
\label{lem:EqualUnlessAlmostOrthogonal}
Let $p \in h \Z^2$, and let $M := \Met(p)$. 
If $\mu(M) > \theta_h$, then $V_h(p) = V(p)$.
In any case, one has for any $M$-reduced basis $(e,f)$:
\begin{align}
\label{VhSupset}
V_h(p) \supset& \{e,f,-e,-f\} \\
\label{VhSubset}
V_h(p) \subset& \{e,f,e+f,e-f, \ -e,-f,-e-f,f-e\} 
\end{align}
\end{lemma}

\begin{proof}
We assume that $\<e,M f\> \leq 0$, up to replacing $f$ with $-f$. 
Let $T \in \cT_h$ be a triangle containing $p$, and let $h e_1, h e_2, h e_3$ be the edges of $T$,  oriented so that $e_1+e_2+e_3 = 0$.
Using point (iii) of Lemma \ref{lem:DescribeT}, and \eqref{KMK}, we obtain for all $1 \leq i \leq 3$, with the convention $e_4 := e_1$
\begin{equation}
\label{eq:eMett}
\<e_i, M e_{i+1} \> \leq \theta_h \leq \theta_0 < \frac 1 {2 \kappaD}\leq \frac 1 2 \min\{\|e_i\|_M^2, \|e_{i+1}\|_M^2\}.
\end{equation}
Denote $E := \{e_1,e_2,e_3\}$, and $-E := \{-e_1,-e_2,-e_3\}$.
Applying Lemma \ref{lem:AlmostAcuteNeigh} to $M$ and the points $(e_1,-e_3,e_2,\linebreak-e_1,e_3,-e_2)$, we obtain that $\{e,f\} \subset E \cup (-E)$. 
Up to exchanging $E$ with $-E$, we thus have $\{e,f\} \subset E$ or $\{e,-f\}\subset E$.
Since the elements of $E$ sum to zero, we conclude that 
\begin{equation}
\label{ECases}
E=\{e,f,-e-f\} \  \text{ or } \ E = \{e,-f,-e+f\},
\end{equation}
which implies \eqref{VhSubset}.

If $\mu(M)=|\<e,M f\>| > \theta_h$, then \eqref{eq:eMett} forbids the second case in \eqref{ECases}. Thus $E = \{e,f,-e-f\}$, and therefore $V_h(p) \subset V(p)$, using \eqref{VReducedBasis}. 

Let $T\in \cT_h$ be a triangle containing $p$ and intersecting the half line $L := \{p+r e;\, r>0\}$. We know \eqref{ECases} that $h e$ is a vector edge of $T$ (i.e.\ the difference between two vertices of $T$).  The corresponding edge segment must be $[p,p+he]$, since otherwise $T \cap L$ would be empty. Thus $e\in V_h(p)$. 
Applying the same argument to $-e,f,-f$, we obtain \eqref{VhSupset}.

If $\mu(M) > \theta_h$, then $h(e+f)$ is also a vector edge of any triangle $T \in \cT_h$ containing $p$, since we eliminated the second case in \eqref{ECases}. Reasoning as above we find that $\{e+f,-e-f\} \subset V_h(p)$, and therefore $V(p) \subset V_h(p)$. Thus $V(p) = V_h(p)$. This concludes the proof.
\QED
\end{proof}

We introduce new stencils $W(p),W'(p)$, for $p\in \R^2$, defined as follows. Let $M := \Met(p)$.
If $\mu(M) \leq \theta_0$, then denoting by $(e,f)$ an $M$-reduced basis,
\begin{align}
\label{def:W1}
W(p)  &:= \{e,f,-e,-f\}, \\ 
\label{def:W2}
W'(p) &:= \{e,f,e+f,e-f,\ -e,-f,-e-f,f-e\}.
\end{align}
On the other hand, if $\mu(M) > \theta_0$, then
\begin{equation}
\label{def:W3}
W(p) := V(p) =: W'(p).
\end{equation}
The previous lemma implies that $W(p) \subset V_h(p) \subset W'(p)$ for any $p \in h \Z^2$. 

\begin{figure}
\centering
\includegraphics[width=8.5cm]{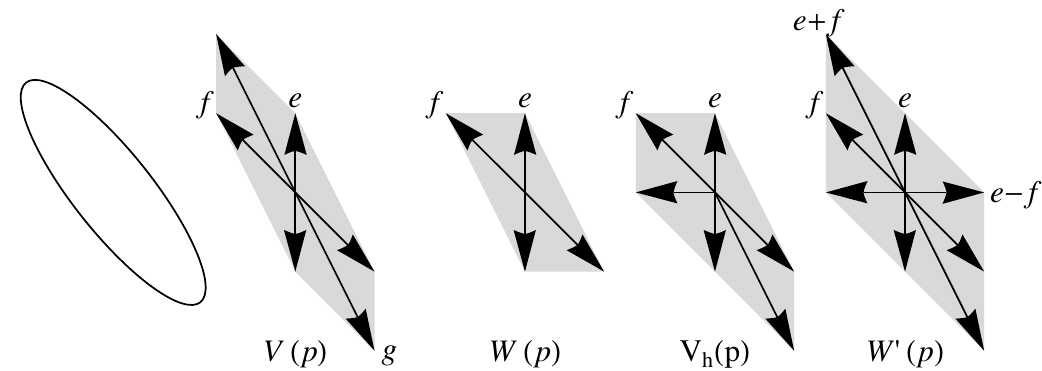} 
\caption{Consider a point $p\in h \Z^2$, and denote $M := \Met(p)$. From left to right: 
ellipse $\{\|z\|_M \leq 1\}$, AD-LBR stencil $V(p)$, stencils $W(p) \subset V_h(p) \subset W'(p)$. For $W(p)$ and $W'(p)$ we assumed that $\mu(M) < \theta_0$, otherwise they are equal to $V(p)$. 
Note that $V(p)$, $W(p)$, $W'(p)$ only depend on $M$, while $V_h(p)$ depends on the structure of the triangulation $\cT_h$.
}
\end{figure}

\begin{lemma}
The stencils $W(p)$, $W'(p)$, do not depend on the choice of $\Met(p)$-reduced basis. 
\end{lemma}

\begin{proof} Let $M := \Met(p)$. If $\mu(M) > \theta_0$, then $W(p)$, $W'(p)$ are defined by \eqref{def:W3}, hence there is nothing to prove.
Otherwise we obtain $\mu(M) \leq \theta_0 \leq 1/(4 \kappaD) \leq \lambda_1(M)^2/4$. Hence, by Lemma \ref{lem:uniqueness}, any two $M$-reduced bases $(e,f)$, $(e',f')$, need to satisfy $\{e',f'\} \subset \{e,f, -e,-f\}$. In view of \eqref{def:W1} and \eqref{def:W2}, they thus yield the same stencils $W(p)$, $W'(p)$.
\QED
\end{proof}

Let $\cF_h, \cF'_h$ be the energies associated to the stencils $W,W'$: for $u\in L^2(\Omega_h)$, extended to $h\Z^2$ by periodicity,
\begin{align*}
\cF_h(u) &:= \sum_{z \in \Omega_h} \sum_{g \in W(z)} |u(z+h g)-u(z)|^2, \\
\cF'_h(u) &:= \sum_{z \in \Omega_h} \sum_{g \in W'(z)} |u(z+h g)-u(z)|^2.
\end{align*}

The outline of the proof of Theorem \ref{th:FE} is as follows. We prove in Lemmas \ref{lem:EEpFp}, \ref{lem:FpF} and \ref{lem:FE} respectively that for any $u \in L^2 (\Omega_h)$:
\begin{align}
\label{ineq1}
|\cE'_h(u) - \cE_h(u) | &\leq (\ve_h + C_0 \theta_h) \cF'_h(u)\\
\cF'_h(u) &\leq C_1 \cF_h(u)\\
\label{ineq3}
\cF_h(u) &\leq C_2 \cE_h(u),
\end{align}
where the constants $C_0,C_1,C_2$ only depend on the metric $\Met$. Combining these inequalities, and recalling that $\theta_h = \cO(h)$ and $\ve_h = \cO(h)$ for Lipschitz metrics ($\ve_h$ is defined in Lemma \ref{lem:GammaHApprox}), we obtain
\begin{equation*}
|\cE'_h(u) - \cE_h(u)| \leq c h \cE_h(u),
\end{equation*}
for some constant $c = c(\Met)$. This establishes \eqref{EnEpn}, and concludes the proof of Theorem \ref{th:FE}.\\

For each $p \in \R^2$, we denote by $\eta_p$, $\eta'_p : \Z^2 \to \{0,1\}$, the characteristic functions of $W(p)$ and $W'(p)$ respectively. The proofs of \eqref{ineq1} and \eqref{ineq3} immediately result from the comparison, in Lemmas \ref{lem:EEpFp} and \ref{lem:FE} respectively, of 
the coefficients $\gamma_p$, $\gamma^h_p$, $\eta_p$, $\eta'_p$ appearing in the expressions of $\cE_h, \cE'_h, \cF_h, \cF'_h$.

In the following, it will be convenient to express the AD-LBR weights, and others, in terms of the scalar product associated to the Riemannian metric. 
We thus recall \eqref{eq:DM}: for any $z \in \Omega$, and any $e,f \in \R^2$, 
\begin{equation*}
\<e^\perp, \Diff(z) f^\perp\> = \DDiff(z) \<e, \Met(z) f\>.
\end{equation*}
We also define the bounds ($0 < \DDMin \leq \DDMax < \infty$)
\begin{equation*}
\DDMin := \min_{z\in \Omega} \DDiff(z), \quad 
\DDMax := \max_{z\in\Omega} \DDiff(z).
\end{equation*}


\begin{lemma}
\label{lem:FE}
For any $p\in \R^2$, one has on $\Z^2$
\begin{equation*}
\eta_p \leq C_2 \gamma_p, \quad \text{ with } C_2 :=  2\DDMax/ \theta_0.
\end{equation*}
\end{lemma}

\begin{proof}
Let $M := \Met(p)$, and let $(e,f,g)$ be an $M$-obtuse superbase of $\Z^2$. We can assume, thanks to Corollary \ref{corol:ReducedBasisInObtuseSuperBase}, that $(e,f)$ is an $M$-reduced basis.
Then using \eqref{NegativeScalProd}
\begin{align*}
2 \DDiff(p) \gamma_p(\pm f) &\geq \frac 1 2 \|e\|_M^2 \geq \frac 1 {2 \kappaD} = 2\theta_0,
\end{align*}
hence $\gamma_p(\pm f) \geq \theta_0/\DDMax$, and likewise $\gamma_p(\pm e) \geq \theta_0/\DDMax$.
If $\mu(M) \leq \theta_0 $, then $W(p) = \{e,f,-e,-f\}$, and this concludes the proof.

Assume now that $\mu(M) > \theta_0$. 
Then
\begin{equation*}
2 \DDiff(p) \gamma_p(\pm g) = -\<e,M f\> = \mu(M) \geq \theta_0,
\end{equation*} 
hence $\gamma_p(\pm g) \geq \theta_0/(2 \DDMax)$. The result follows since $W(p) = \{e,f,g,-e,-f,-g\}$.
\QED
\end{proof}

Let $p\in h \Z^2$ and let $e_1,\cdots, e_k$ be the consecutive elements of $V_h(p)$, in trigonometric order. We define for all $1 \leq i \leq k$, denoting $M:=\Met(p)$,
\begin{equation*}
\tilde \gamma_p^h(e_i) := - \frac {\DDiff(p)} 4 (\<e_i-e_{i-1},\, M e_{i-1}\> + \<e_i-e_{i+1},\, M e_{i+1}\>),
\end{equation*}
with the periodic conventions $e_{k+1} := e_1$, $e_0 := e_k$. We also set $\tilde \gamma_p^h = 0$ on $\Z^2 \sm \{e_1, \cdots, e_k\}$.

\begin{lemma}
\label{lem:EEpFp}
For any $p \in h \Z^2$, one has on $\Z^2$
\begin{equation}
\label{GammaHGamma}
| \gamma^h_p - \tilde \gamma^h_p| \leq \ve_h \eta'_p, \quad \text{ and } \quad |\tilde \gamma^h_p - \gamma_p| \leq C_0 \theta_h \eta'_p,
\end{equation}
where $\ve_h$ is given in Lemma \ref{lem:GammaHApprox}, and $C_0 = 1/\DDMin$. 
\end{lemma}

\begin{proof}
The coefficients $\gamma_p$, $\gamma^h_p$, $\tilde\gamma^h_p$, are all equal to zero outside of $W'(p)$. This holds by construction of $\gamma_p$, and by Lemma \ref{lem:EqualUnlessAlmostOrthogonal} for $\gamma^h_p$, $\tilde\gamma^h_p$. We may therefore forget about the presence of $\eta'_p$ in \eqref{GammaHGamma}.

First inequality. 
Lemma \ref{lem:GammaHApprox} states that $| \gamma^h_p - \tilde \gamma^h_p| \leq \ve_h$ on $\Z^2$, 
which concludes the proof.

Second inequality.
If $\mu(M) > \theta_h$, then $V_h(p) = V(p)$. Comparing the definition of $\tilde \gamma^h_p$ with that of $\gamma_p$ \eqref{def:GammaZ} we observe that $\tilde\gamma^h_p = \gamma_p$ on $\Z^2$, which concludes the proof in this case. 

Assume now that $\mu(M) \leq \theta_h$. Let $(e,f,g)$ be an $M$-obtuse superbase of $\Z^2$. We can assume, thanks to Corollary \ref{corol:ReducedBasisInObtuseSuperBase} that $(e,f)$ is an $M$-reduced basis. Looking at \eqref{def:GammaZ} and denoting $\delta := 2 \DDiff(p)$,
we find that 
\begin{equation*}
|\delta \gamma_p(\pm e) - \|f\|_M^2| = | \<e,M f\>| = \mu(M) \leq \theta_h. 
\end{equation*}
Likewise  $|\delta \gamma_p(\pm f) - \|e\|_M^2| \leq \theta_h$.
In addition
\begin{equation*}
\delta \gamma_p(\pm (e+f)) = \mu(M) \leq \theta_h, \text{ and } \gamma_p(\pm (e-f)) = 0.
\end{equation*}

Combining the definition of $\tilde \gamma^h_p$ with the description of the stencil $V_h(p)$ in Lemma \ref{lem:EqualUnlessAlmostOrthogonal},
we obtain that
\begin{equation*}
2 \delta \,\tilde \gamma^h_p(e) =
\left\{ 
\begin{array}{c}
\<f-e, M f\>\\
\text{or}\\
\<f+e, M f\>
\end{array}
\right\}
+
\left\{ 
\begin{array}{c}
\<f-e, M f\>\\
\text{or}\\
\<f+e, M f\>
\end{array}
\right\}.
\end{equation*}
In any case $|\delta \, \tilde \gamma^h_p(e) - \|f\|_M^2| \leq \theta_h$. 
The expressions and estimates of $\tilde \gamma^h_p$ at the points $-e,f,-f$ are obtained similarly.
Likewise, using Lemma \ref{lem:EqualUnlessAlmostOrthogonal},
\begin{equation*}
2 \delta \,\tilde \gamma^h_p(e+f) =
\left\{
\begin{array}{ll}
\<e, M f\> + \<e, M f\> &\text{ if } e+f \in V_h(p),\\
0 & \text{ otherwise.}
\end{array}
\right.
\end{equation*}
In any case $| \delta \,\tilde \gamma^h_p(e+f)| \leq \theta_h$. The expressions and estimates of $\tilde \gamma_p^h$ at the points $-(e+f), e-f, -(e-f)$ are similar. 
Comparing the above estimates of $\gamma_p$, $\tilde \gamma^h_p$, 
we obtain that $\delta |\gamma_p - \tilde \gamma^h_p| \leq 2 \theta_h$
on $\{e,f,e+f,e-f, \linebreak -e,-f,-e-f,f-e\}=W'(p)$. 
Since $\delta = 2 \DDiff(p) \geq 2 \DDMin = 2/ C_0$, this concludes the proof.
\QED
\end{proof}


%
%

In the last lemma of this section, we control the contribution to the energy $\cF'_h$ of a stencil $W'(p)$, $p\in h \Z^2$, in terms of the contributions to $\cF_h$ of $W(p)$ and of the neighboring stencils $W(p+he)$, $e \in W(p)$. This leads to an estimate of $\cF'_h$ in terms of $\cF_h$, which concludes the proof of Theorem \ref{th:FE}.
\begin{lemma}
\label{lem:FpF}
One has $\cF'_h(u) \leq C_1 \cF_h(u)$, for any $u \in L^2(\Omega_h)$, with $C_1 := 17$.
\end{lemma}

\begin{proof}
Consider a grid point $p\in h \Z^2$, and denote $M := \Met(p)$. Assume first that $\mu(M) \leq \theta_0$, so that $W(p) \subsetneq W'(p)$.
Consider also an arbitrary $g\in W'(p)\sm W(p)$, and observe that $g= e+f$ for some $M$-reduced basis $(e,f)$. 

We set $p' := p+e$ and $M' := \Met(p')$. Applying point (ii) of Lemma \ref{lem:uniqueness}, we find that $(e,f)$ is also an $M'$-reduced basis. Indeed we have as required
\begin{equation*}
4 \mu(M) \leq 4\theta_0 = \kappaD^{-1} \leq \lambda_1(M)^2,
\end{equation*}
using \eqref{KMK},
and the assumption on $\tau$ follows from \eqref{assum:h} and \eqref{def:tauM}.
Therefore
\begin{equation}
\label{fWpp}
f \in W(p'), \text{ and } h^{-1} (p'-p) =e \in W(p').
\end{equation}
 We obtain 
 \begin{align}
 \label{ineq:upg}
&|u(p+g)-u(p)|^2 \\
\nonumber
&= |u(p+e+f) -u(p)|^2\\
\nonumber
&\leq 2 ( |u(p+e+f)-u(p+e)|^2 +|u(p+e)-u(p)|^2)\\
\nonumber
&= 2 (|u(p'+f)-u(p')|^2 + |u(p+e)-u(p)|^2).
 \end{align}
 
Denote, for all $q \in h\Z^2$,
\begin{align*}
\cF_h(u; q) &:=\sum_{g \in W(q)} |u(q+h g)-u(q)|^2, \\ 
\cF'_h(u; q) &:= \sum_{g \in W'(q)} |u(q+h g)-u(q)|^2.
\end{align*}
Using \eqref{ineq:upg}, we obtain
\begin{equation}
\label{FFpG}
\cF'_h(u; p) - \cF_h(u;  p) 
\leq \cG_h(u;  p)
\end{equation}
where $\cG_h(u;  p)$ is given by 
\begin{equation*}
\left\{
\begin{array}{ll}
\displaystyle
4 \cF_h(u;  p)+ 2 \sum_{g\in W(p)} \cF_h(u;  p+g), & \text{ if } \mu(\Met(p)) \leq \theta_0\\
0,  & \text{ if }  \mu(\Met(p)) > \theta_0
\end{array}
\right.
\end{equation*}
When $\cF_h(u;p')$ appears in $\cG_h(u; p)$, with $p,p' \in h \Z^2$, $p\neq p'$, we have $h^{-1}(p'-p) \in W(p')$, see \eqref{fWpp}. For each $p'\in h \Z^2$, there are thus at most $\#(W(p')) \leq 6$ points $p \in h\Z^2\sm \{p'\}$ such that $\cF_h(u;p')$ appears in $\cG_h(u; p)$. 
Summing \iref{FFpG} over $p\in \Omega_h$, we thus obtain $\cF'_h(u) -\cF_h(u) \leq (4+2\times 6) \cF_h(u)$ (the constant could easily be improved), which concludes the proof. 
\QED\end{proof}

\section{Numerical experiments}
\label{sec:num}


We compare our scheme AD-LBR with a family of other schemes: finite difference, finite elements, and two sche\-mes from the image processing literature. We begin with a quantitative comparison for the discretization of the restoration equation, in a synthetic case where the exact solution is analytically available for reference. The second test case is a qualitative comparison of Coherence-Enhancing Diffusion (CED) \cite{W98}, on a real image and the quality assessment is by visual inspection. Finally we present a 3D implementation of AD-LBR for proof of feasibility, featuring a synthetic CED experiment, and a application of Edge Enhancing Diffusion to MRI data.

\subsection{The different schemes}
\label{sec:Schemes}
Our two dimensional numerical experiments feature the following six numerical schemes for anisotropic diffusion.

\begin{enumerate}
\item[\ding{227}] AD-LBR: the scheme presented in this work.

\item[\ding{227}] 
Finite Differences (FD). The gradient and the divergence are discretized using standard centered finite differences \cite{Mitchell80}, see Remark \ref{rem:FD} for details. This approach, arguably the most straightforward, leads to a 9 point stencil.

\item[\ding{227}] Bilinear Finite Elements (Q1). Bilinear finite elements, also referred to as Q1 finite elements, are linear with respect to each space direction. This amounts to use a 9 points stencil, where the coefficients are different from the previous scheme.

\item[\ding{227}] Weickert-Scharr scheme (WS). 
This scheme, introduced in \cite{WS02}, is based on a second order approximation of the gradient using a $3 \times 3$ centered stencil. As a result, it offers good accuracy and rotation invariance when applied to sufficiently smooth functions, but lacks robustness guarantees such as the discrete maximum principle and spectral correctness (see \S \ref{sec:intro}), even for $\Diff = \Id$. The stencil for this scheme has size $5\times5$.

\item[\ding{227}] Weickert's Non-Negative scheme (W-NN). The coefficients of this scheme, detailed in \cite{W98} page 95, are non-negative as long as the anisotropy ratio \eqref{def:Kappa} satisfies 
$\kappa\le1+\sqrt{2}\sim 2.41$.

\item[\ding{227}] Axes-directed Non-Negative scheme (\WG). This six point non-negative scheme is implicitly defined in the proof Theorem 6 in \cite{W98}, and can be regarded as a generalisation of W-NN. See Remark \ref{rem:WG} below for details. Among the 6 points of the stencil, 4 points are along the axes of coordinates. 
\end{enumerate}

Note that other schemes exist, see for instance \cite{MN01,WS08}. While an exhaustive comparison is in principle desirable, it could not be done here due to time and space constraints.

To fix the ideas and illustrate the difference between the schemes, we propose to compute the stencil and the coefficients for different \emph{constant} diffusion tensors $\Diff$, in isotropic and anisotropic cases. 
Denoting by $R$ the matrix of rotation by the angle $\theta=\pi/6$, and by $\kappa\geq 1$ the chosen anisotropy ratio, we set, identically on $\R^2$:
\begin{equation}
\label{def:DTest}
\Diff := R 
\left(
\begin{array}{cc}
1 & 0\\
0 & \kappa^{-2} 
\end{array}
\right)
R^\trans.
\end{equation}
The results are presented in Tables \ref{tab:coeff1} and \ref{tab:coeff2}.
Note that for the two last cases (anisotropy $\kappa=\sqrt{10}$ and $\kappa=\sqrt{50}$) the AD-LBR stencil contains points that are outside the $3\times 3$ neighborhood of the pixel. However the stencil contains 6 points, as expected. This contrasts with the schemes FD, Q1, W-NN where only the $3\times 3$  neighborhood is involved.   Another observation is that the off-center stencil coefficients of the AD-LBR are non-positive (this gives non-negative off-diagonal coefficients for $\diver(\DD\nabla)$), in contrast with schemes FD, Q1, WS, and with scheme W-NN for anisotropy $\kappa > 1+\sqrt 2$. This is an essential property of AD-LBR (and \WG), and as a consequence our scheme satisfies, unconditionally, the discrete maximum  principle \cite{Al93,DA07}. 

The largest eigenvalue of the discrete operator\linebreak $-\diver(\DD \nabla)$ is given in Table \ref{tab:ev1}, for the different schemes. It turns out that AD-LBR has in most cases the smallest eigenvalues among all schemes, except for scheme WS and occasionally \WG. This property allows (although this was not done in our numerical experiments) to use larger time steps for AD-LBR than for the other schemes, when solving parabolic equations \iref{pde} or \eqref{eq:parabolique} with an explicit time discretization.

Operator splitting is a classical approach to further increase the timestep in (potentially anisotropic and non-linear) diffusion PDEs \cite{W98,WRV98,BSIK03}. The AD-LBR is compatible with Additive Operator Splitting, by applying Remark (e) page 111 in \cite{W98}, although the efficiency of this technique is here compromised by the potentially large number of directions in our adaptive stencils. Let us also mention Multiplicative Operator Splittings, and Additive-Multiplicative Operator Splittings, which allow to combine different time-steps \cite{BSIK03,SJA10}. None of these methods was used in our experiments.

\begin{remark}[Axes-directed non negative six point scheme]
\label{rem:WG}
The following six point scheme \WG~is, in our belief, the best possible implementation of the constructive proof in \cite{W98} of the existence of non-negative schemes for two dimensional anisotropic diffusion. Like AD-LBR, this scheme is defined by the data at each point $z \in \Omega$ of a stencil $V(z)$, and of non-negative weights $\gamma_z$.

Let 
$\Diff(z) = \left(
\begin{array}{cc}
a & b \\ 
b & c
\end{array}
\right).$ In the diagonal case $b=0$, the scheme \WG~relies on the classical four points stencil. Otherwise note that 
\begin{equation*}
\frac a {|b|} - \frac {|b|} c = \frac {ac-b^2} {|b| c}> 0.
\end{equation*}
Let $p,q \in \Z\sm \{0\}$ be such that $b p q\geq 0$,
\begin{equation}
\label{ineqPQ}
\frac {|b|} c \leq \left | \frac p q \right| \leq \frac a {|b|},
\end{equation} 
and $\max(|p|,|q|)$ is minimal.  The scheme \WG~is defined by the six point stencil
\begin{equation*}
V(z) := \{(\pm 1,0), (0,\pm 1), \pm (p,q)\}
\end{equation*}
and the non-negative weights
\begin{align*}
2 \gamma(\pm 1,0) &= a-\frac p q b,&
2 \gamma(0,\pm 1) &= c-\frac q p b, \\
2 \gamma(\pm (p,q)) &= \frac b {p q}. 
\end{align*}
These coefficients are non-negative by construction, and consistency \eqref{DSum} is easily checked. 
Contrary to AD-LBR, the coordinate axes play a privileged role in A-NN. This introduces axis aligned artifacts which are visible in Figure \ref{fig:fingerprintzoom} (g).
\end{remark}

\begin{remark}[Stencil radius]
The two dimensional stencils of AD-LBR coincide with those of FM-LBR, a numerical scheme for anisotropic static Hamilton-Jacobi PDEs introduced in \cite{M12} the second author. As shown in Proposition 1.6 of \cite{M12}, the euclidean radius 
\begin{equation*}
r=\max\{\|v\|; v \in V(z)\}
\end{equation*}
 of this stencil is bounded by $\kappa(\Diff(z))$. 

In contrast, consider for $0 < \ve < 1/4$ the matrix 
\begin{equation*}
D := \left(
\begin{array}{cc}
1 & 1-2 \ve\\
1-2 \ve & 1-3 \ve
\end{array}
\right).
\end{equation*}
If follows from \eqref{ineqPQ} that $1+\ve \leq p/q + \cO(\ve^2) \leq 1+2 \ve$.
From this point, one easily obtains that $q \gtrsim \ve^{-1} \approx \kappa(D)^2$. The radius of the A-NN stencil, at a point $z\in \Omega$, may thus be of the order of $\kappa(\Diff(z))^2$.
The radii of the AD-LBR and \WG~stencils, computed for diffusion tensors of anisotropy $\kappa=10$ and of various orientations, are illustrated on Figure \ref{fig:Radius}.
\end{remark}

\begin{figure}
\centering
\includegraphics[width=6cm]{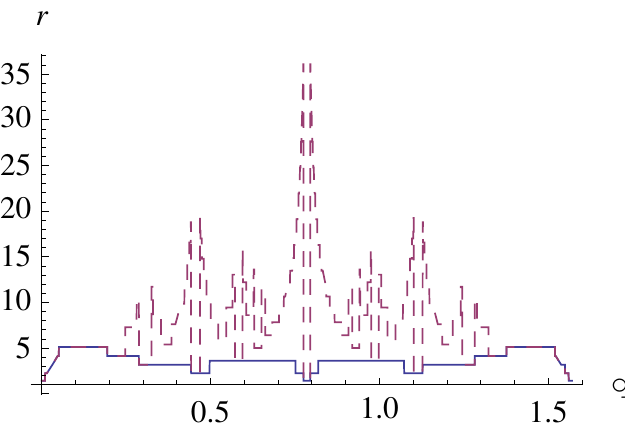} 
\caption{
Radius of the AD-LBR stencil (plain), and of the \WG~stencil (dashed), for a matrix $D_\theta$ of anisotropy ratio $\kappa = 10$ and eigenvector $(\cos\theta,\sin \theta)$. The AD-LBR stencil is here always the smallest, and its radius does not exceed $5.1$, versus $36.1$ for \WG.
}
\label{fig:Radius}
\end{figure}

\begin{figure}
\centering
\begin{tabular}{cc}
\includegraphics[width=6cm]{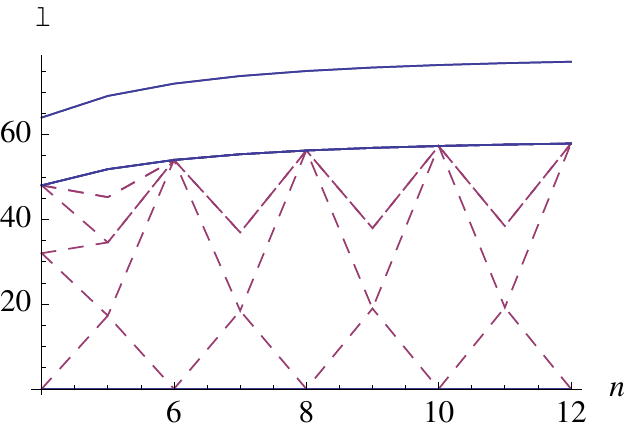} &
\hspace{-0.5cm}{\raise 1.5cm \hbox{
$
D = 
\left(
\begin{array}{cc}
19/2 & 6\\
6 & 4
\end{array}
\right)
$
}}
\end{tabular}
\caption{
The seven smallest eigenvalues of the operator $-\diver(D \nabla )$, on $[0,1]^2$ with periodic boundary conditions, discretized on a $n\times n$ grid with $4 \leq n \leq 12$. Plain: AD-LBR discretization; Dashed: \WG~discretization.
Some eigenvalues have multiplicities, hence less than 14 graphs are visible. 
Eigenvalues of the \WG~discretization here ``oscillate'' with the dimension.
In contrast, thanks to its asymptotic equivalence with a finite element scheme, the smallest eigenvalues of the AD-LBR discretization converge towards those of the continuous partial differential operator. 
}
\label{fig:Eigen}
\end{figure}

\begin{remark}[Scheme FD]
\label{rem:FD}
The operator $\diver(\Diff \nabla \cdot)$ is discretized using centered finite differences \cite{Mitchell80}. This involves quantities defined at half integer indices, and in particular the diffusion tensor is here given on the offsetted grid $(i+1/2,\, j+1/2)$, $(i,j)\in \Z^2$. For the sake of readability, we thus define $i^+ := i+1/2$ and $i^- := i-1/2$.
The gradient operator is discretized by:
\begin{equation*}
(\partial_x u)_{i^+,j}=u_{i+1,j}-u_{i,j}, \quad 
(\partial_y u)_{i,j^+}=u_{i,j+1}-u_{i,j}.
\end{equation*}
The divergence is defined as follows:
\begin{align*}
\diver(\DD\nabla u)_{i,j} &= \partial_x(\DD^{11}\partial_x u+\DD^{12}\partial_y u)_{i,j}\\
&+\partial_y(\DD^{21}\partial_x u+\DD^{22}\partial_y u)_{i,j},
\end{align*}
with
\begin{align*}
(\DD^{11}\partial_xu)_{i^+,j} &= \frac{1}{2}\left(\DD^{11}_{i^+,j^+}+\DD^{11}_{i^+,j^-}\right)(\partial_x u)_{i^+,j},\\
\partial_x(\DD^{11}\partial_xu)_{i,j} &= (\DD^{11}\partial_xu)_{i^+,j}-(\DD^{11}\partial_xu)_{i^-,j}\\
(\DD^{21}\partial_xu)_{i^+,j^+} &= \dfrac{1}{2}\DD^{21}_{i^+,j^+}\left((\partial_x u)_{i^+,j}+(\partial_x u)_{i^+,j+1}\right),\\
\partial_y(\DD^{21}\partial_xu)_{i,j} &= \dfrac{1}{2}\left((\DD^{21}\partial_xu)_{i^+,j^+}-(\DD^{21}\partial_xu)_{i^+,j^-}\right.\\
&\phantom{=} \left.+(\DD^{21}\partial_xu)_{i^-,j^+}-(\DD^{21}\partial_xu)_{i^-,j^-}\right),
\end{align*}
and similar terms involving $\partial_yu$.
\end{remark}

\begin{table}[h]
\centering
\small
\caption{
The stencil coefficients for different \emph{constant} diffusion tensors, and the different schemes presented. The value of the anisotropy ratio $\kappa$ is given in the second row, and the orientation of the principal axis is $\theta = \pi/6$, see \eqref{def:DTest}. The bold coefficient indicates the center node. In some examples we present for clarity reasons only half of the stencil (the other half can be deduced by symmetry).
Stencil entries are highlighted when they are positive and off-center - an undesirable property which gives rise to stability issues. For small anisotropies, $\kappa \leq 1+\sqrt 2$, one has AD-LBR = W-NN = \WG.
}
\label{tab:coeff1}
\vspace{0.3cm}
\hspace{-1cm}
\begin{tabular}{|c|c|c|c|}
\hline
$\kappa$ & 

$\kappa=1$ ($\Diff = \Id$)& 

$\kappa=\sqrt{2}$\\

\hline
\begin{minipage}{1.4cm}stencil for\\ AD-LBR\end{minipage} & 
$
\begin{array}{ccc}
0 & -1 & 0\\
-1 & {\bf 4} & -1\\
0 & -1 & 0
\end{array}
$ 

& 

$
\begin{array}{ccc}
0 & -0.41 & -0.22\\
-0.66 & {\bf 2.57} & -0.66\\
-0.22 & -0.41 & 0
\end{array}
$\\

 \hline
\begin{minipage}{1.4cm}stencil for\\  FD\end{minipage}  & 
$
\begin{array}{ccc}
0 & -1 & 0\\
-1 & {\bf 4} & -1\\
0 & -1 & 0
\end{array}
$

&

$
\begin{array}{ccc}
\alert{0.11} & -0.63 & -0.11\\
-0.88 & {\bf 3} & -0.88\\
-0.11 & -0.63 & \alert{0.11}
\end{array}
$\\

 \hline
 \begin{minipage}{1.4cm}stencil for\\ Q1\end{minipage} & 
 $
 \dfrac{1}{3}
 \left(
 \begin{array}{ccc}
-1 & -1 & -1\\
-1 & {\bf 8} & -1\\
-1 & -1 & -1
\end{array}
\right)
$ 

& 

$
\begin{array}{ccc}
-0.14 & -0.13 & -0.36\\
-0.38 & {\bf 2} & -0.38\\
-0.36 & -0.13 & -0.14
\end{array}
$\\

\hline
 \begin{minipage}{1.4cm}stencil for\\ WS\end{minipage} & 
 $
 \begin{array}{ccc}
 -0.1 & -0.06 & -0.02\\
 \alert{0.12} & 0 & -0.06\\
{\bf 0.46} & \alert{0.12} & -0.1\\
 \alert{0.12} & 0 & -0.06\\
 -0.1 & -0.06 & -0.02
\end{array}
$ 

& 

$
\begin{array}{ccc}
 -0.06 & -0.05 & -0.02\\
 \alert{0.01} & -0.04 & -0.06\\
 {\bf 0.35} & \alert{0.07} & -0.09\\
 \alert{0.01} & \alert{0.04} & -0.04\\
 -0.06 & -0.02 & -0.01
\end{array}
$\\

 \hline
 \begin{minipage}{1.4cm}stencil for\\ W-NN\end{minipage} & 
 $
 \begin{array}{ccc}
0 & -1 & 0\\
-1 & {\bf 4} & -1\\
0 & -1 & 0
\end{array}
$ 

& 

$
\begin{array}{ccc}
0 & -0.41 & -0.22\\
-0.66 & {\bf 2.57} & -0.66\\
-0.22 & -0.41 & 0
\end{array}
$\\

\hline
 \begin{minipage}{1.4cm}stencil for\\ \WG \end{minipage} & 
 $
 \begin{array}{ccc}
0 & -1 & 0\\
-1 & {\bf 4} & -1\\
0 & -1 & 0
\end{array}
$ 

& 

$
\begin{array}{ccc}
0 & -0.41 & -0.22\\
-0.66 & {\bf 2.57} & -0.66\\
-0.22 & -0.41 & 0
\end{array}
$\\

\hline
\end{tabular}
\end{table}

\begin{table}[htb]
\centering
\small
\caption{The stencil coefficients for different metrics and the different schemes presented, similarly to Table \ref{tab:coeff1} but with more pronounced anisotropies. For the scheme \WG~some points of the stencil are too far from the center node to be represented here, so we indicate the coordinates of these points and the associated coefficient.}\label{tab:coeff2}
\vspace{0.3cm}
\begin{tabular}{|c|c|c|}
\hline
$\kappa$ & 

$\kappa=\sqrt{10}$ 

& 

 $\kappa=\sqrt{50}$ 
\\ 
\hline
\begin{minipage}{1.4cm}stencil for\\ AD-LBR\end{minipage} &  

$
\begin{array}{ccc}
 0 & -0.26 & -0.06\\
 {\bf 1.16}  & -0.26 & 0
\end{array}
$

&

$
\begin{array}{ccc}
 0 & -0.11 & -0.16\\
 {\bf 0.55} & -0.01 & 0
\end{array}
$

\\ \hline
\begin{minipage}{1.4cm}stencil for\\  FD\end{minipage}  &

$
\begin{array}{ccc}
\alert{0.19} & -0.32 & -0.19\\
-0.77 & {\bf 2.2} & -0.77\\
-0.19 & -0.32 & \alert{0.19}
\end{array}$

&

$\begin{array}{ccc}
\alert{0.21} & -0.27 & -0.21\\
-0.76 & {\bf 2.04} & -0.76\\
-0.21 & -0.27 & \alert{0.21}
\end{array}$

 \\ \hline
 \begin{minipage}{1.4cm}stencil for\\ Q1\end{minipage} &

$\begin{array}{ccc}
\alert{0.01} & \alert{0.04} & -0.38\\
-0.41 & {\bf 1.47} & -0.41\\
-0.38 & \alert{0.04} & \alert{0.01}
\end{array}$

&

$
\begin{array}{ccc}
\alert{0.04} & \alert{0.08} & -0.38\\
-0.42 & {\bf 1.36} & -0.42\\
-0.38 & \alert{0.08} & \alert{0.04}
\end{array}
$

 \\ \hline
 \begin{minipage}{1.4cm}stencil for\\ WS\end{minipage} &

$
\begin{array}{ccc}
 -0.02 & -0.04 & -0.02\\
 \alert{0.09} & -0.08 & -0.07\\
 {\bf 0.25} & \alert{0.04} & -0.08\\
\alert{0.09} & \alert{0.08} & -0.02\\
 -0.02 & \alert{0.004} & -0.003 
\end{array}
$

&

$
\begin{array}{ccc}
 -0.02 & -0.04 & -0.02\\
 \alert{0.09} & -0.08 & -0.07\\
 {\bf 0.24} & \alert{0.03} & -0.08\\
\alert{0.09} & \alert{0.08} & -0.02\\
 -0.02 & \alert{0.01} & -0.002
\end{array}
$

 \\ \hline
 \begin{minipage}{1.4cm}stencil for\\ W-NN\end{minipage} &

$
\begin{array}{ccc}
0 & \alert{0.06} & -0.39\\
-0.39 & {\bf 1.42} & -0.39\\
-0.39 & \alert{0.06} & 0
\end{array}
$

&

$
\begin{array}{ccc}
0 & \alert{0.16} & -0.42\\
-0.33 & {\bf 1.19} & -0.33\\
-0.42 & \alert{0.16} & 0
\end{array}
$
\\
 \hline
 \begin{minipage}{1.4cm}stencil for\\ \WG \end{minipage} &

$
\begin{array}{ccc}
 -0.07 & 0\\
{\bf 0.64} & -0.19
\end{array}
$

&

$
\begin{array}{ccc}
 -0.01 & 0\\
{\bf 0.17} & -0.05
\end{array}
$

\\
&$\gamma(3,2)=-0.06$ & $\gamma(5,3)=-0.03$  

\\
\hline
\end{tabular}
\end{table}

\begin{table}[h]
\centering
\small
\caption{
Largest eigenvalue of the discretized operator $-\diver(\DD\nabla)$, for the constant metric $\Diff=D$, where the matrix $D$ is given on Tables \ref{tab:coeff1} and \ref{tab:coeff2}.
The time step, in the explicit discretization of \eqref{eq:parabolique}, should not exceed the inverse of this value. 
}\label{tab:ev1}
\vspace{0.3cm}
\begin{tabular}{|c|c|c|c|c|}
\hline

$\kappa$ & $\kappa=1$ & $\kappa=\sqrt{2}$ & $\kappa = \sqrt {10}$ & $\kappa = \sqrt{50}$ \\
\hline
\begin{minipage}{1.4cm}eigenvalue\\ AD-LBR\end{minipage} & 8 & 4.27 &2.06 & 1.06\\
\hline
\begin{minipage}{1.4cm}eigenvalue\\  FD\end{minipage} & 8 &6.22 & 5.06 & 4.85\\
 \hline
\begin{minipage}{1.4cm}eigenvalue\\ Q1\end{minipage} & 5.7 &  4.94 & 4.32 & 4.20\\
\hline
\begin{minipage}{1.4cm}eigenvalue\\ WS\end{minipage} & 1 &  1 & 1 & 1
 \\ \hline
\begin{minipage}{1.4cm}eigenvalue\\ W-NN\end{minipage} & 8 &  4.27 & 3.1 & 3.02\\ \hline
\begin{minipage}{1.4cm}eigenvalue\\ \WG\end{minipage} & 8 &  4.27 & 1.04 & 0.3\\
\hline
\end{tabular}
\end{table}

%
%
%

\subsection{A test case with an explicit solution} 
\label{ssec:explicit}

Consider an image $v\in L^2(\Omega)$, defined on a domain $\Omega$, and a diffusion tensor field $\DD : \Omega \to S_2^+$. 
A classical approach to restore the image $v$, if it has been corrupted by additive noise, is to find $u \in H^1(\Omega)$ which minimizes:
\begin{equation}
\label{eq:restoration}
j(u)=\int_\Omega |u-v|^2 + \lambda \int_\Omega ||\nabla u||^2_{\DD}.
\end{equation}
In other words, $u$ is a penalized least squares approximation of $v$. 
The parameter $\lambda>0$ should be adjusted so as to avoid excessive smoothing (for large $\lambda$), or insufficient denoising (for small $\lambda$).
The solution $u$ can be characterized as the solution to the static elliptic PDE:
\begin{equation}
\label{elliptic}
\left\{
\begin{array}{ll}
-\lambda\diver(\DD\nabla u)+u = v, & \text{ on } \Omega.\\
\<\nabla u, n\>=0, & \text{ on } \partial\Omega.
\end{array}
\right.
\end{equation}
In applications \cite{PM90,W96} the diffusion tensor $\Diff$ is usually adapted to the local image structure, in order to avoid smoothing the edges of $v$.
We construct below a test case (image $v$ and tensor field $\DD$), for which the solution $u$ is known analytically. 

In order to obtain an analytic solution, we first consider a separable problem where the image is invariant by translation along the horizontal axis, and the metric is constant with axes parallel to the coordinate axes. This first problem is invariant under translations along the $x$-axis, and therefore boils down to a 1-dimensional problem. This separable problem is then transported by a diffeomorphism in order to obtain a new problem where the axes of the metric are no more parallel to the coordinate axes. 

The analytical image is composed of a black and a white stripe: $v_0(x,y)={\bf 1}_{y<0.5}$, see Figure \ref{fig:stripe}. Given $\kappa\geq 1$, we consider the constant diffusion tensor 
\begin{equation*}
\DD_0=\left(
\begin{array}{cc}
1 & 0\\
0 & \kappa^{-2} 
\end{array}
\right).
\end{equation*}

The analytical solution $u_0$ of \eqref{elliptic}, applied to $\Diff_0$ and $v_0$, is known in the case of the infinite domain $\Omega = \R^2$.
In Fourier domain all the coefficients are real and:
\begin{equation*}
\widehat{u_0}(\xi)=\widehat{v_0}(\xi)/(1+\<\xi,\DD_0 \xi\>).
\end{equation*}
This separable problem is transformed using the following diffeomorphism: for $(x,y) \in \Omega$
\begin{equation*}
f(x,y)=(x,y+\alpha\cos(2\pi x)).
\end{equation*}
The Jacobian of $f$ is
\begin{equation*}
J(x,y)=\left(
\begin{array}{cc}
1 & 0\\
-2\pi \alpha\sin(2\pi x) & 1
\end{array}
\right)
\end{equation*}
We apply the different restoration schemes to the image  $v=v_0\circ f$, and the following diffusion tensor: 
\begin{align*}
\DD(z)&=|\det J(z)| \ J(z)^{-1}\,\DD_0\, (J(z)^{-1})^\trans\\
&=J(z)^{-1}\, \DD_0 \, (J(z)^{-1})^\trans=\left(
\begin{array}{cc}
1 & s\\
s & s^2+\kappa^{-2}
\end{array}
\right),
\end{align*}
where we denoted $z=(x,y) \in \Omega$ and $s=2\pi \alpha \sin(2\pi x)$. The numerical solution is compared to the analytical function $u = u_0 \circ f$, which is the exact solution in the case of the infinite domain $\Omega = \R^2$.
This numerical solution was obtained on the bounded domain $\Omega = [0,1[^2$, equipped with reflecting boundary conditions. 
Numerical evidence suggests that this change of domain and of boundary conditions has only an anecdotic impact on the solution of \eqref{elliptic}, with the parameters chosen in this test case.

We used $\alpha := 1/3$ in the numerical experiments. The maximum value of $\kappa(\Diff(x))$, among all $x \in \Omega$, is equivalent to $\kappa_{\max} := \kappa \sqrt{1+(2 \pi \alpha)^2}  \simeq 2.3 \kappa$.

\begin{figure}[htb]
\begin{centering}
\includegraphics[width=41mm]{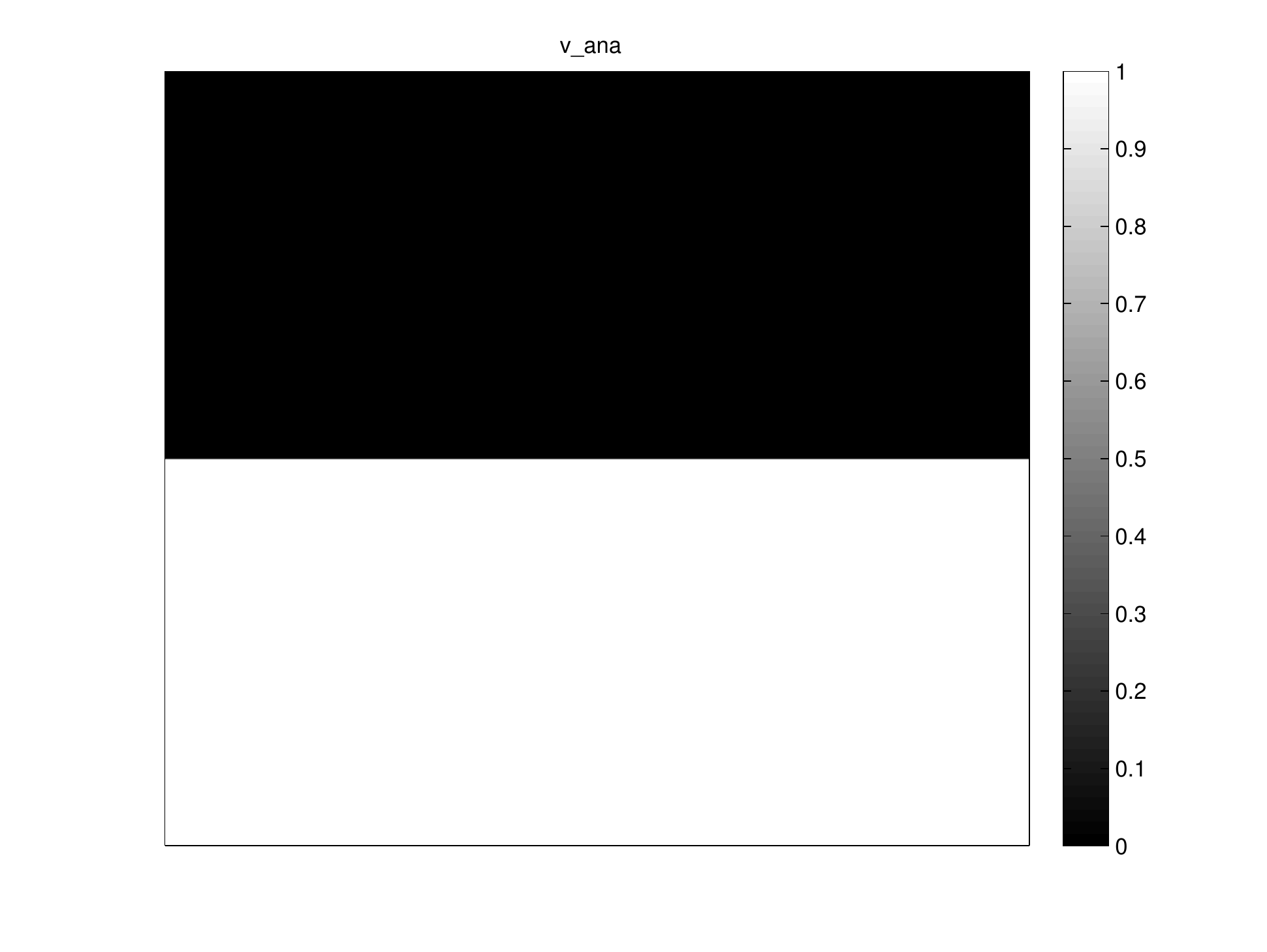}
\includegraphics[width=41mm]{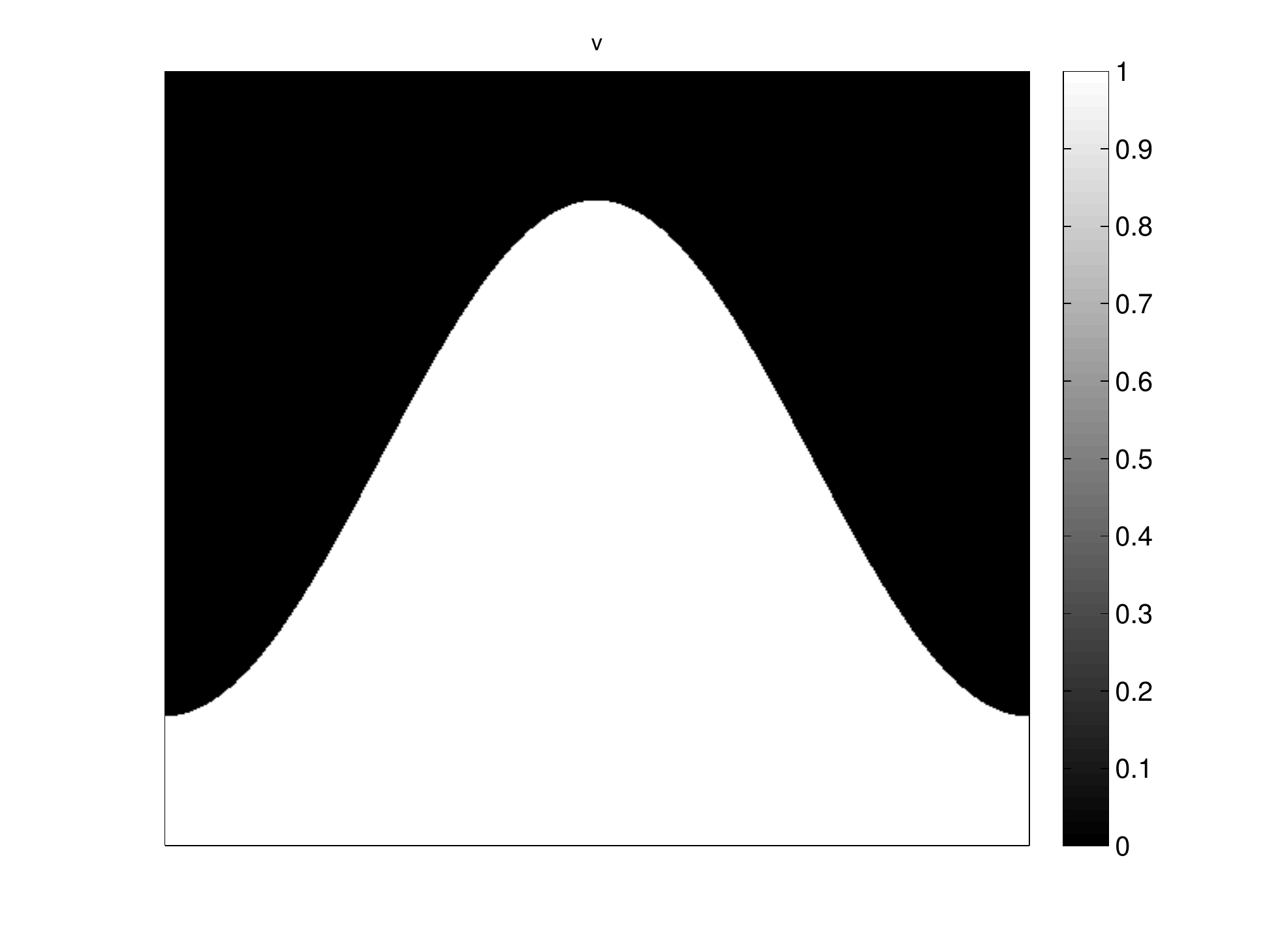}
\caption{Left: image $v_{\rm ana}$. Right: image $v=v_{\rm ana}\circ f$ transformed by the diffeomorphism $f$.}\label{fig:stripe}
\end{centering}
\end{figure}

\subsection{Results for the synthetic test case}

We present in Figure \ref{fig:comparaisonk1} the performance results of the different schemes, for different values of the anisotropy $\kappa$, obtained on a series of grids of size ranging from $100\times 100$ to $1200\times 1200$. The anisotropy varies from $\kappa=2$ to $\kappa=10$, which are relevant values for imaging applications, see the numerical experiments in \S \ref{sec:co-en-dif}. The quality of a scheme is measured by the $L^2$ difference and the $H^1$ semi-norm difference between the numerical solution and the analytical solution. Note that the error is concentrated close to the discontinuity, since the solution tends rapidly to a constant (0 or 1) far from the discontinuity. 
We chose the smoothing parameter $\lambda=10^{-3}$ in \eqref{eq:restoration}.
The linear equation obtained by the discretization of \eqref{elliptic} is solved using Conjugate Gradient. 

We also tested extreme anisotropies, $\kappa \geq 100$ (thus $\kappa_{\max} \geq 230$), which can be relevant in physics related applications. None of the tested schemes showed convincing results: methods based on fixed stencils fail because the discrete operator looses positivity, while the AD-LBR (and \WG~even more) suffers from under-sampling due to the large radius of its stencils.
We thus refer to \cite{degond2012} for a radically different approach tailored for this setting. 
This method introduces an auxiliary one-dimensional unknown, which is constant on the field lines (obtained in a preprocessing step) of the anisotropy direction field, and varies orthogonally to them.


The performance advantage of the AD-LBR is particularly clear when the error is measured in the $H^1$ semi-norm: for the anisotropy $\kappa = 10$ and the resolution $500\times 500$, which are relevant values in image processing, AD-LBR outperforms its alternatives by a factor ranging from 3 to 5.

\begin{figure}[htb]
\begin{centering}
\includegraphics[width=4.1cm]{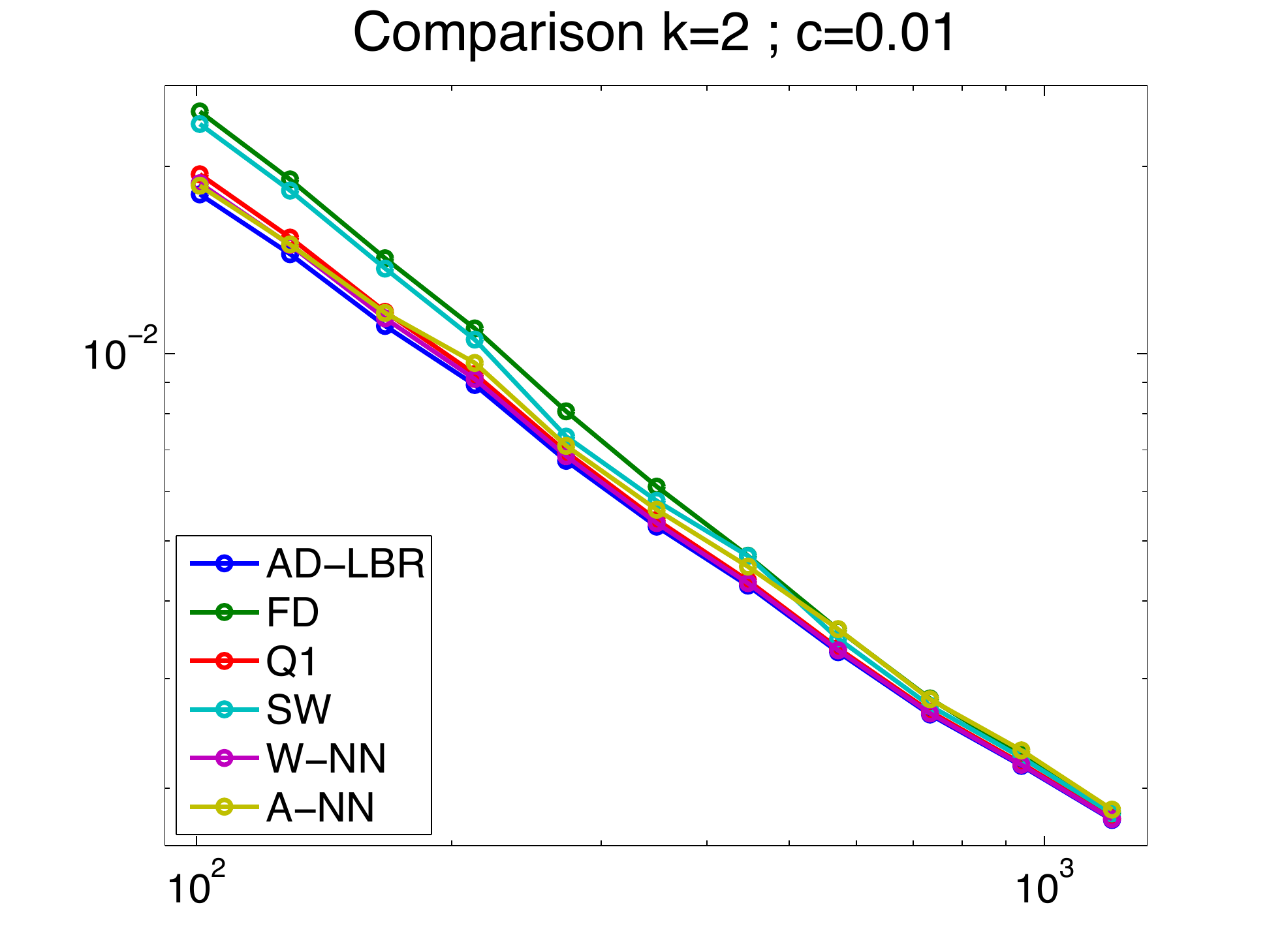}
\includegraphics[width=4.1cm]{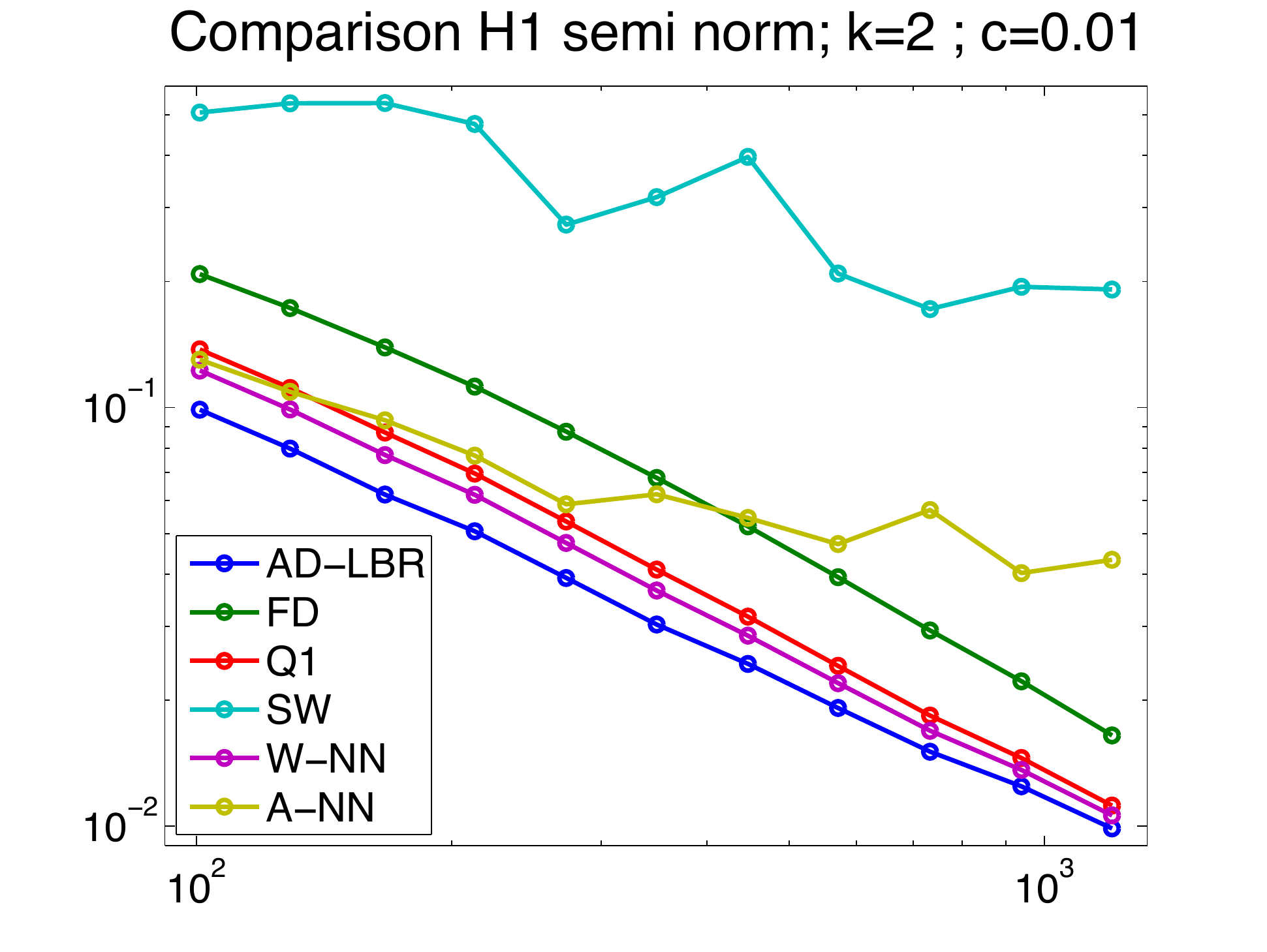}
\includegraphics[width=4.1cm]{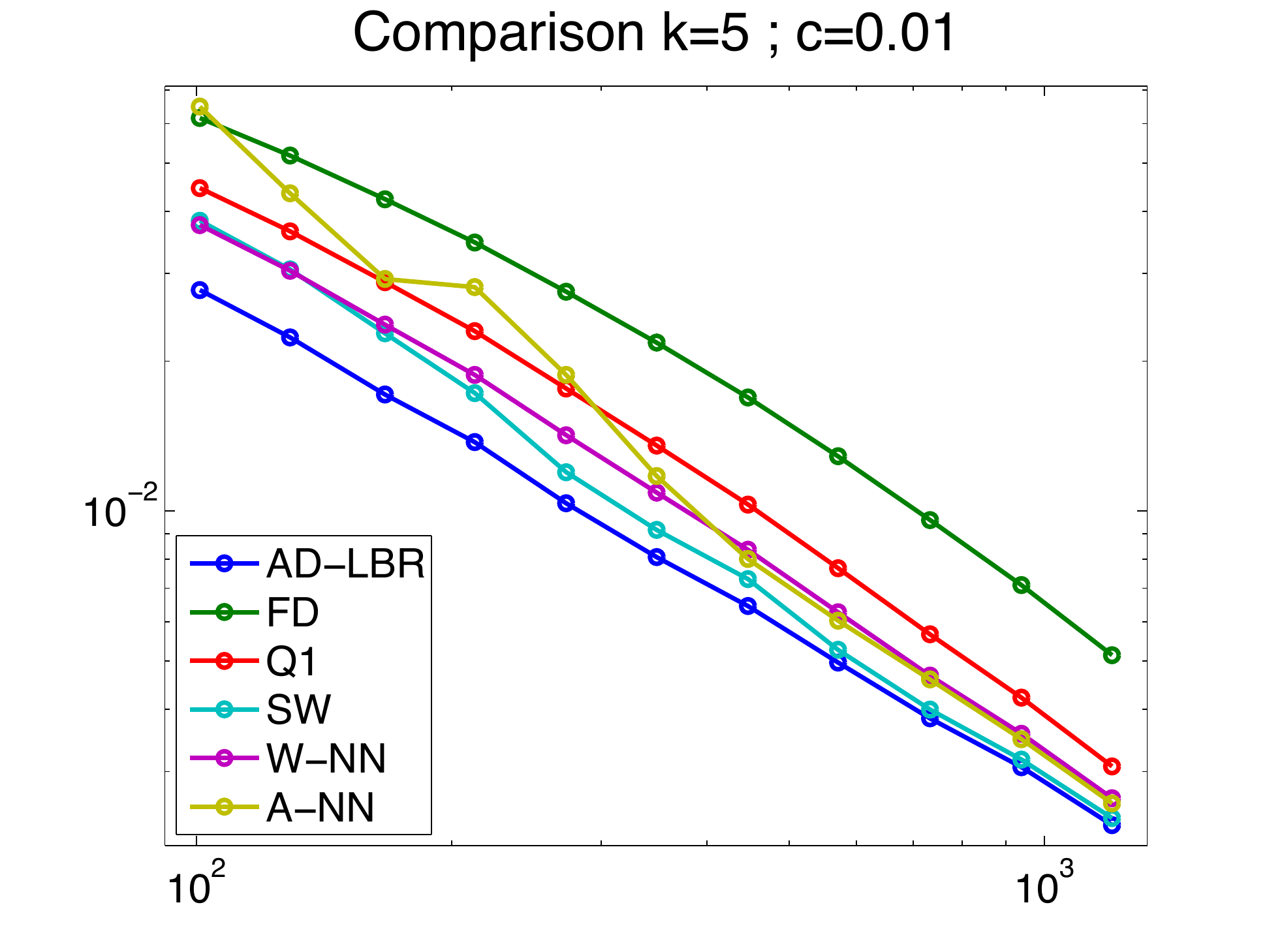}
\includegraphics[width=4.1cm]{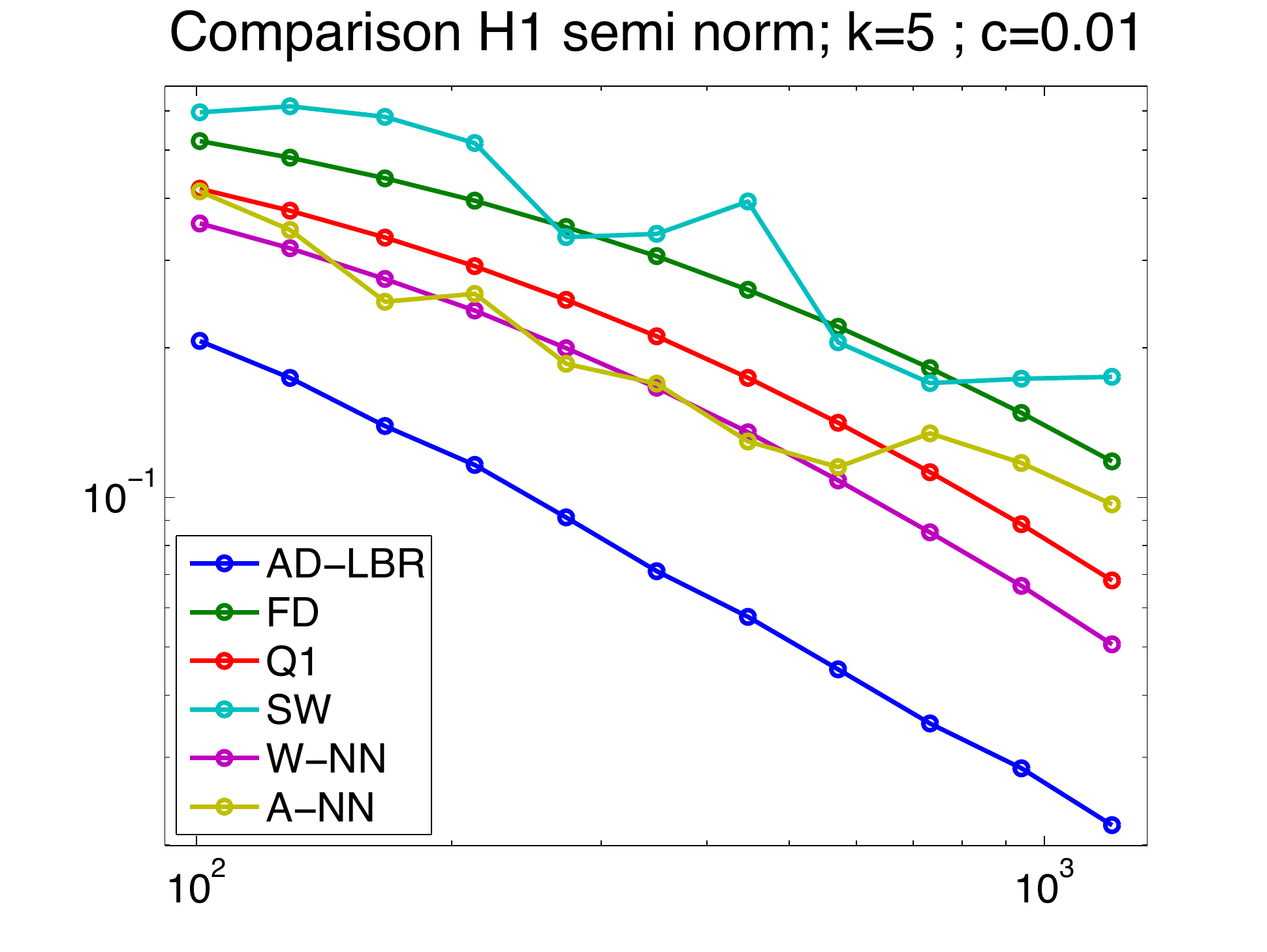}
\includegraphics[width=4.1cm]{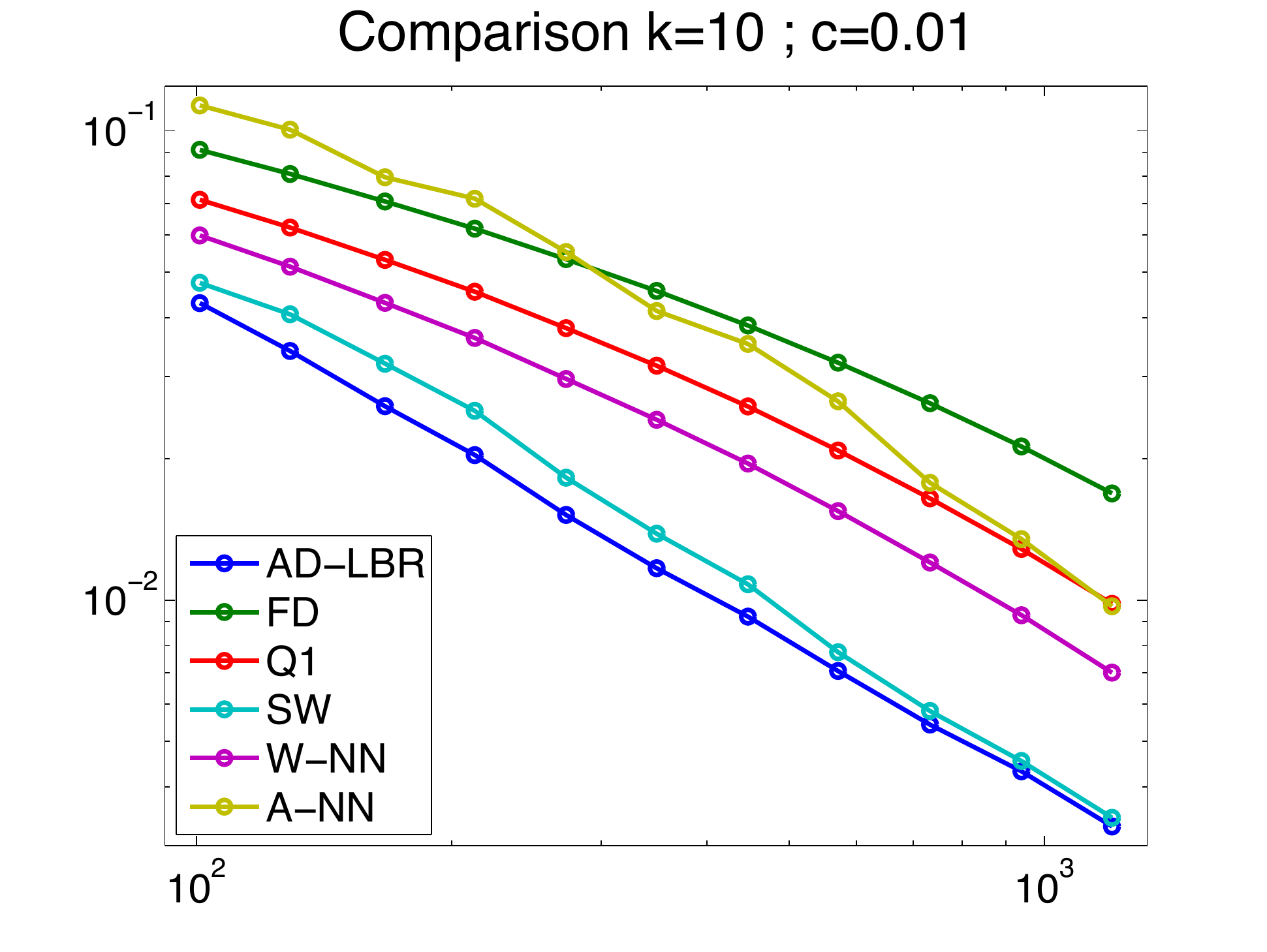}
\includegraphics[width=4.1cm]{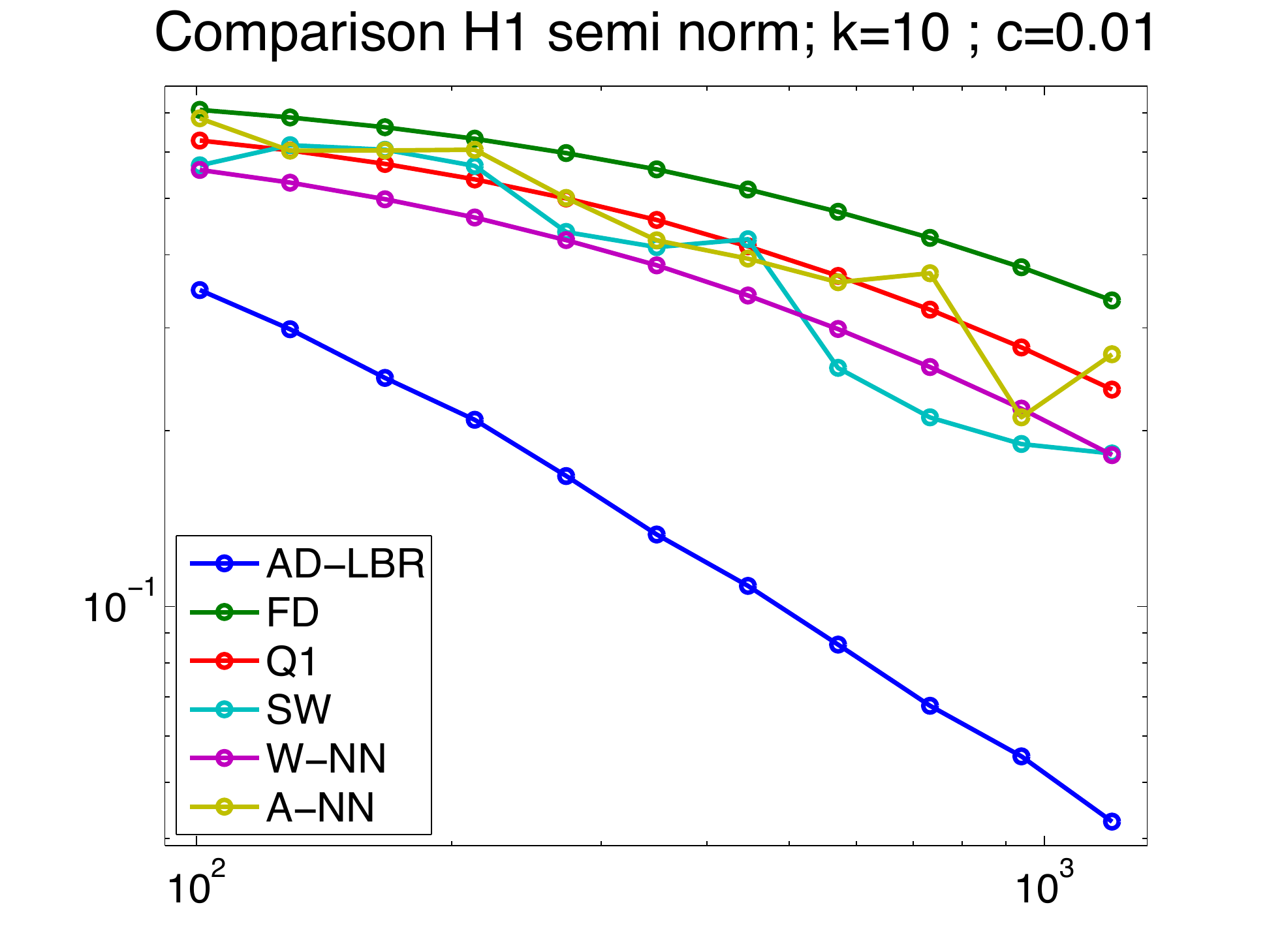}
\caption{
Numerical results for the synthetic test case, with different values of the anisotropy factor: $\kappa=2$ (top row), $\kappa=5$ (middle row), $\kappa = 10$ (bottom row).
Vertical axis: relative error in $L^2$ norm (left column), or $H^1$ semi-norm (right column), for the six schemes tested. Horizontal axis: integer $N$, where the image resolution is $N\times N$. Since the tested schemes are first order, numerical error is expected to be proportional to $N^{-1}$. Log-log scale. 
}
\label{fig:comparaisonk1}
\end{centering}
\end{figure}


\subsection{Coherence-enhancing diffusion}
\label{sec:co-en-dif}

In order to document the interest of our discretization, we implement Coherence-Enhancing Diffusion \cite{W98} using the different numerical schemes at our disposal. 
The following parabolic equation is considered:
\begin{equation}\label{eq:parabolique}
\partial_t u=\diver(\DD(J_\rho(\nabla u_\sigma))\nabla u).
\end{equation}
This equation is non-linear since the diffusion tensor depends on the solution $u$. This tensor also depends on four user defined parameters $\sigma, \rho, C \in \R_+$, $\alpha \in ]0,1[$. 
Let $K_\sigma$ (resp. $K_\rho$), be the Gaussian kernel of variance $\sigma$ (resp. $\rho$). Define the convolution $u_\sigma := K_\sigma \star u$, and the structure tensor $J_\rho := K_\rho \star(\nabla u_\sigma \nabla u_\sigma^T)$.
The diffusion tensor $\DD(J_\rho)$ possesses the same eigenvectors $(v_1,v_2)$ as $J_\rho$, and if the eigenvalues of $J_\rho$ are $\mu_1\ge \mu_2$ then the eigenvalues of $\DD(J_\rho)$ are 
\begin{align*}
\lambda_1 & :=\alpha\\
\lambda_2 & :=\alpha+(1-\alpha)\exp\left(\dfrac{-C}{(\mu_1-\mu_2)^2}\right).
\end{align*}
This ensures that one smoothes preferably along the coherence direction $v_2$, with a diffusivity that increases with respect to the coherence $(\mu_1-\mu_2)^2$. When the time parameter $t$ becomes large, the image tends to a constant image, therefore it is necessary to stop the process at some finite time $T$. The ratio of the eigenvalues is bounded by $\lambda_2/\lambda_1\le 1/\alpha$, hence $\kappa\le 1/\sqrt{\alpha}$.

We used an explicit time discretization for \eqref{eq:parabolique}, with time step $\Delta t$. The image $u^{n+1}$ at time  $(n+1)\Delta t$ is defined by the explicit equation:
\begin{equation*}
\dfrac{u^{n+1}-u^n}{\Delta t}=\diver(\DD(J_\rho(\nabla u_\sigma^n))\nabla u^n).
\end{equation*}

The parameters  used in our simulation were: $\sigma=0.5$, $\rho=4$, $C=10^{-5}$, $\alpha=10^{-2}$ and $\Delta t=0.02$. This gives a maximum anisotropy of $\kappa=10$. The algorithm was applied to a fingerprint image. The results obtained for $T=10$ are shown in Figures \ref{fig:fingerprint} and \ref{fig:fingerprintzoom}, and they document the ability of our scheme to close interrupted lines more efficiently than the other schemes. 
The largest eigenvalue of the discrete operator $-\diver(\DD \nabla)$ at $t=0$ is given in Table \ref{tab:evced} for the different schemes. As was already noticed in the constant metric case, it turns out that AD-LBR has the smallest eigenvalues among all schemes, except for scheme WS. This property allows (although this was not done in our numerical experiments) to use larger time steps for AD-LBR than for the other schemes.

Note also that ridges are clearer, and valleys are darker, using AD-LBR than with the other schemes.
(Gray-scale range is the same for all images, see also Figure \ref{fig:coupes}).
This reflects the fact that AD-LBR avoids, better than the other schemes, smoothing transversally to the orientation encoded in the continuous anisotropic PDE \eqref{eq:parabolique}.

\begin{remark}[Computation time]
Numerical solvers of the parabolic PDE \eqref{eq:parabolique} combine three main components: (i) Constructing the diffusion tensor. (ii) Assembling the discretization stencils and the operator sparse matrix. (iii) Performing an explicit time step. Components (i) and (ii) are executed exactly the same number of times, while step (iii) is generally more frequent: in order to save CPU time, one typically does not update the diffusion operator at each time step.
We produced a C++ implementation of AD-LBR, within the Insight Toolkit open source library. 
Although our code is neither parallel nor aggressively optimized, we believe that comparing the CPU times for steps (i), (ii) and (iii) is informative, and allows to estimate the additional cost of AD-LBR which is essentially contained in step (ii).

For our 2D Coherence-Enhancing Diffusion (CED) experiment, on the $512\times 512$ fingerprint image, (i) takes 0.21s, (ii) 0.027s, (iii) 0.005s.  For our 3D CED Experiment, on $100\times 100\times 100$ synthetic data, (i) takes 1.35s, (ii) 0.51s, (iii) 0.035s. In both cases, the AD-LBR specific step (ii) is dominated by the construction of the diffusion tensor (i). Step (ii) may also be dominated by the mere cost (iii) of iterations, provided the operator is updated less than once every 6 explicit steps in 2D (14 in 3D). To our eyes, the limited additional cost (ii) of AD-LBR is acceptable in view of the strong theoretical guarantees, and qualitative improvements, brought by this scheme.


\end{remark}

\begin{table}
\caption{Largest eigenvalue of the discretized operator $-\diver(\DD\nabla)$, where $\DD=\DD(J_\rho(\nabla u_\sigma))$ at $t=0$.
}\label{tab:evced}
\hspace{-1cm}
\begin{tabular}{|c|c|c|c|c|c|c|}
\hline
scheme & AD-LBR & FD & Q1 & WS & W-NN & \WG\\
\hline
eigenvalue & 3.75 & 5.67 & 5.09 & 0.96 & 3.83 & 6.23\\
\hline
\end{tabular}
\end{table}

\begin{figure}[htb]
\begin{centering}
\subfigure[][Original image]{\includegraphics[width=4.1cm]{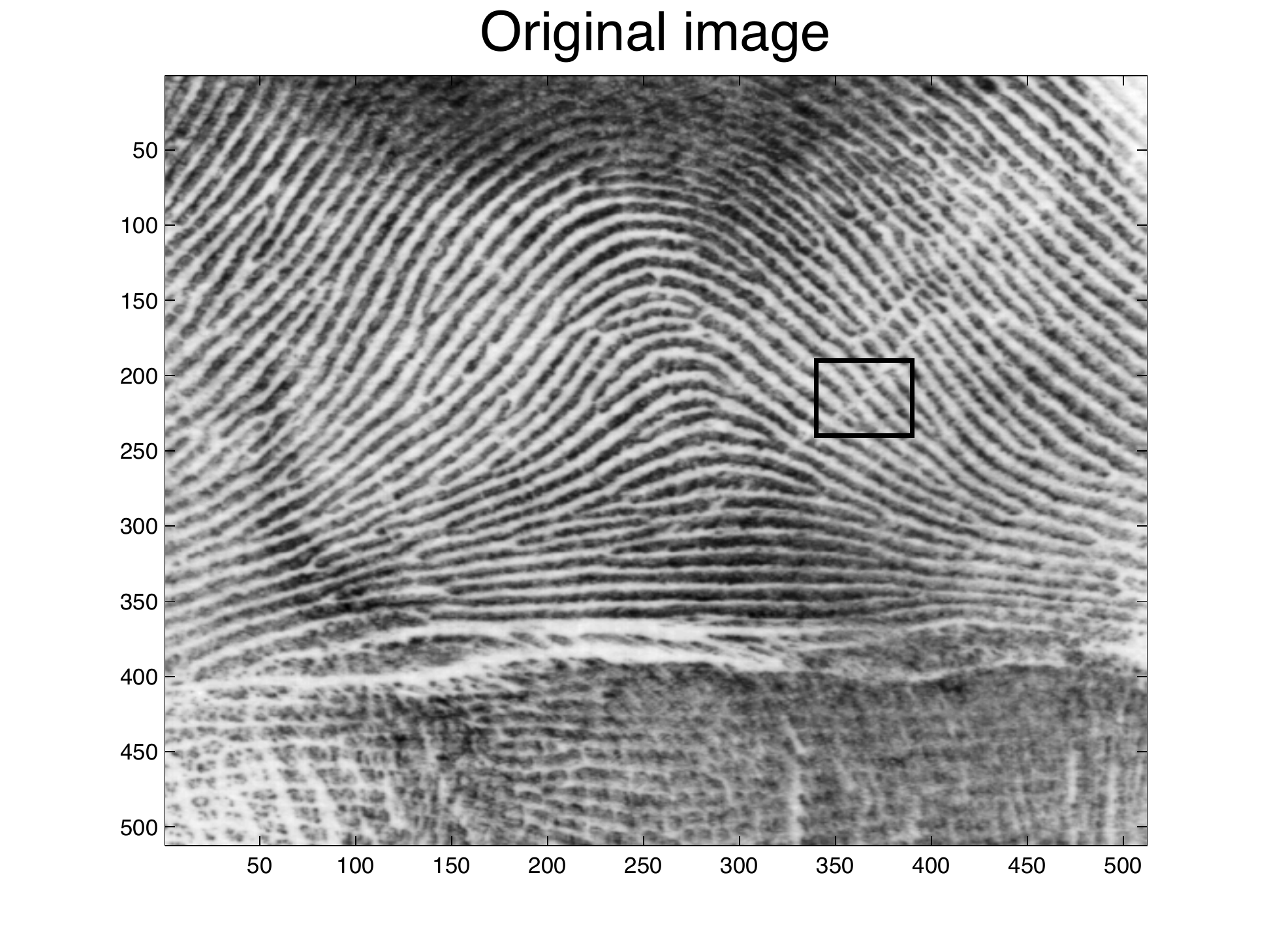}}
\subfigure[][AD-LBR]{\includegraphics[width=4.1cm]{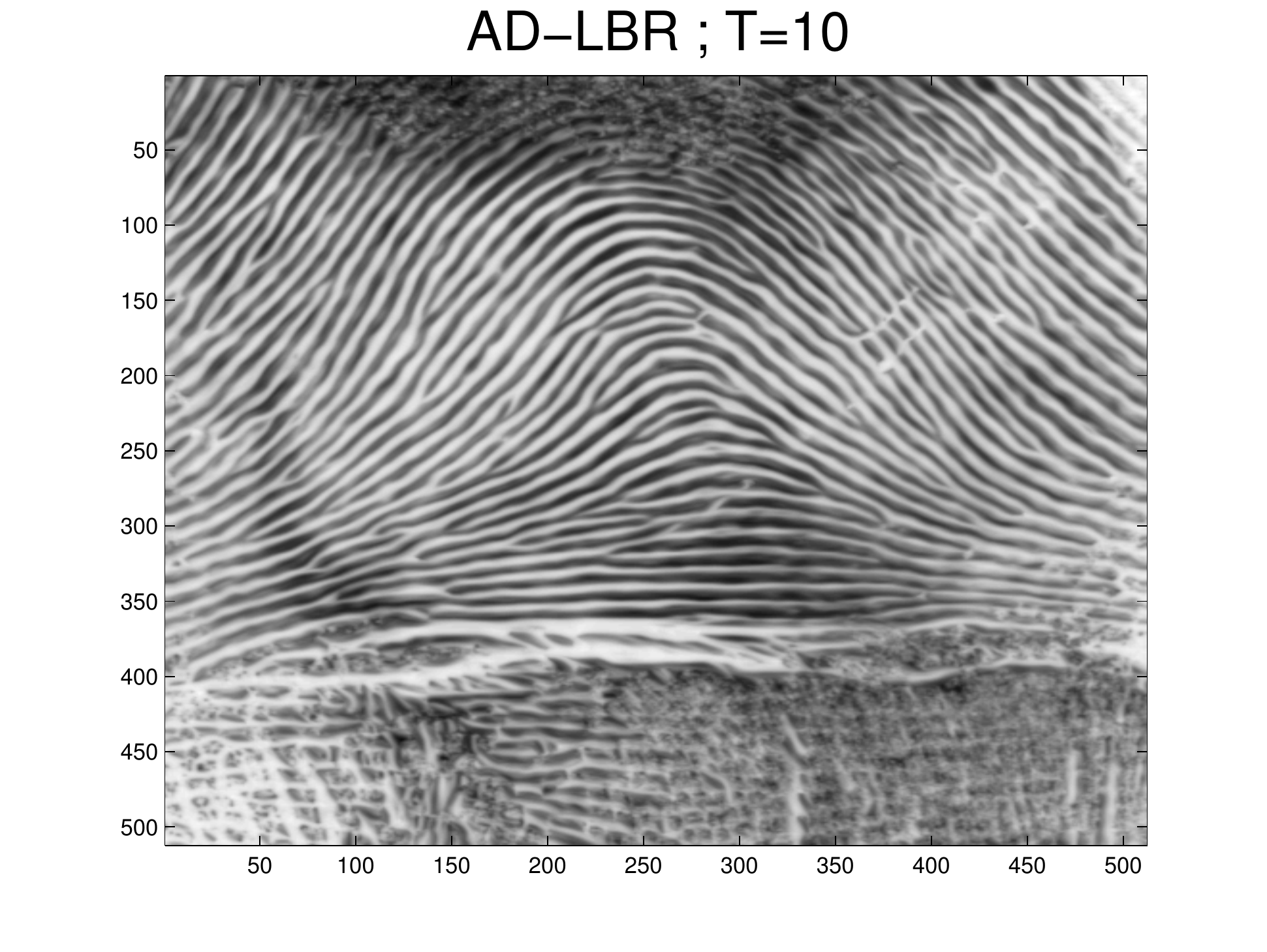}}
\subfigure[][FD]{\includegraphics[width=4.1cm]{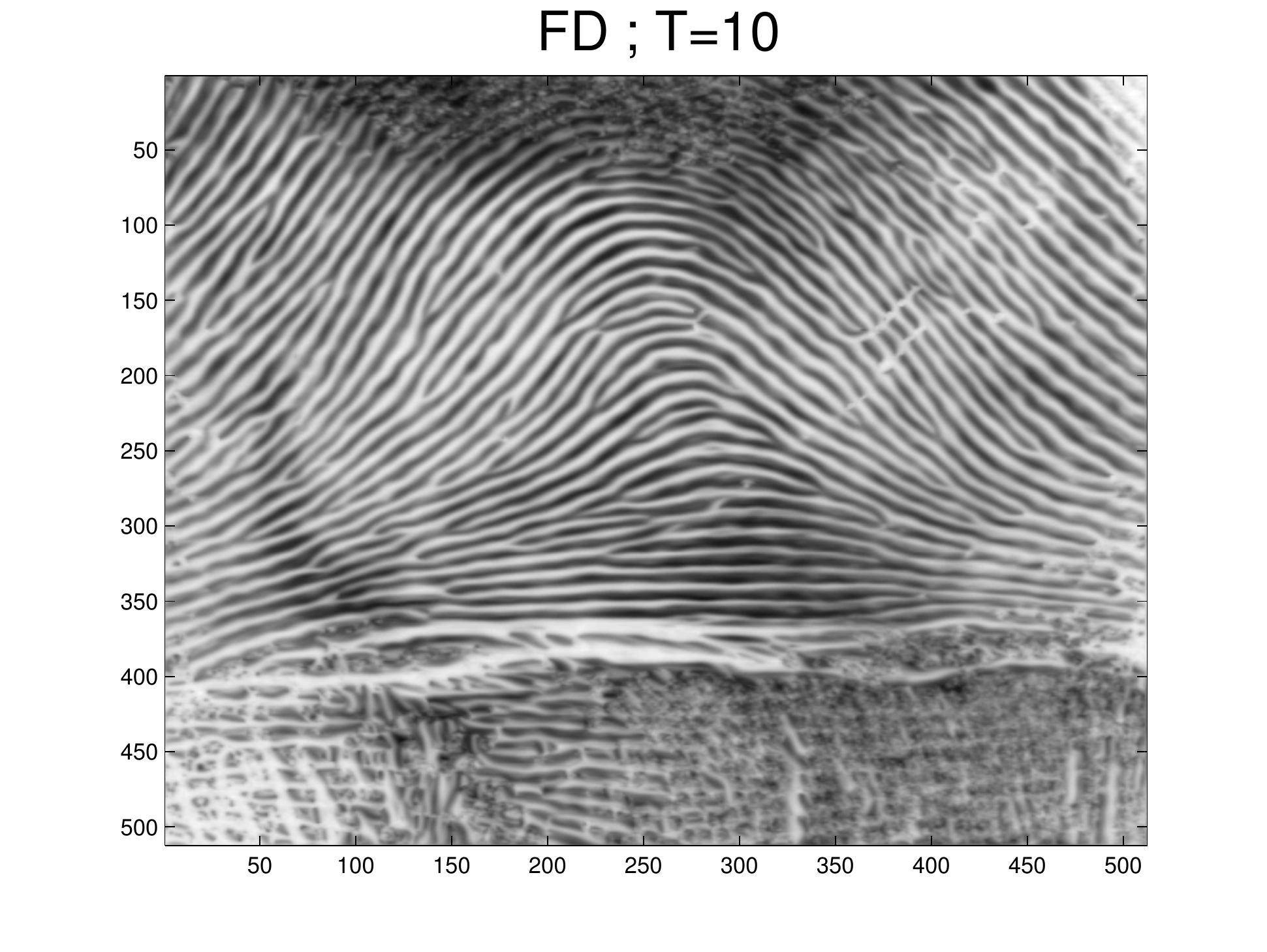}}
\subfigure[][Q1]{\includegraphics[width=4.1cm]{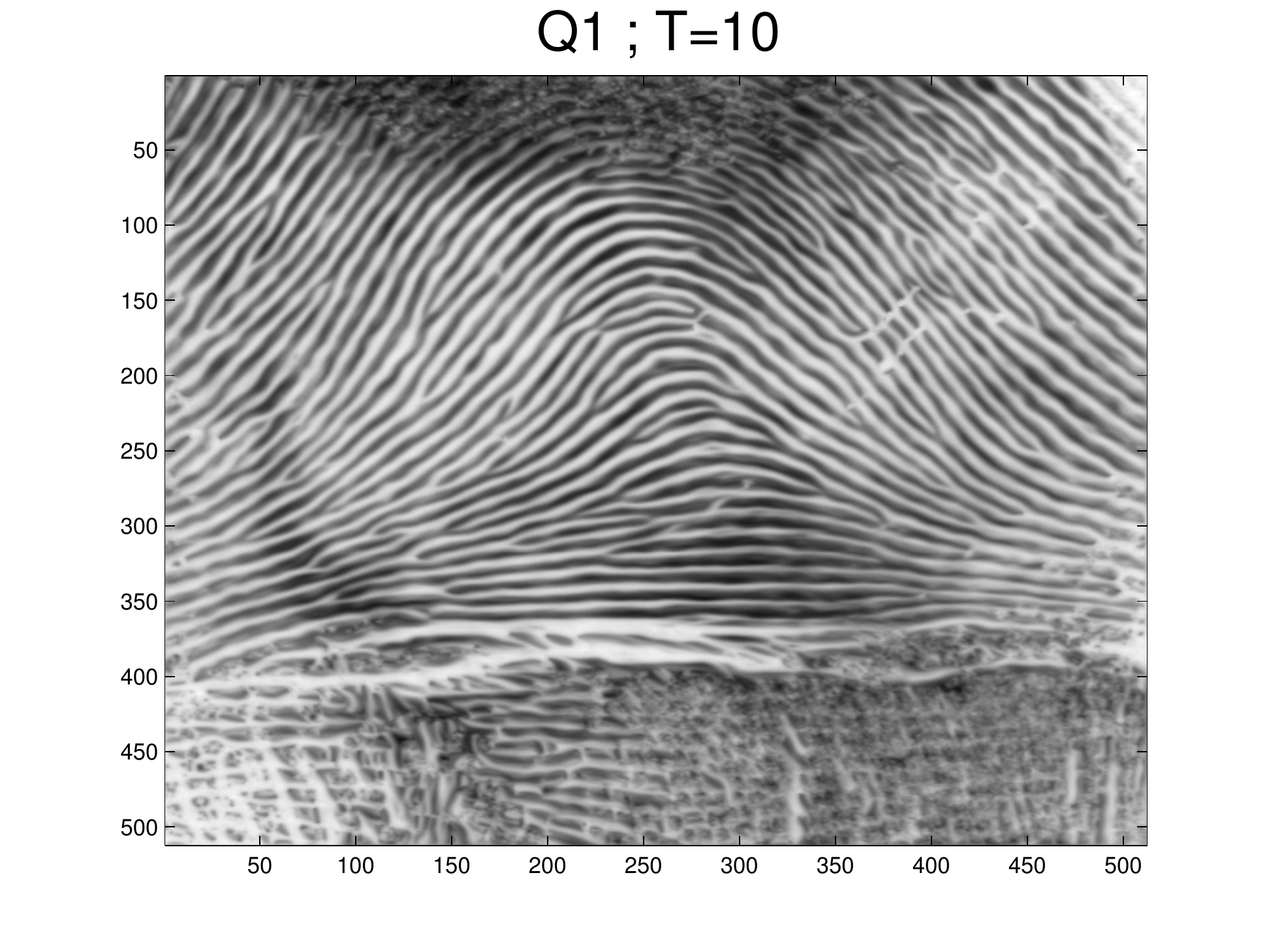}}
\subfigure[][WS]{\includegraphics[width=4.1cm]{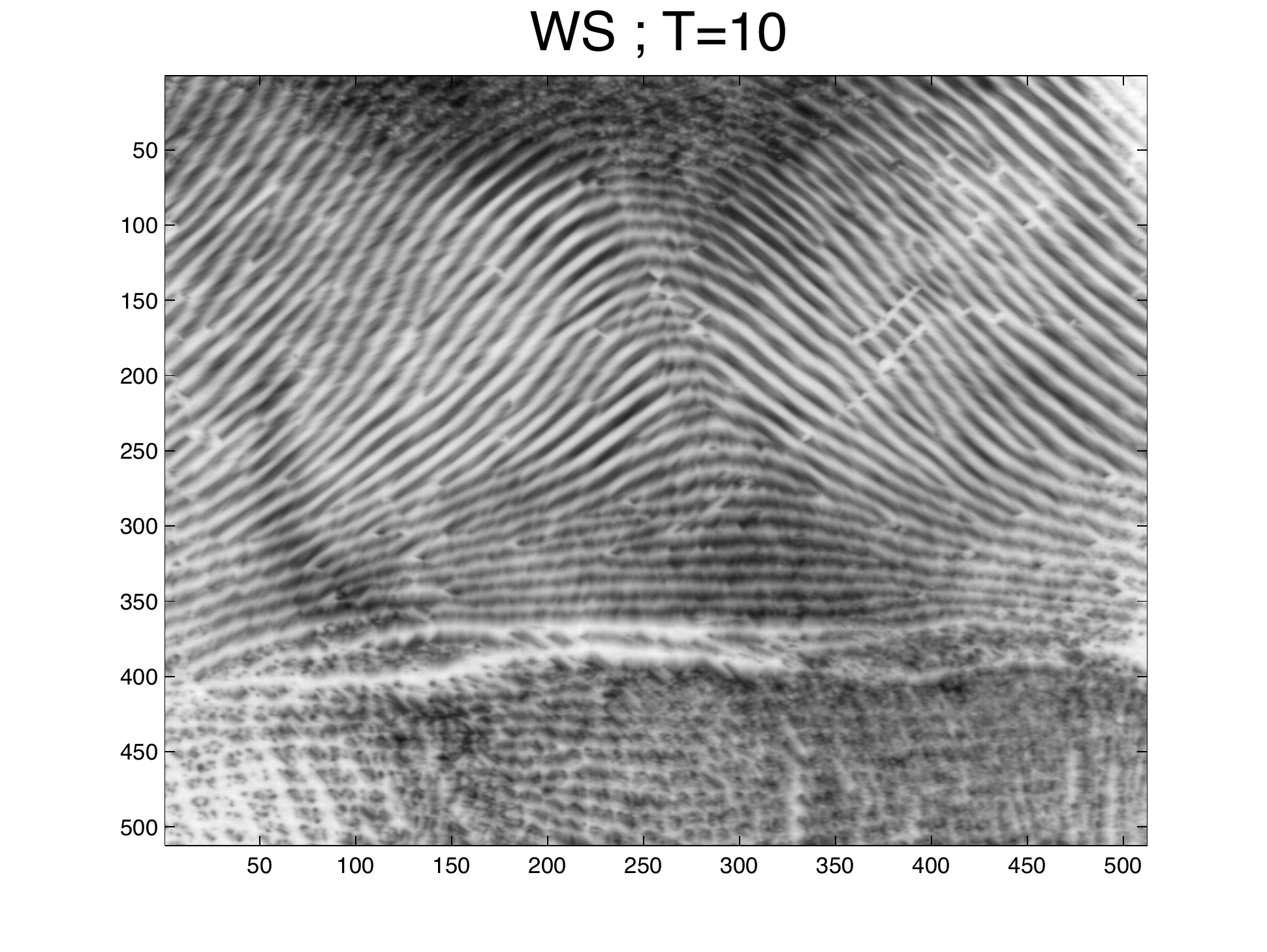}}
\subfigure[][W-NN]{\includegraphics[width=4.1cm]{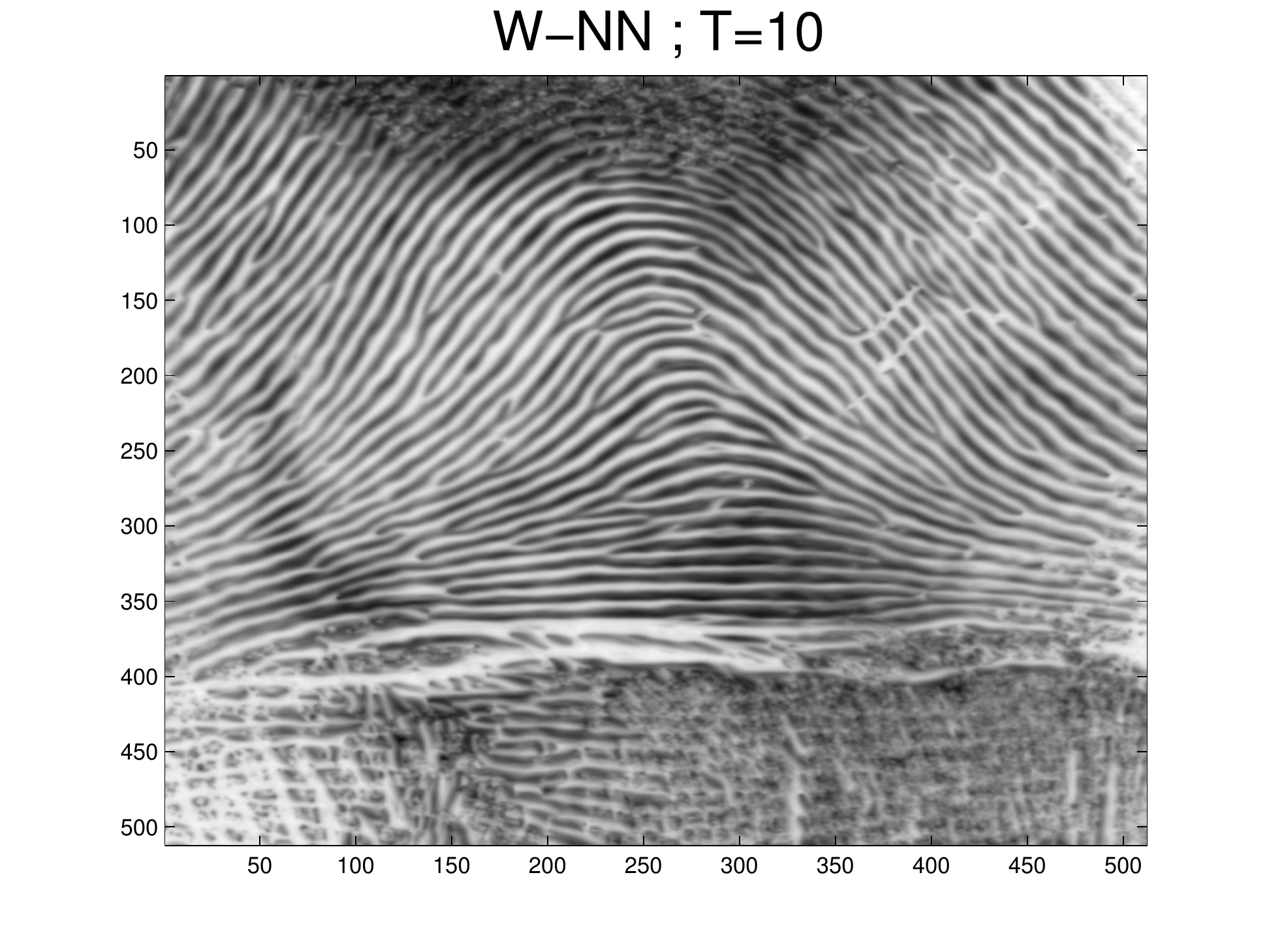}}
\subfigure[][\WG]{\includegraphics[width=4.1cm]{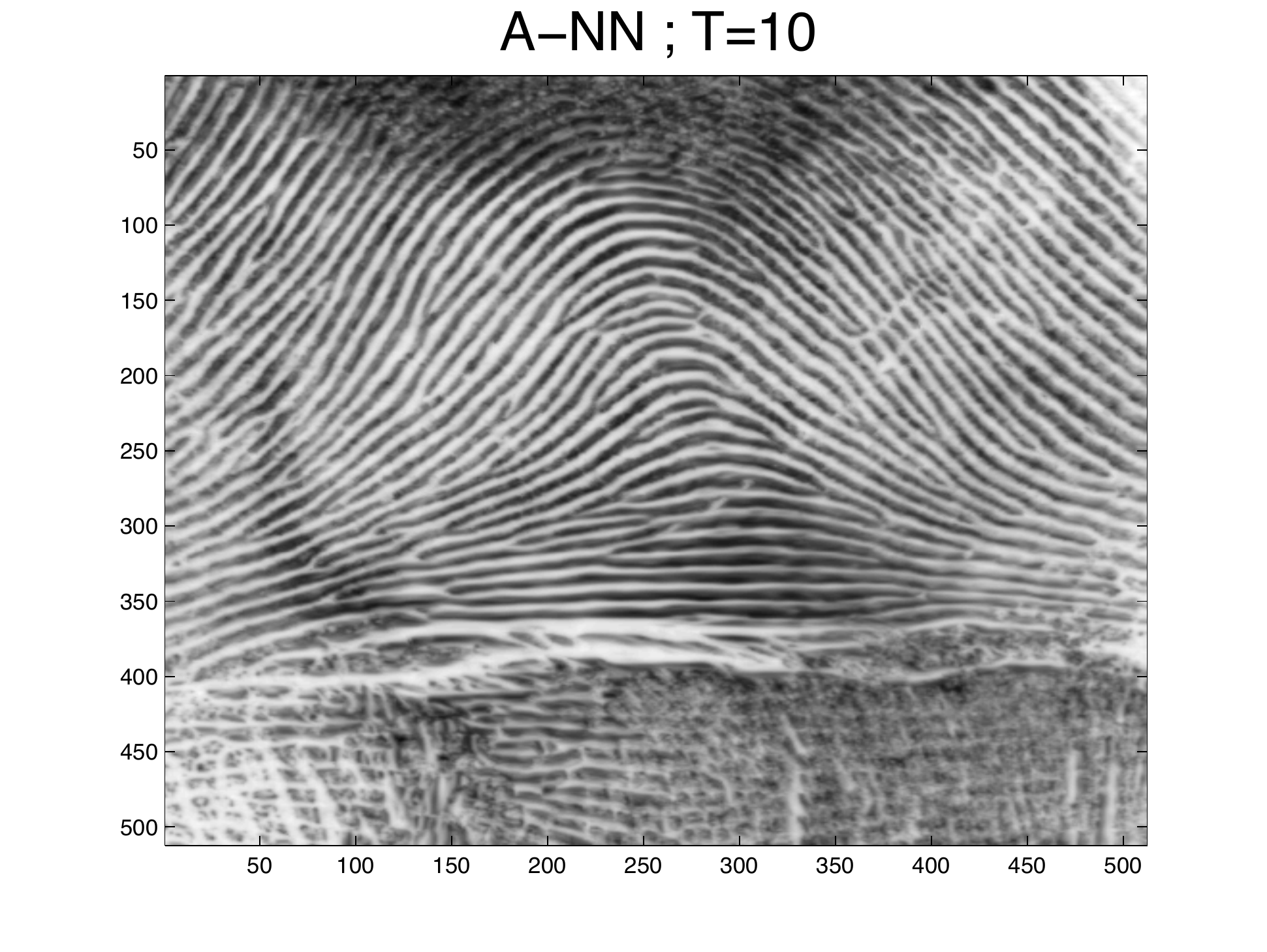}}
\caption{From top to bottom and from left to right: Original image (with two regions highlighted); diffused image using AD-LBR;  FD; Q1; WS; W-NN; \WG. Here $T=10$.}
\label{fig:fingerprint}
\end{centering}
\end{figure}

\begin{figure}[htb]
\begin{centering}
\subfigure[][Original image]{\includegraphics[width=4.1cm]{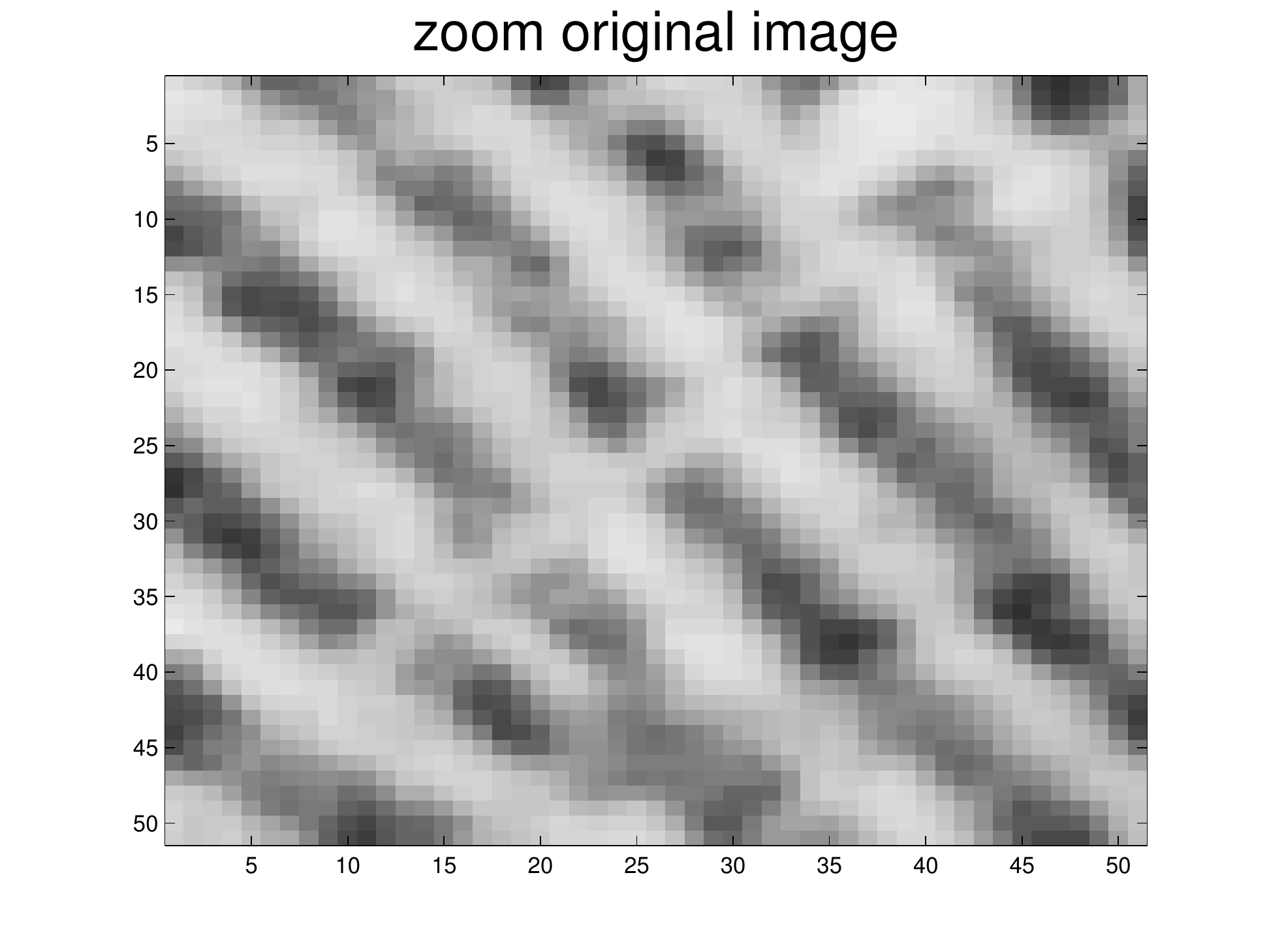}}
\subfigure[][AD-LBR]{\includegraphics[width=4.1cm]{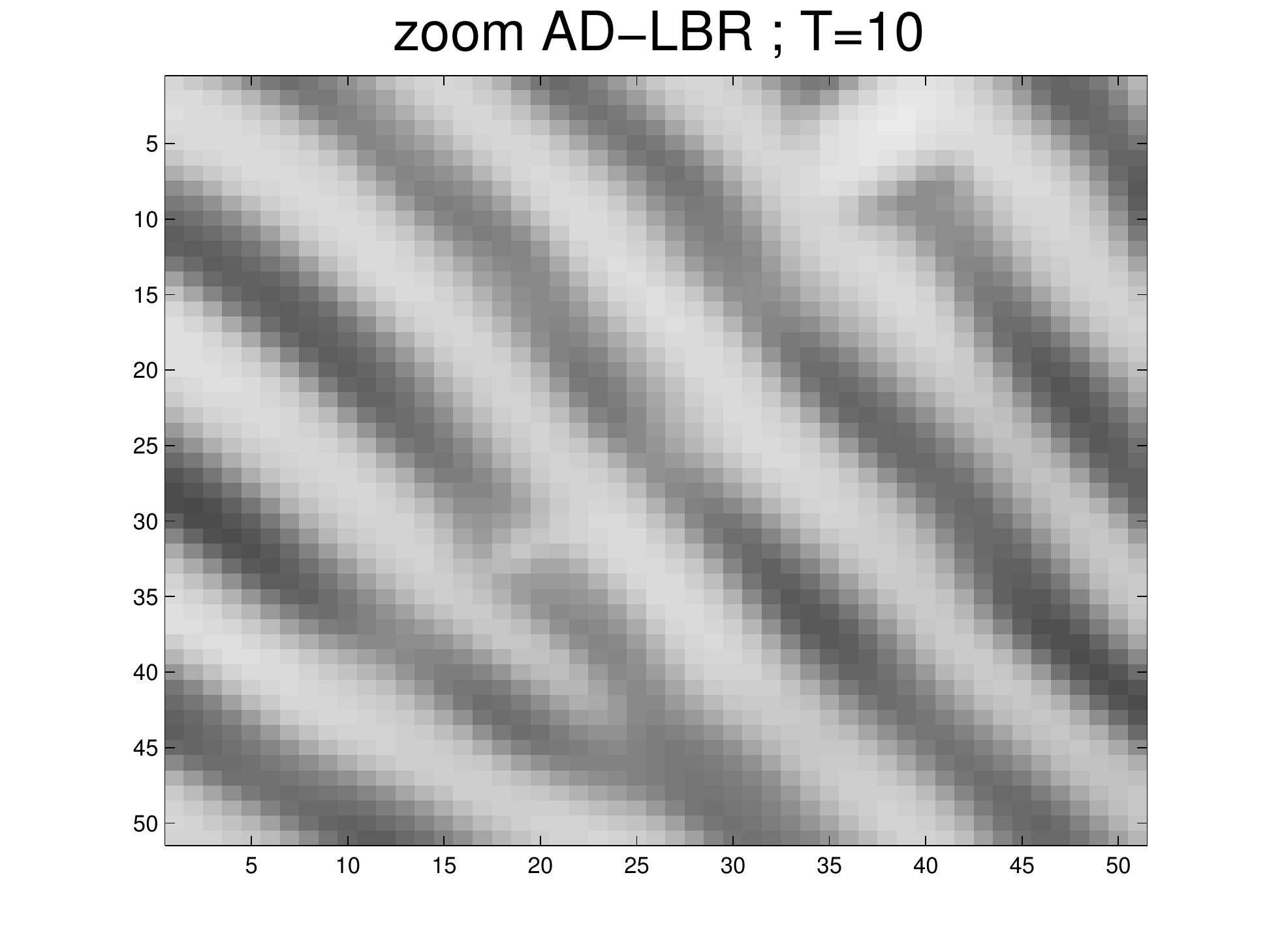}}
\subfigure[][FD]{\includegraphics[width=4.1cm]{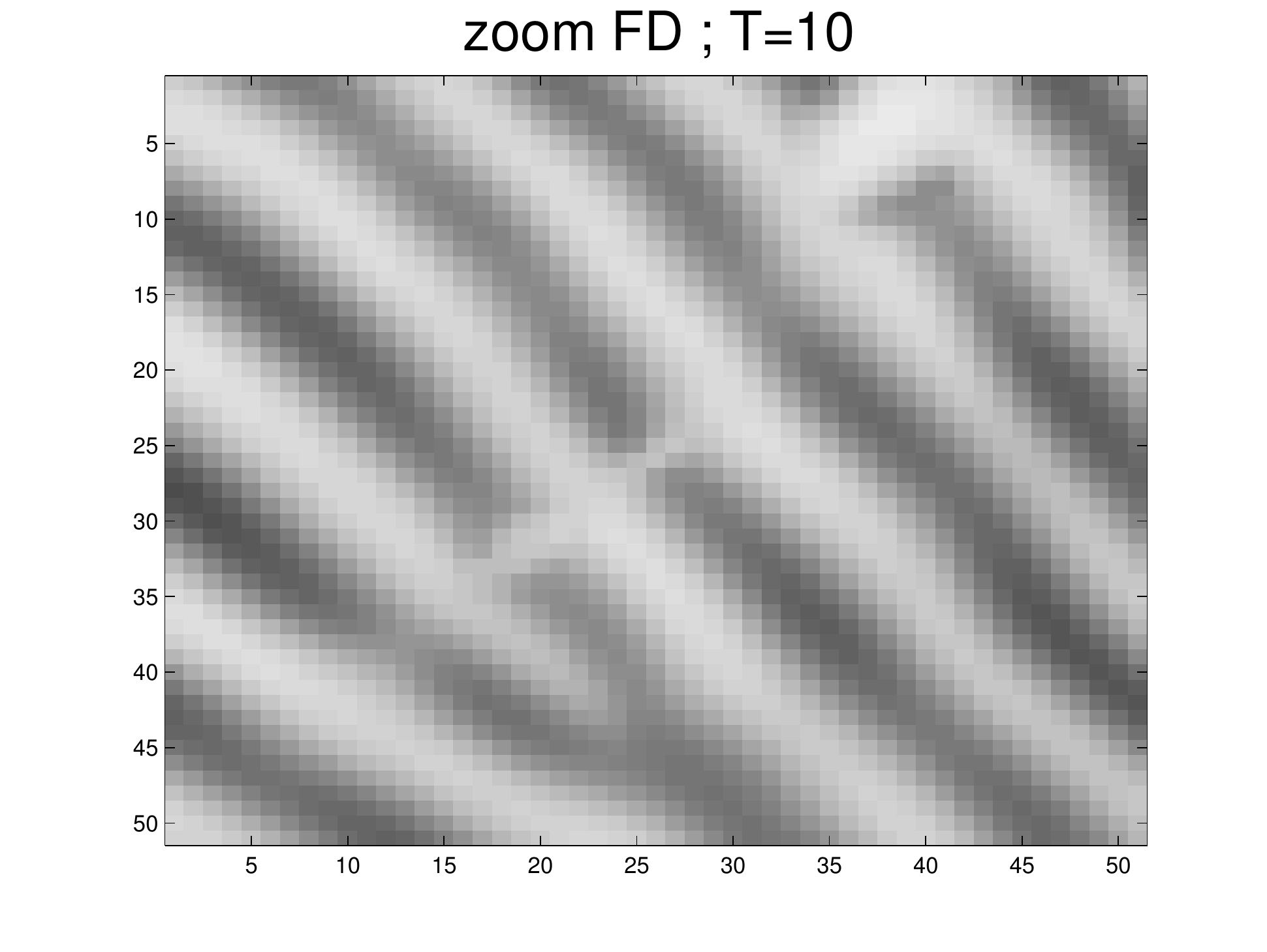}}
\subfigure[][Q1]{\includegraphics[width=4.1cm]{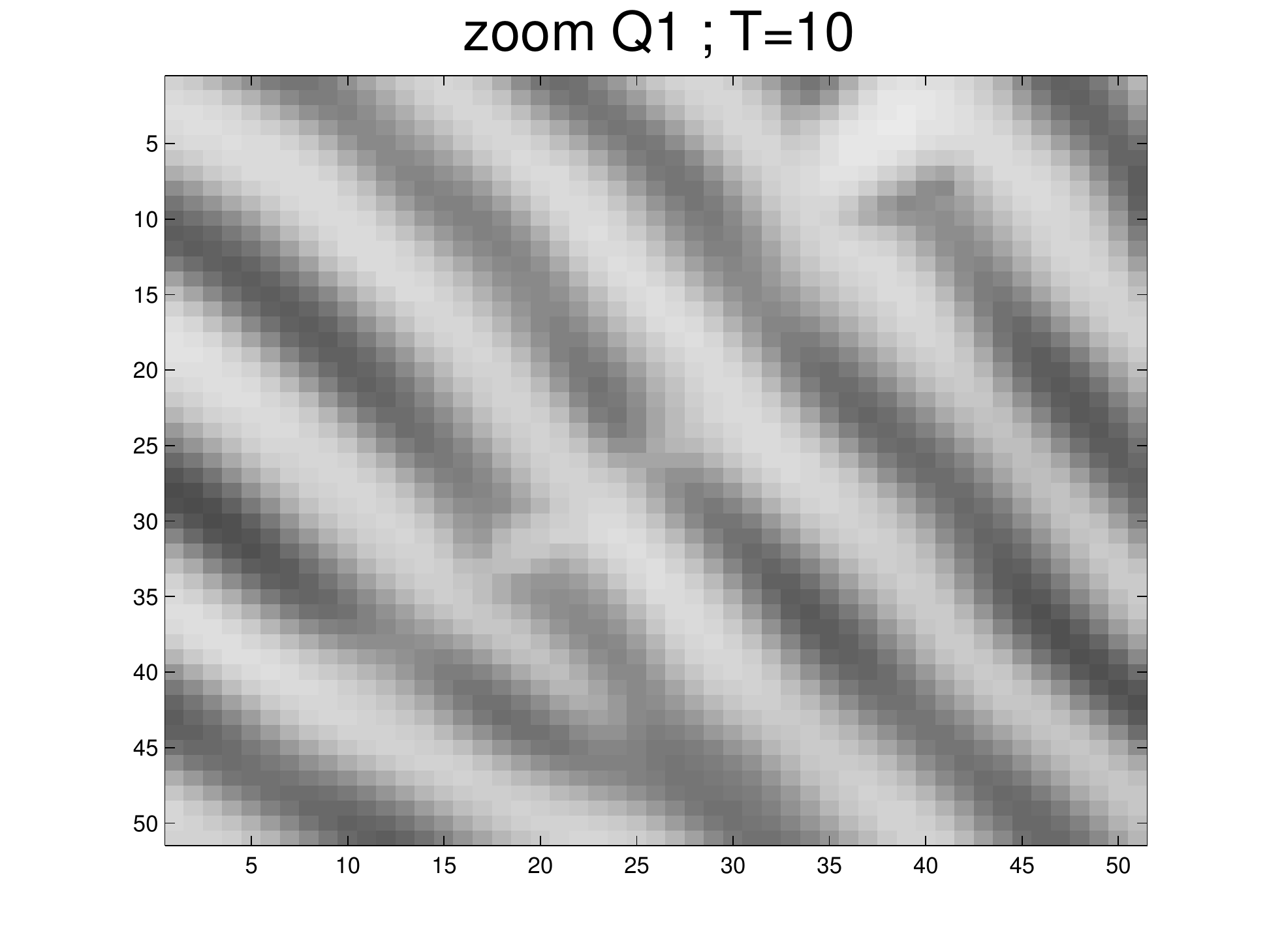}}
\subfigure[][WS]{\includegraphics[width=4.1cm]{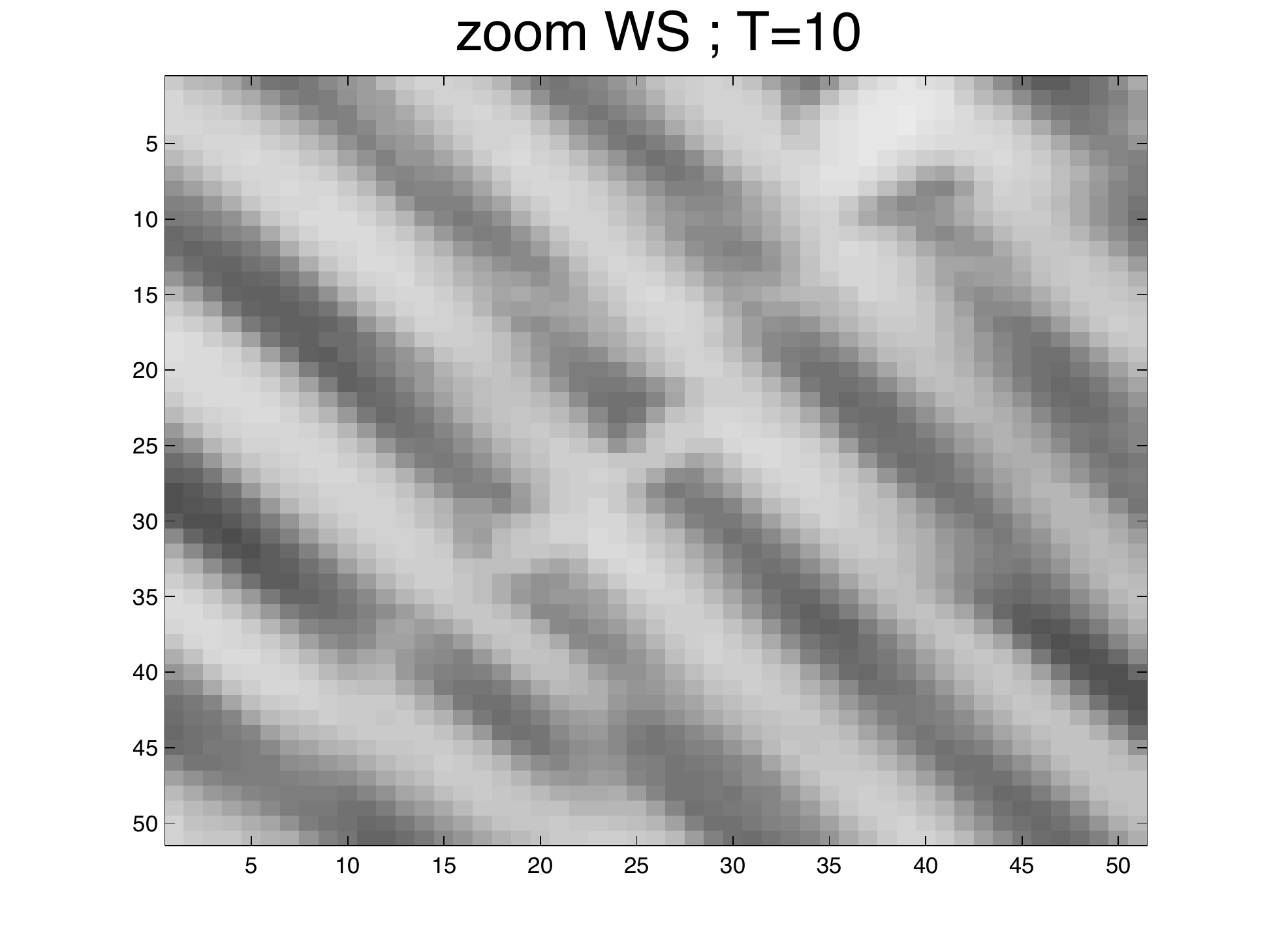}}
\subfigure[][W-NN]{\includegraphics[width=4.1cm]{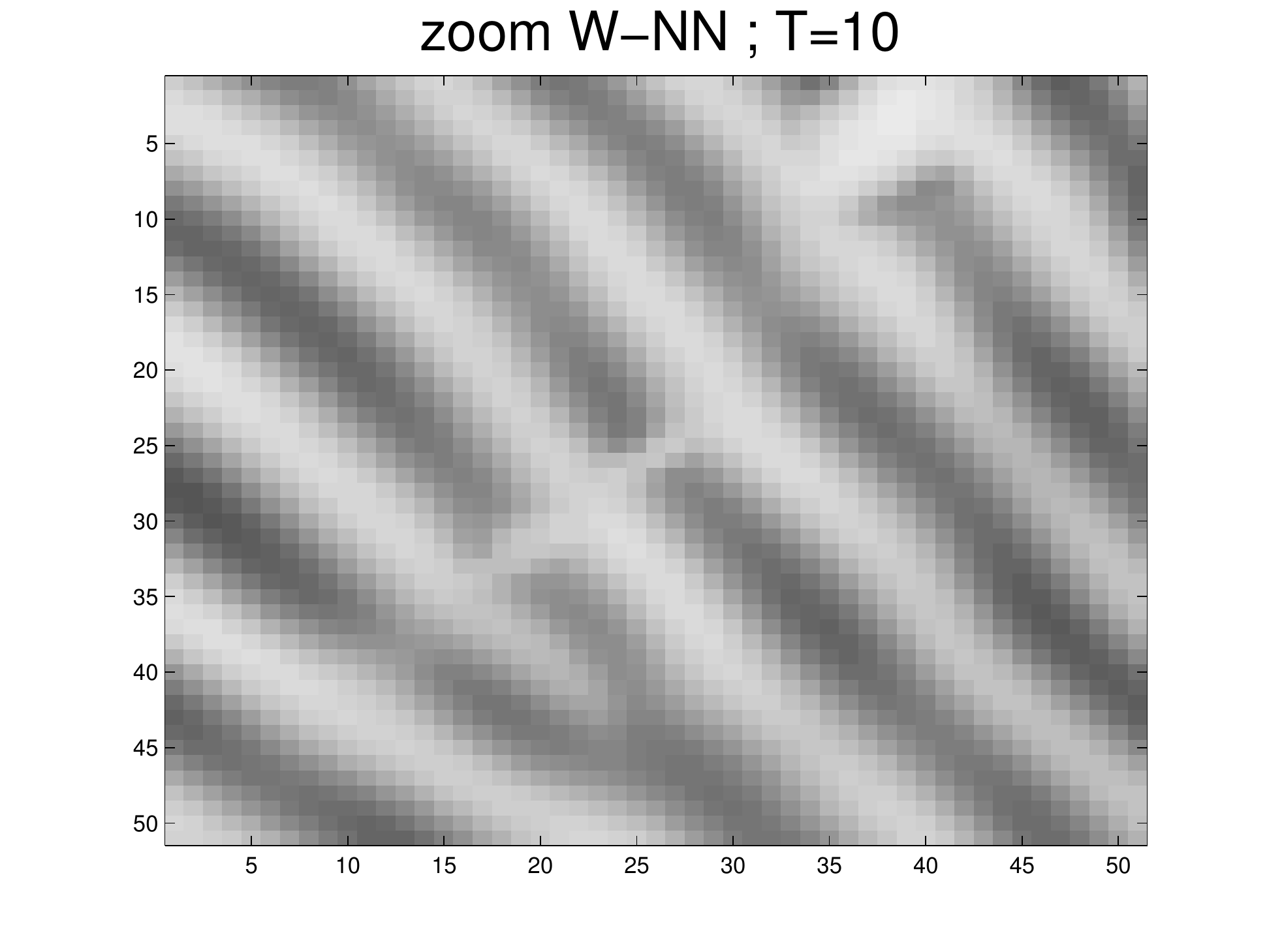}}
\subfigure[][\WG]{\includegraphics[width=4.1cm]{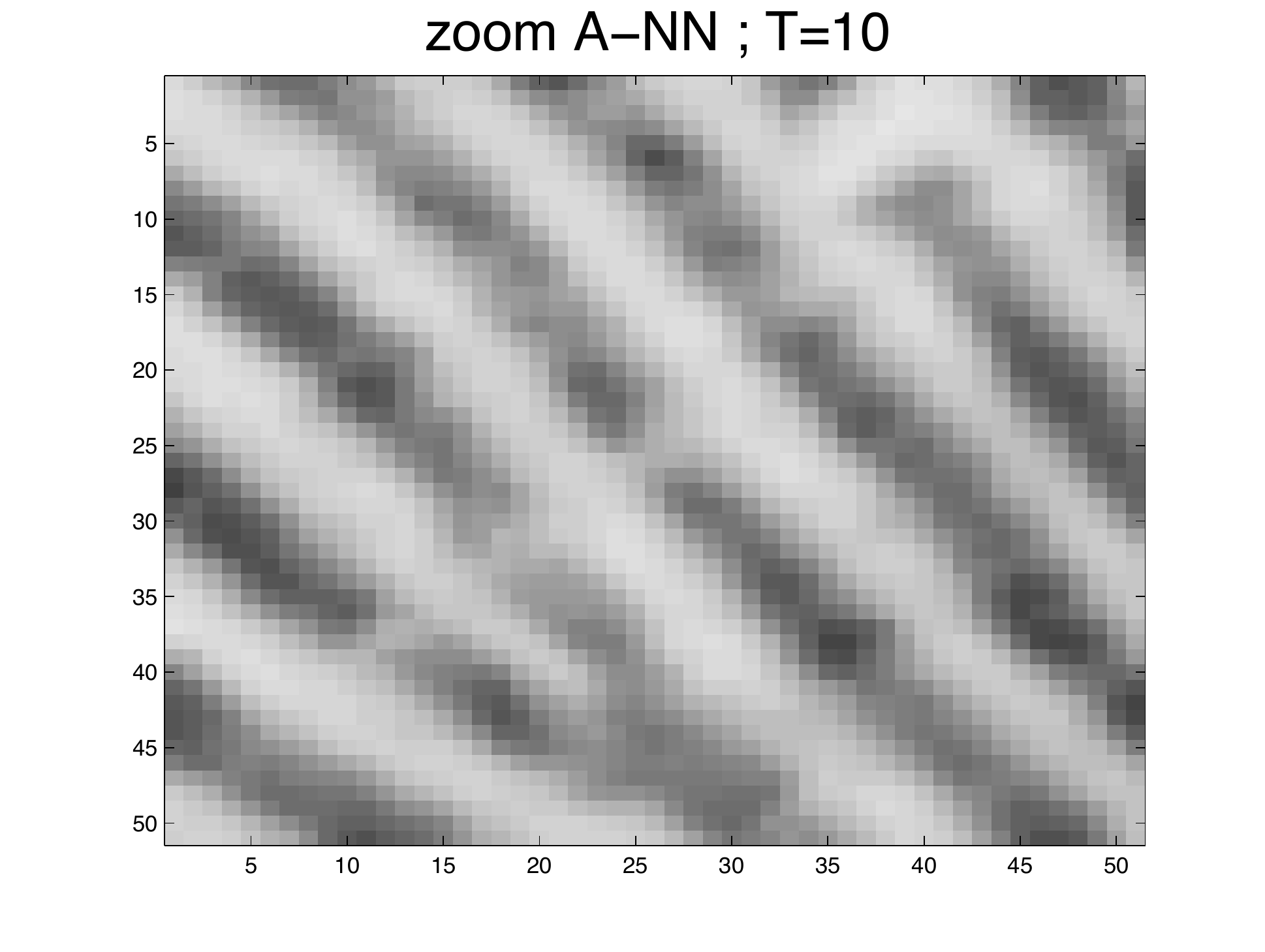}}
\caption{Detail of the region on the right. From top to bottom and from left to right: original image; diffused image using AD-LBR;  FD; Q1; WS; W-NN; \WG.}
\label{fig:fingerprintzoom}
\end{centering}
\end{figure}

\begin{figure}[htb]
\begin{centering}
\includegraphics[width=5.1cm]{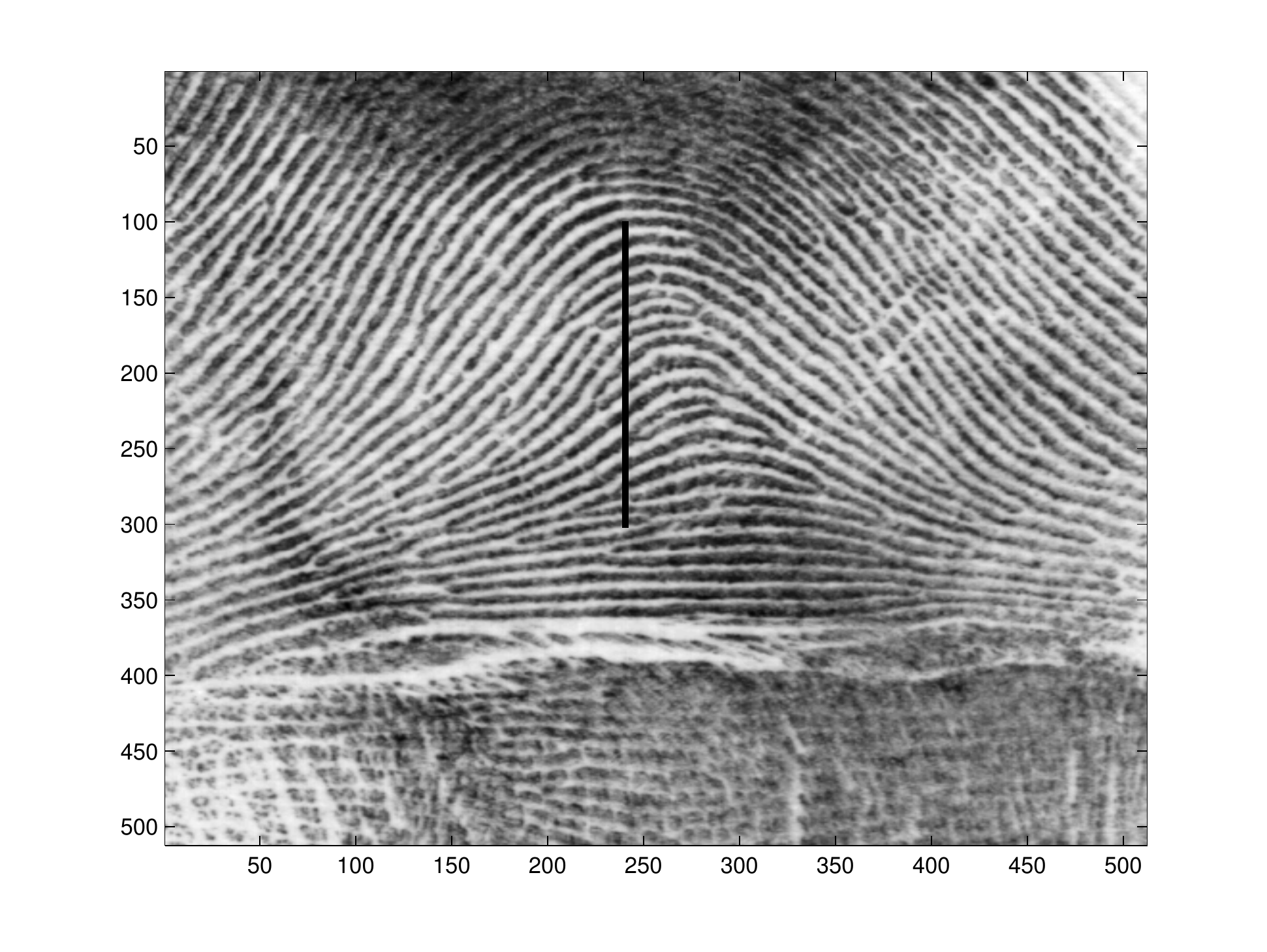}\\
\hbox{ \hspace{-0.8cm} {
\includegraphics[width=9.2cm]{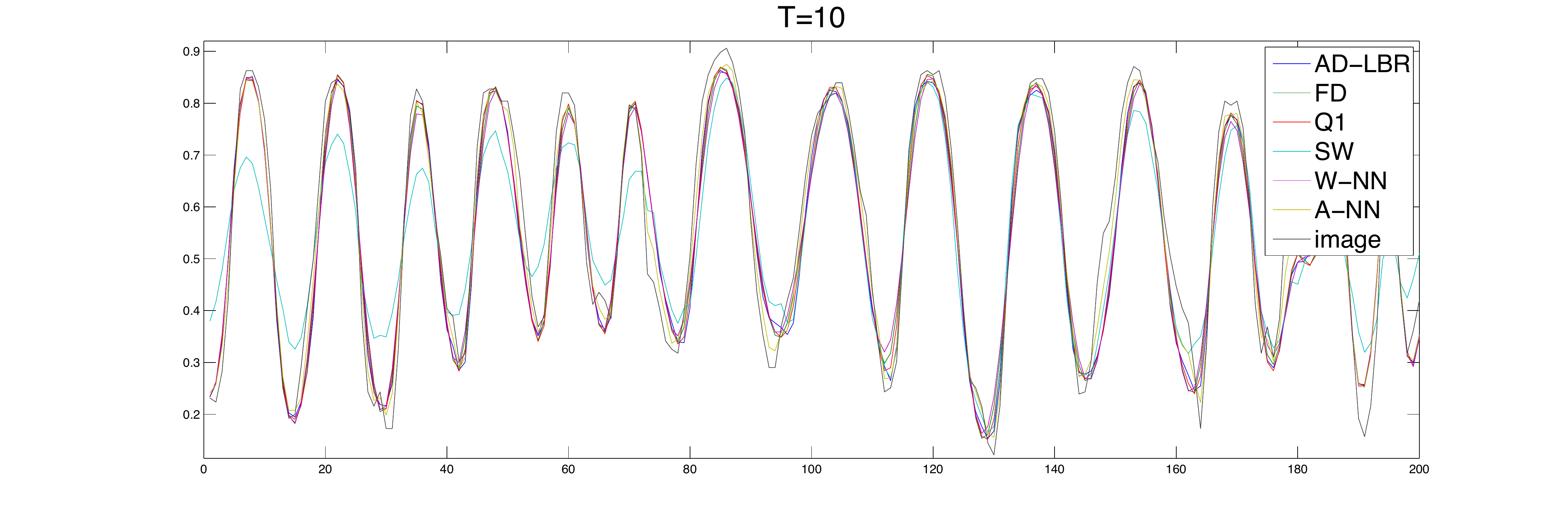}
}}
\hbox{ \hspace{-0.8cm} {
\includegraphics[width=9.2cm]{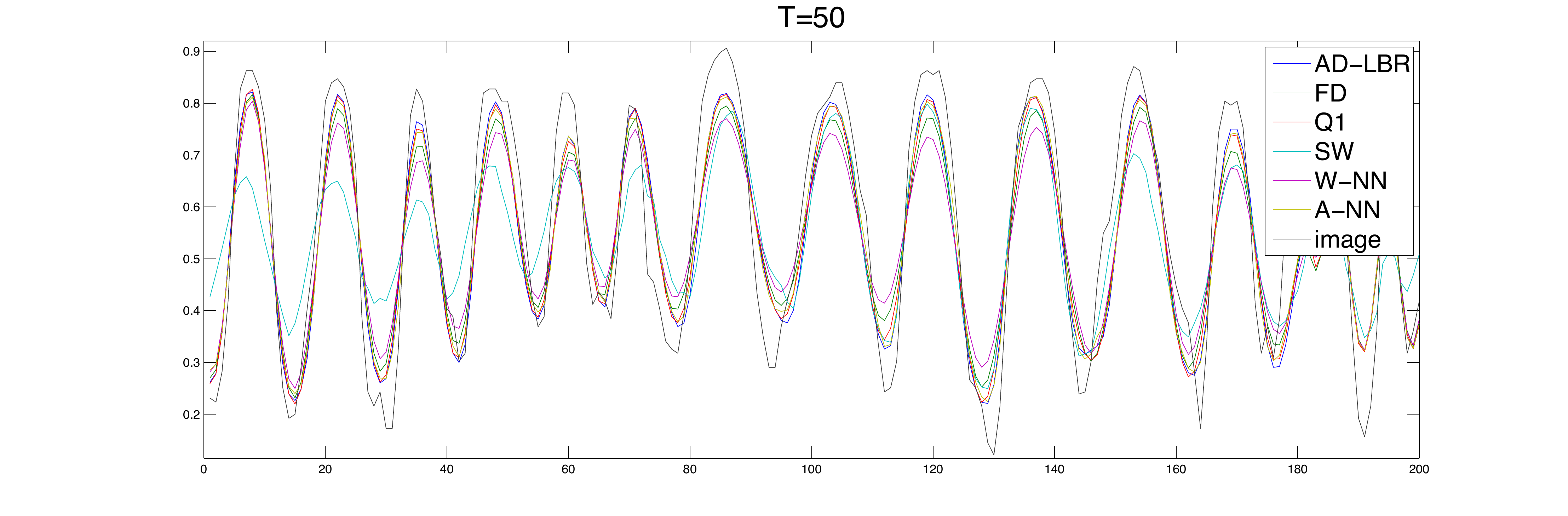}
}}
\hbox{ \hspace{-0.8cm} {
\includegraphics[width=9.2cm]{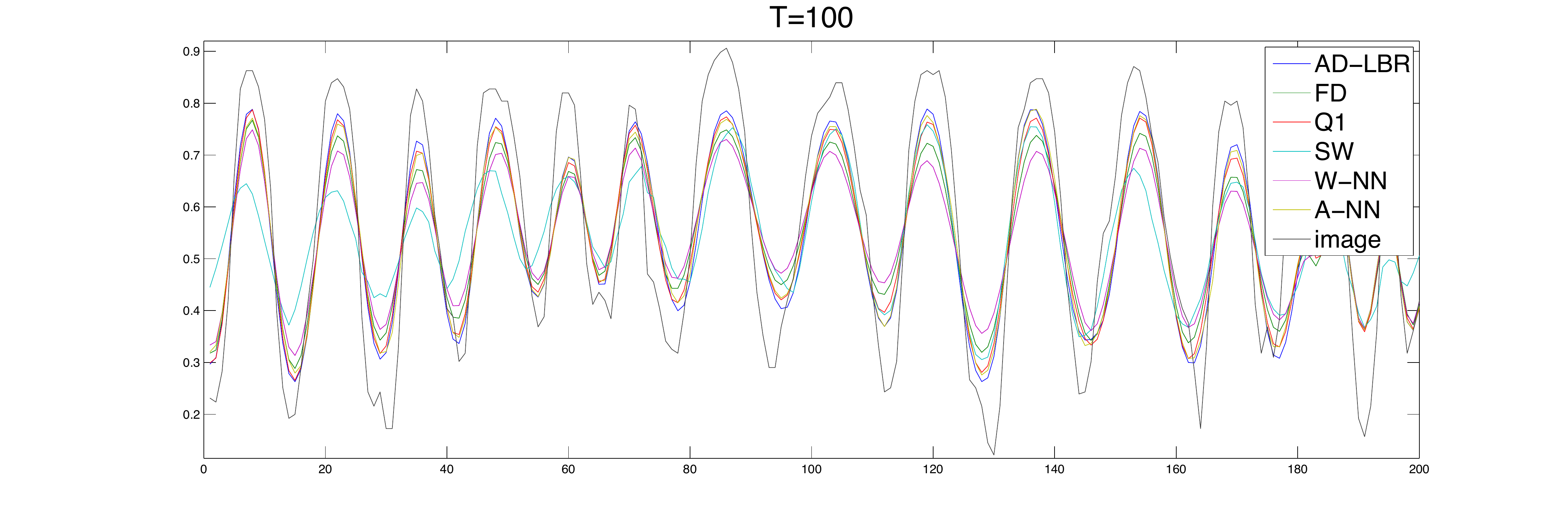}
}}
\caption{
Evolution  under CED of a section of the fingerprint image.
The ridges in the evolved image are higher, and the valleys are deeper, with AD-LBR than with the other schemes. 
This illustrates the fact that AD-LBR, respecting the continuous PDE, diffuses more along the structure and less in the orthogonal direction.
From top to bottom: location of the section of the image; section at $T=10$; section at $T=50$; section at $T=100$.}\label{fig:coupes}
\end{centering}
\end{figure}

\subsection{3-dimensional experiments}
\label{sec:co-en-dif3D}

In order to illustrate the feasibility of our scheme in 3D space, we present the action of anisotropic diffusion PDEs on two examples. The first example is a 3D analog of the synthetic test case presented in \cite{WS02}, featuring Coherence-Enhancing Diffusion. The second one is the application of Edge-Enhancing Diffusion to a MRI scan.


\paragraph{Synthetic example\\}

The original, radially varying image is defined on the cube $[0,1]^3$. 
The gray-level at a point $x$ is defined by  
\begin{equation*}
u^{0}(x)=\cos\left(2(r/R)^3\right),
\end{equation*}
where $r:=|x|$ and $R:=1/2$. 
This image presents a series of concentric level-sets. We present in Figure \ref{fig:3Doriginal} the level sets $\{u^0 = 0\}$, and a slice through the plane $z=0.7$.

The image $u^0$ is perturbed by
\begin{equation*}
u:=u^{0}+n,
\end{equation*}
where $n$ is an additive Gaussian noise of variance $\sigma=0.5$. 
The reconstructed image is obtained using a 3D Coherence-Enhancing Diffusion PDE \cite{W98}, similar to the 2D one in section \ref{sec:co-en-dif}:
\begin{equation*} 
\partial_t u=\diver(\DD(J_\rho(\nabla u_\sigma))\nabla u),
\end{equation*}
where $J_\rho$ is the structure tensor defined by $J_\rho :=K_\rho\star(\nabla u_\sigma \nabla u_\sigma^T)$, $u_\sigma :=K_\sigma\star u$. The tensor $\DD(J_\rho)$ possesses the same eigenvectors $(v_1,v_2,v_3)$ as $J_\rho$, and if the eigenvalues of $J_\rho$ are $\mu_1\ge \mu_2\ge \mu_3$ then the eigenvalues of $\DD(J_\rho)$ are 
\begin{equation*}
\lambda_1 :=\alpha
\end{equation*}
\begin{equation*}
\lambda_2 :=\alpha+(1-\alpha)\exp\left(\dfrac{-C}{(\mu_1-\mu_2)^2}\right),
\end{equation*}
\begin{equation*}
\lambda_3 :=\alpha+(1-\alpha)\exp\left(\dfrac{-C}{(\mu_1-\mu_3)^2}\right),
\end{equation*}

where $\alpha=10^{-2}$. The anisotropy ratio is bounded by $\kappa = 1/\sqrt{\alpha} = 10$.
We used the values $\sigma=0.5$, $\rho=4$. 
The problem is discretized using $100^3$ voxels.
We present in Figure \ref{fig:3Doriginal} the noisy image $u$ (levelset 0 and planar slice) and  the result after 20 time-steps of $\Delta t=10^{-3}$.

\begin{figure*}
\begin{centering}
\includegraphics[width=6.3cm]{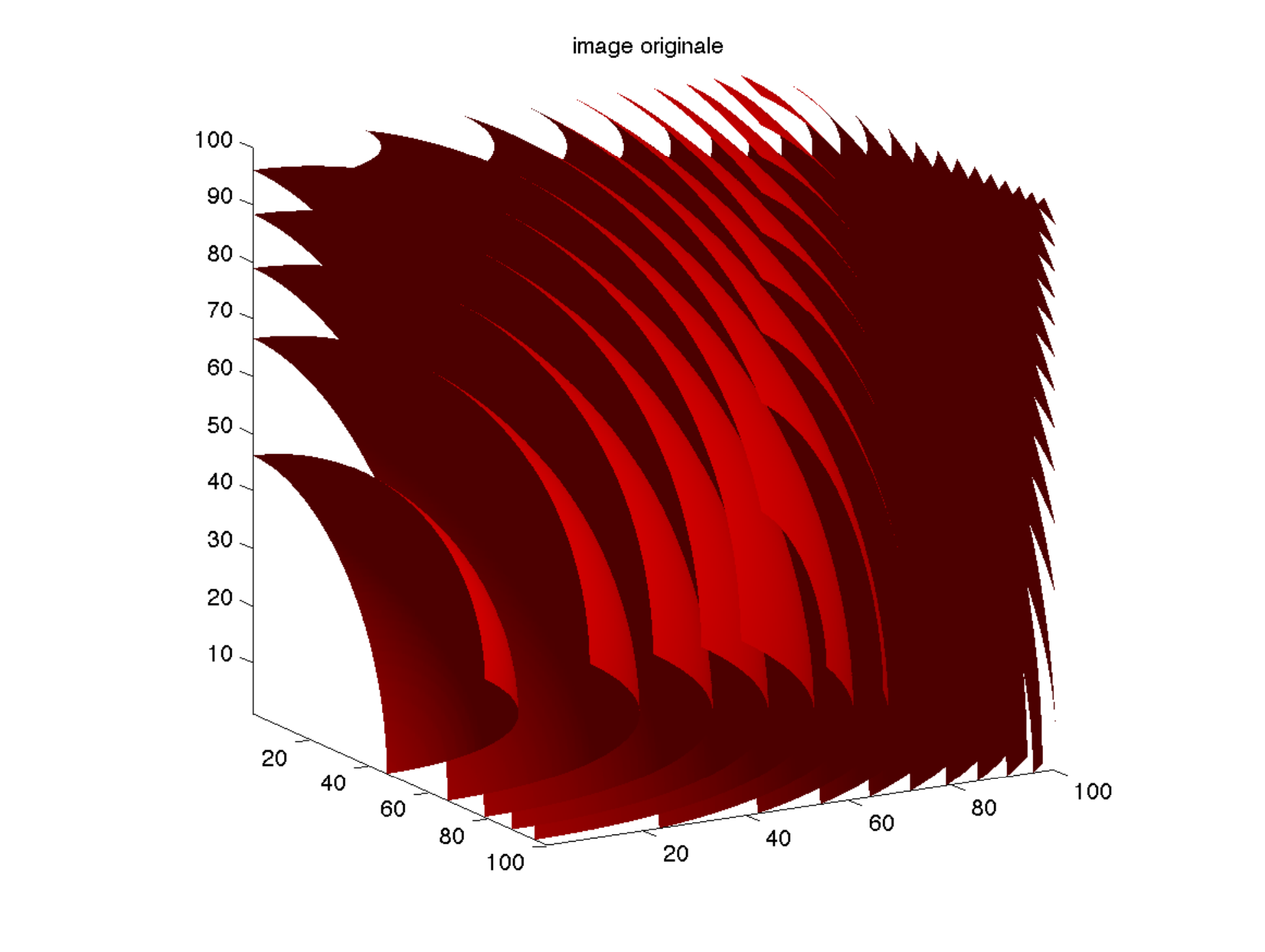}
\hspace{-1.cm}
\includegraphics[width=6.3cm]{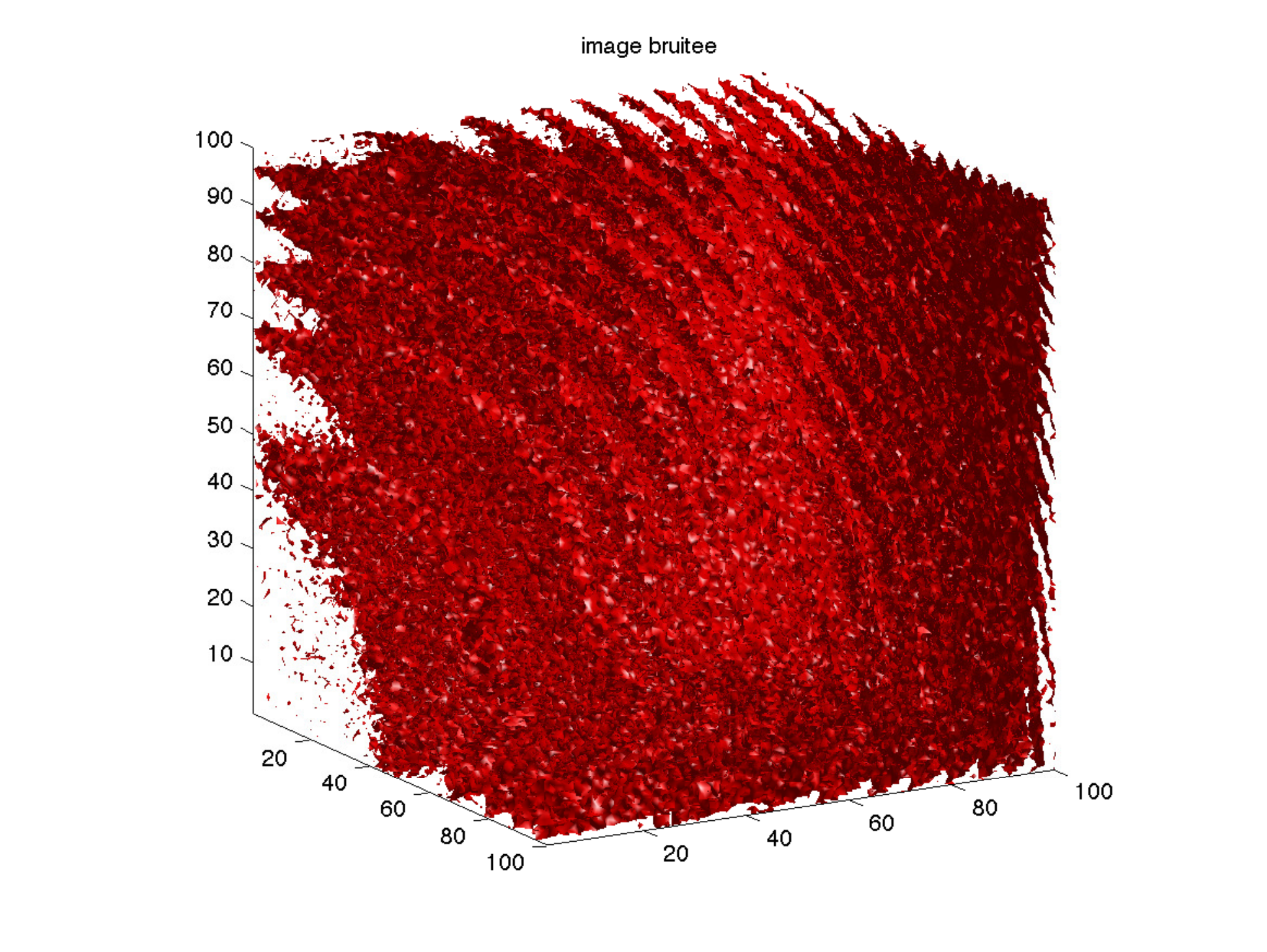}
\hspace{-1.cm}
\includegraphics[width=6.3cm]{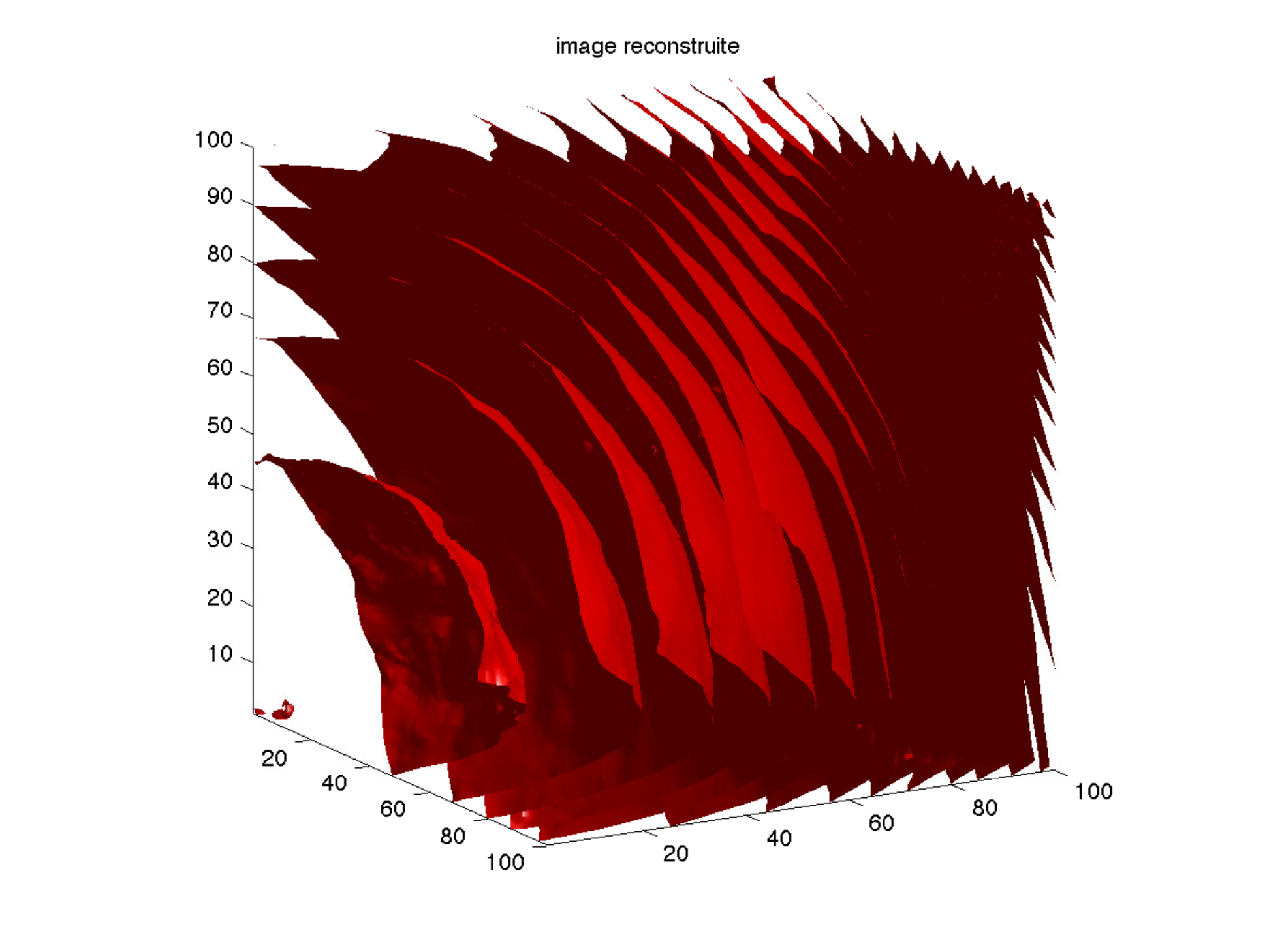}
\\
\includegraphics[width=5.1cm]{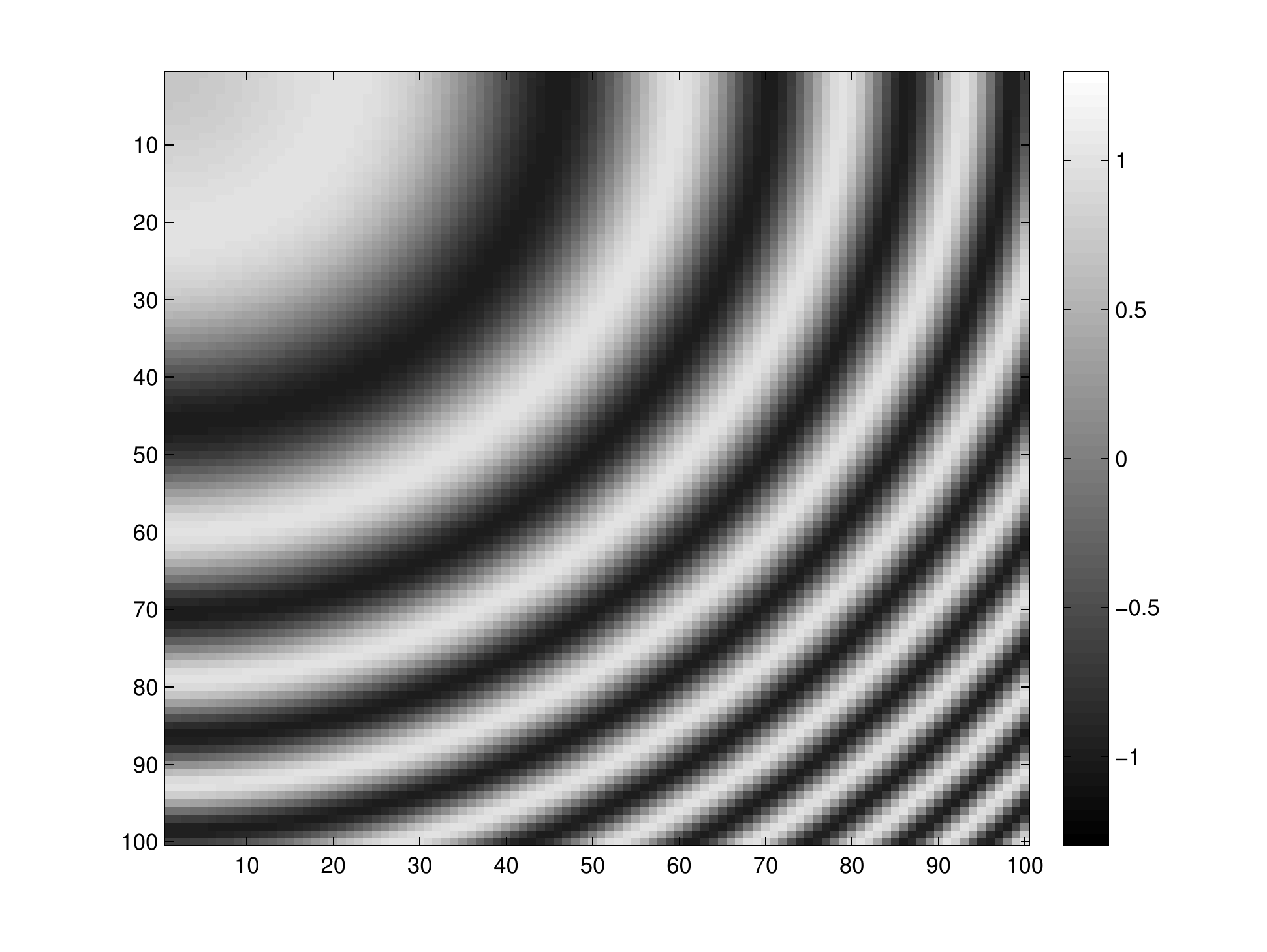}
\includegraphics[width=5.1cm]{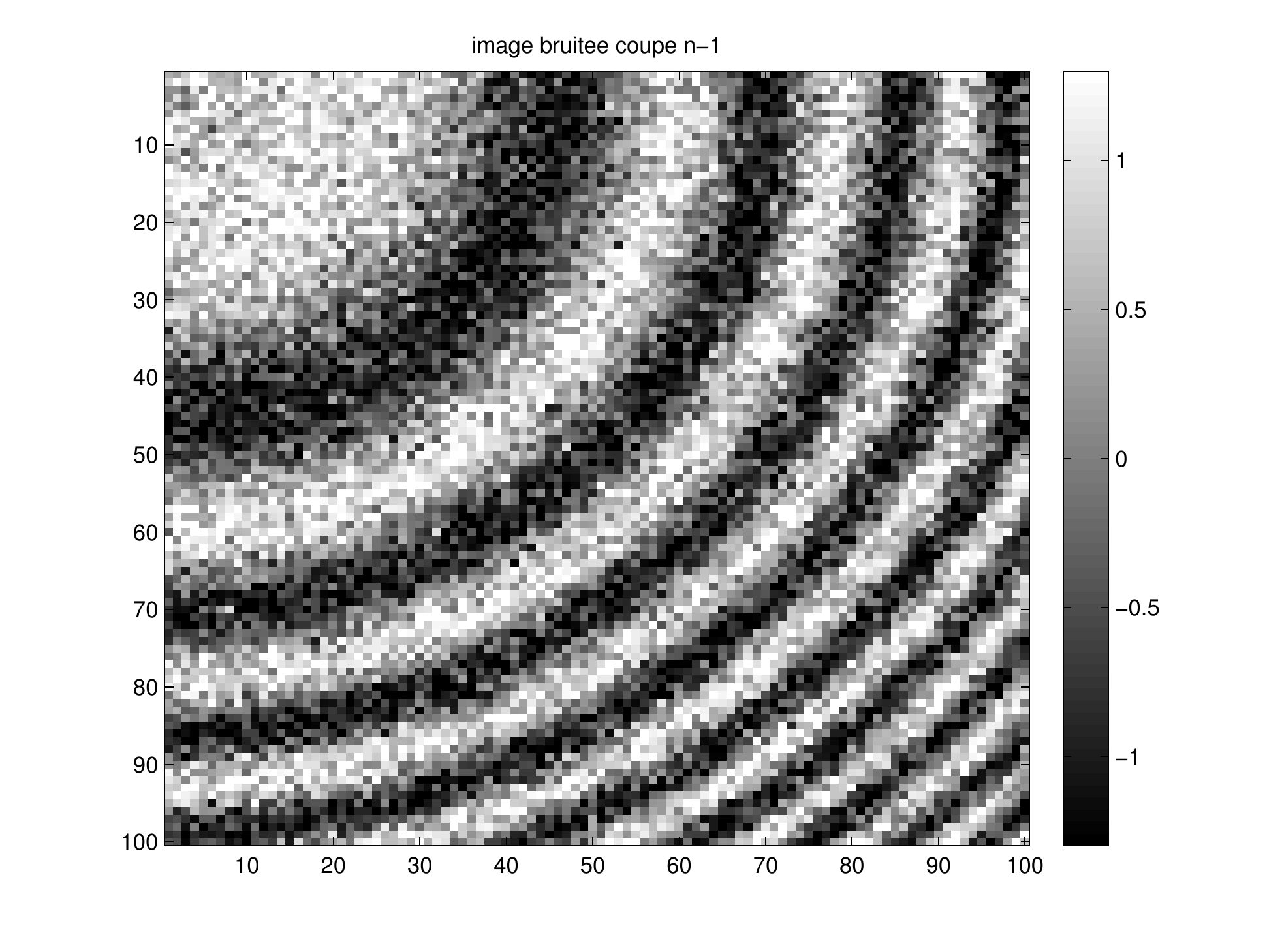}
\includegraphics[width=5.1cm]{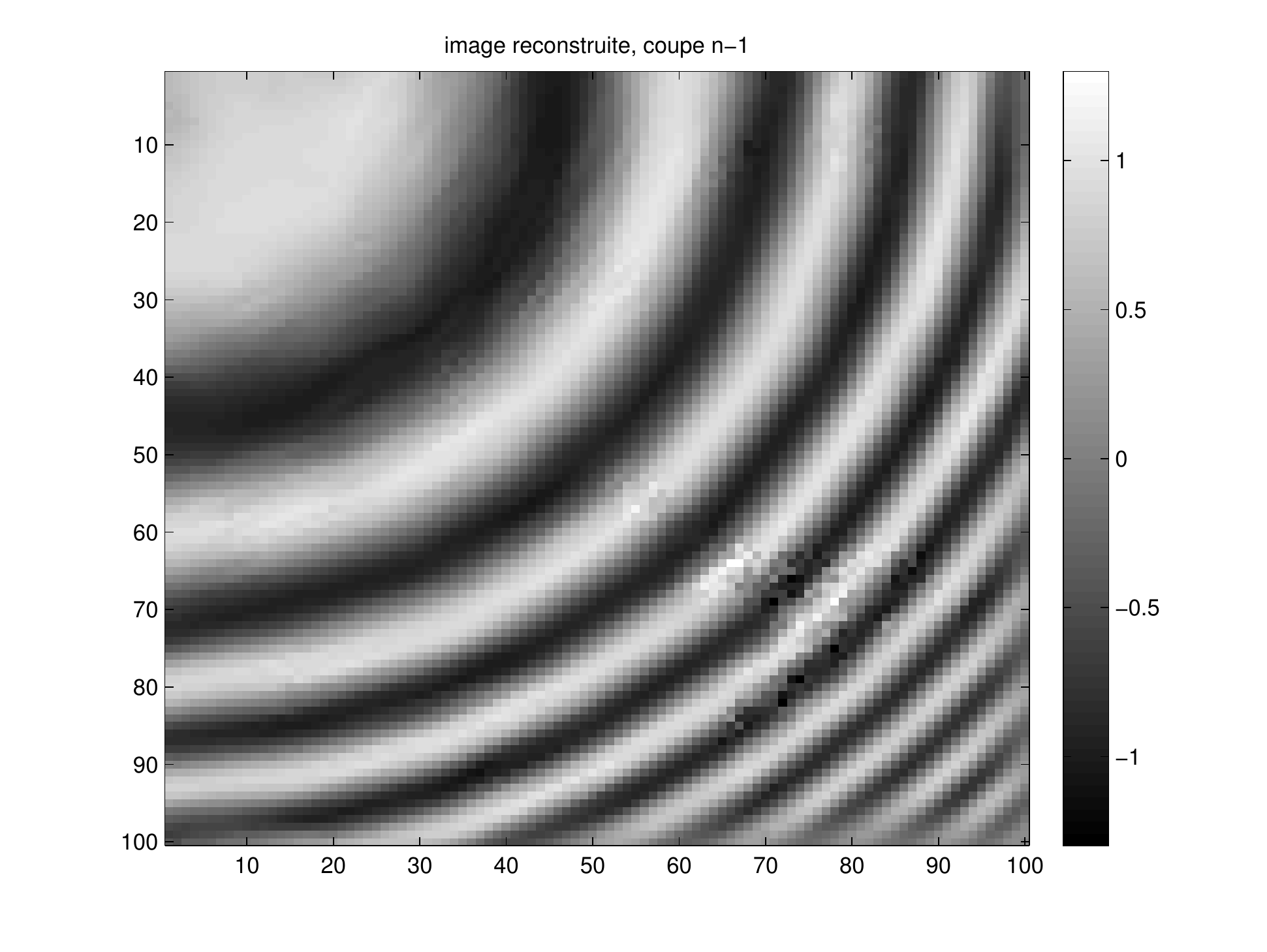}
\caption{
Levelset 0 (top) and slice (bottom) of a 3D image. Original (left), noisy (center), and reconstructed (right) images.
Slice in the plane $z=0.7$, with values clipped to the range $[-1.3,1.3]$.
}
\label{fig:3Doriginal}
\end{centering}
\end{figure*}

\paragraph{3D MRI data\\}

The data is a $256\times256\times100$ Magnetic Resonance Imaging scan of a skull, and was obtained from the "University of North Carolina Volume Rendering Test Data Set" archive.

The reconstructed image is obtained using a 3D Edge-Enhancing Diffusion PDE \cite{W98}, which differs from the above Coherence-Enhancing Diffusion one by the choice of the diffusion tensor eigenvalues. The optimal choice of these eigenvalues indeed depends on the application, and is still an active subject of research \cite{MVRV09}.
With the above notations, the eigenvalues of $\DD(J_\rho)$ are 
\begin{align*}
\lambda_1 & :=1-\exp\left(\dfrac{-C}{\mu_1^2}\right)\\
\lambda_2 & :=1-\exp\left(\dfrac{-C}{\mu_2^2}\right),\\
\lambda_3 & :=1.
\end{align*}
We used the values $\sigma=0.5$, $\rho=4$. In our experiment, the maximum anisotropy ratio was $\kappa = 11.2$.
We present in Figure \ref{fig:3Dmri} the original image and two slices of the result after 10 time-steps of $\Delta t=10^{-4}$.

\begin{figure*}
\begin{center}
\includegraphics[width=4.2cm]{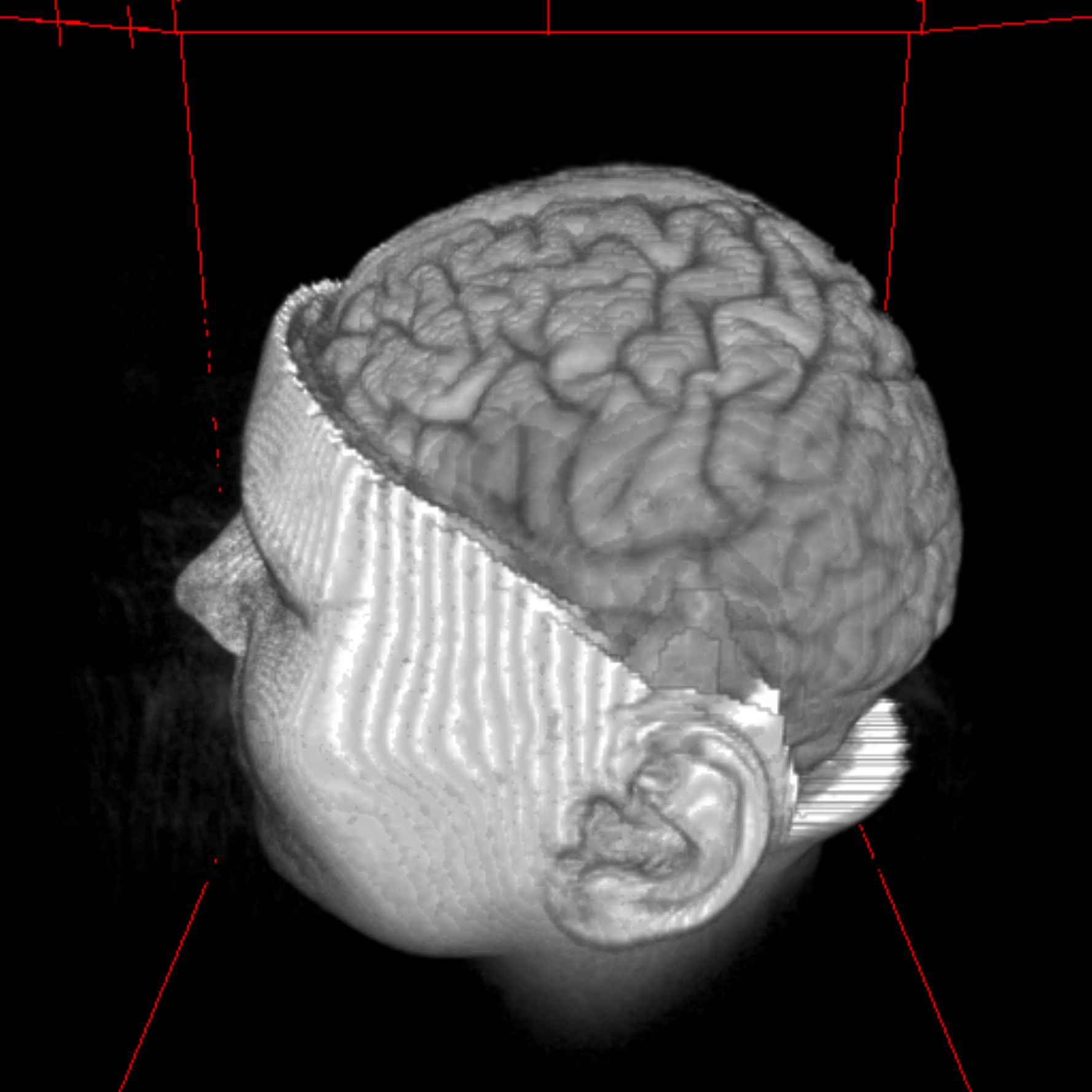}
\hspace{-0.cm}
\includegraphics[width=6.5cm]{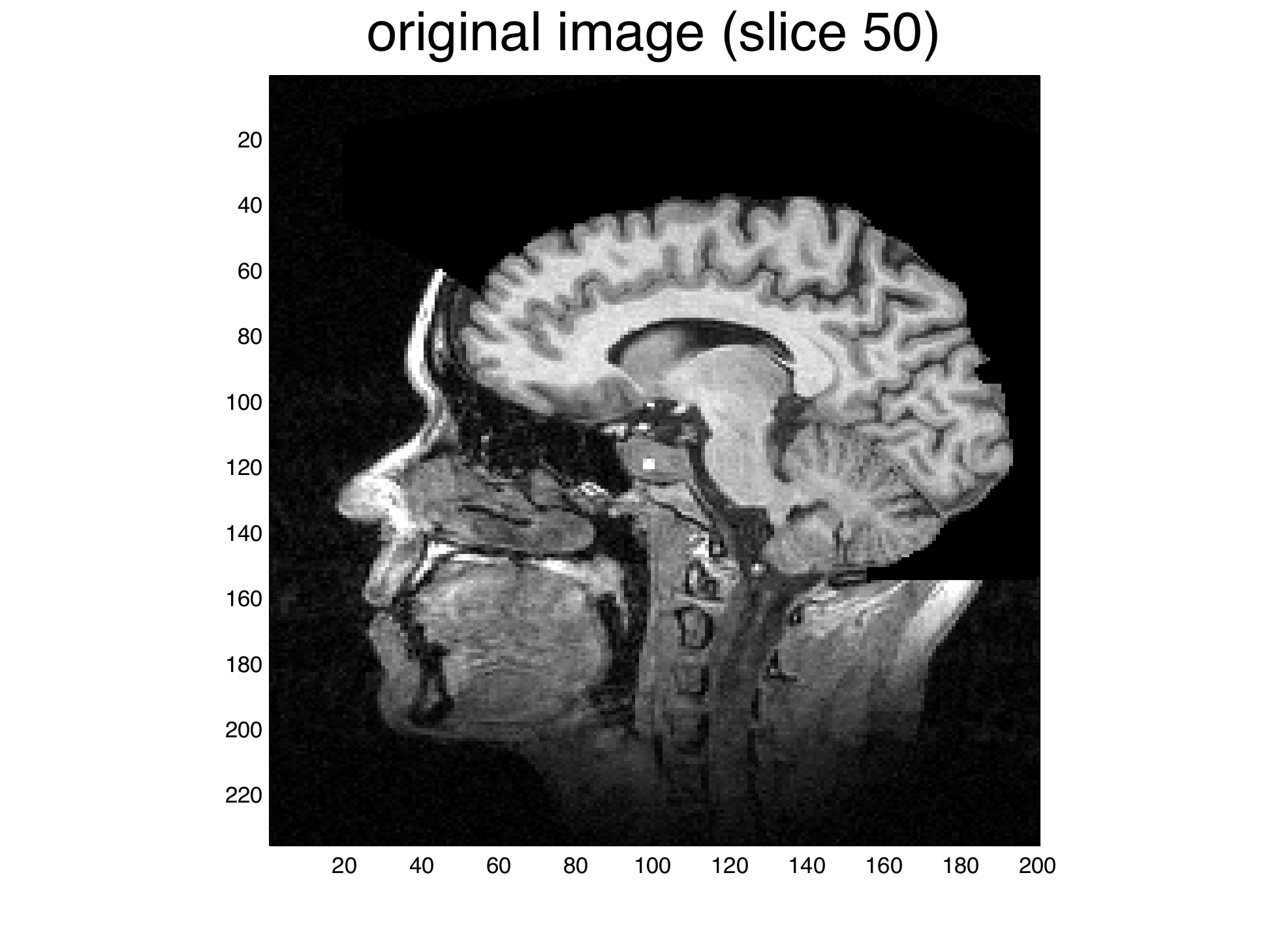}
\hspace{-1.3cm}
\includegraphics[width=6.5cm]{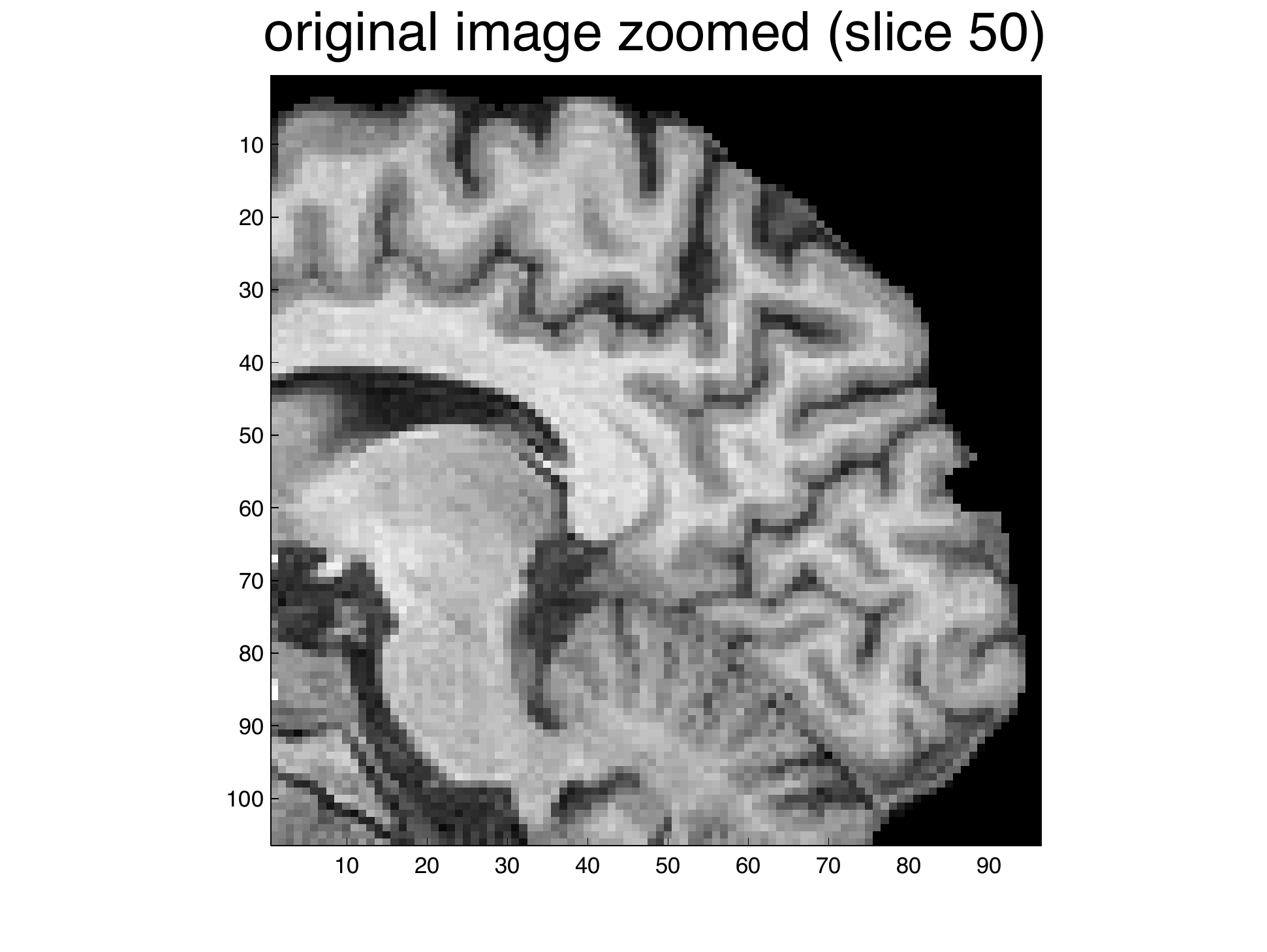}
\hspace{-1cm}
\end{center}

\begin{center}
\includegraphics[width=4.2cm]{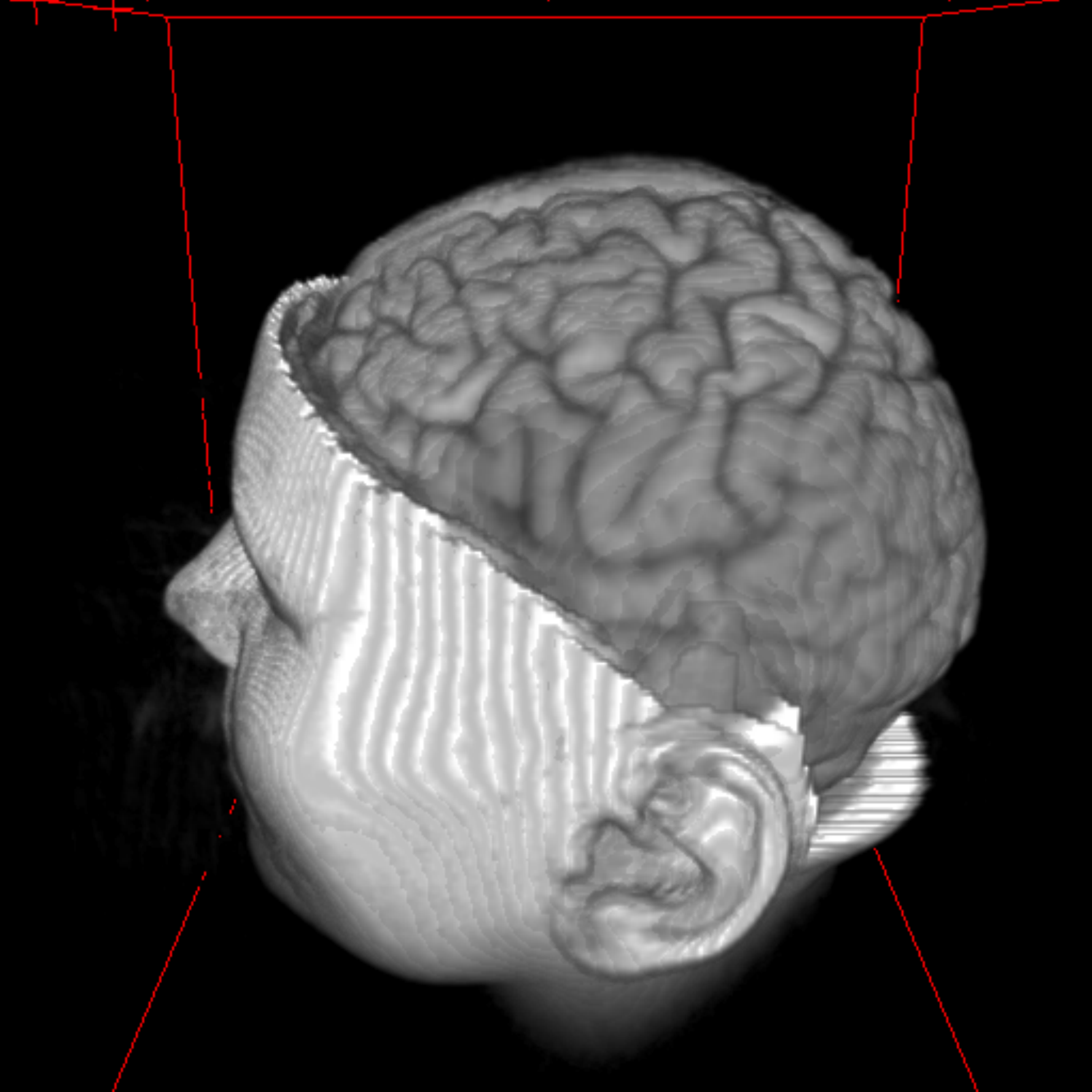}
\hspace{-0.cm}
\includegraphics[width=6.5cm]{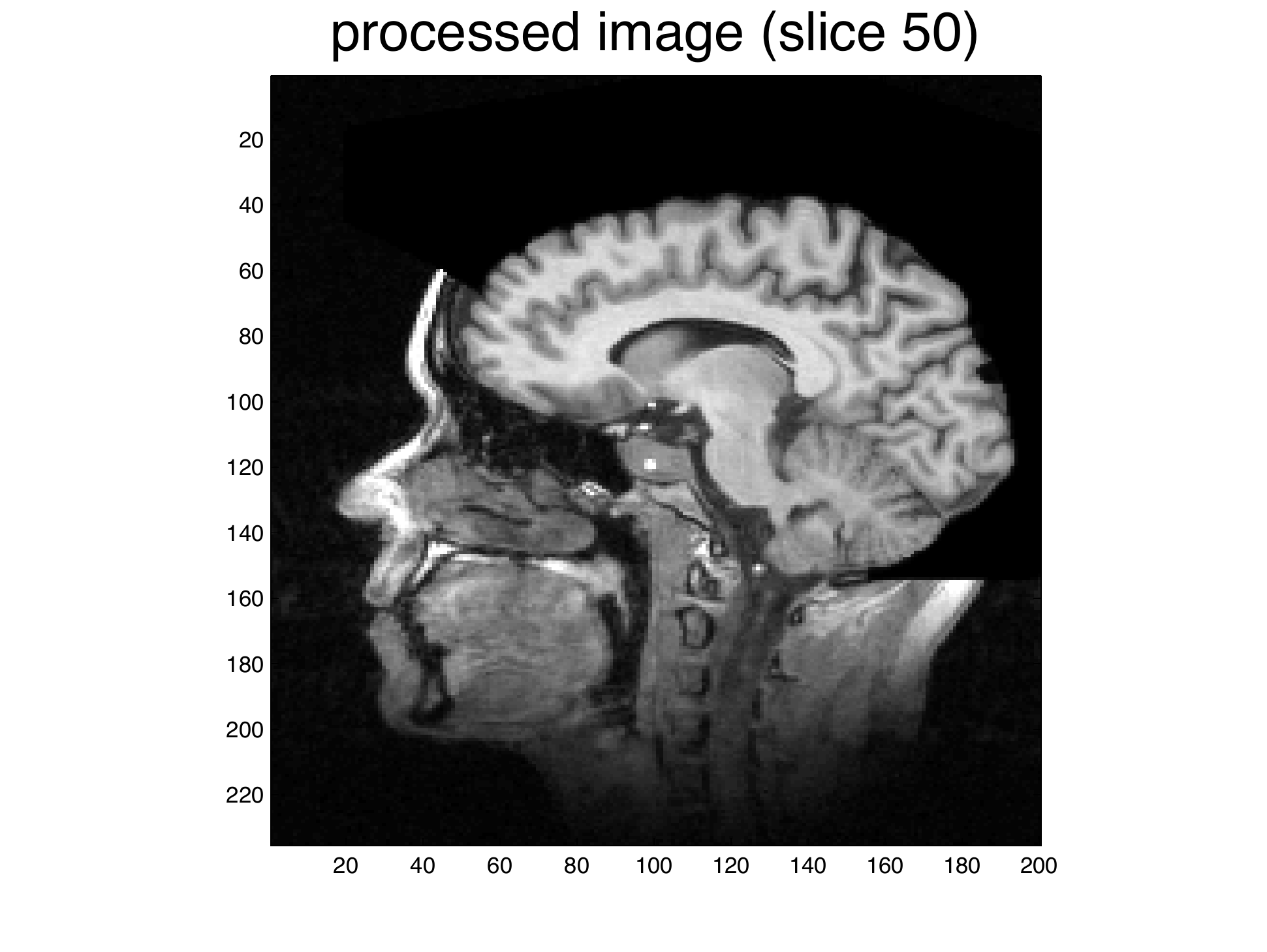}
\hspace{-1.3cm}
\includegraphics[width=6.5cm]{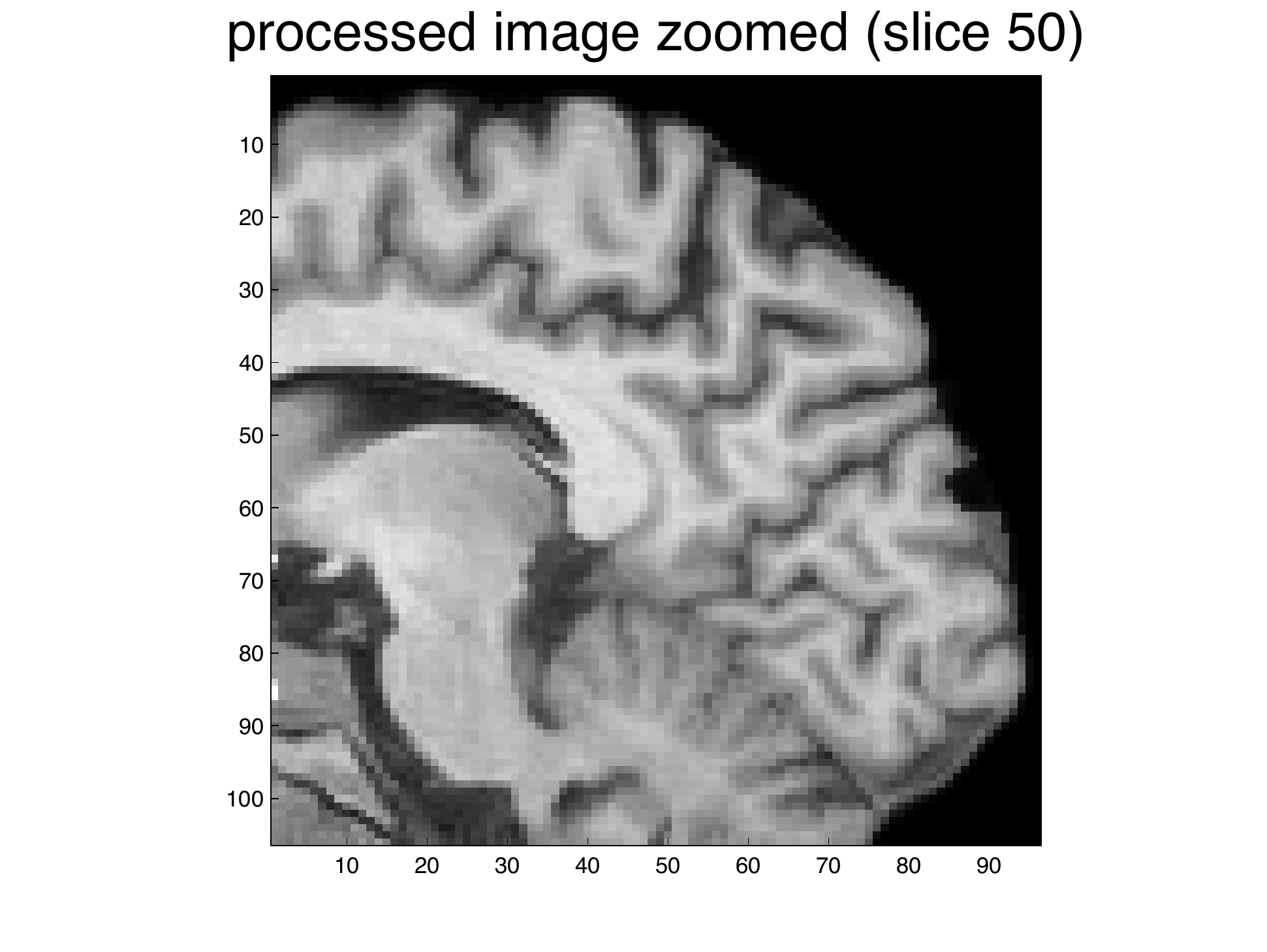}
\hspace{-1cm}
\end{center}
\caption{Top: MRI data. Bottom: Data processed via 3D Edge Enhancing diffusion, using AD-LBR. Left: 3D rendering of the original volume data and the processed volume, the 3D rendering was obtained using ImageJ 3D viewer \cite{schmid} (the effect of anisotropic diffusion is not much visible in this first representation). Center: slice of the original and processed volume. Right: details of the original and processed slices.
}
\label{fig:3Dmri}
\end{figure*}

\section*{Conclusion}
We introduced in this paper a new numerical scheme, AD-LBR, for anisotropic diffusion in image processing.
This scheme is non-negative, and its stencils have a limited support: 6 points in 2D, 12 points in 3D.  The former  property implies that our scheme respects the maximum principle of Alvarez, Gui\-chard, Lions and Morel, which is an essential feature of parabolic PDEs. 

AD-LBR outperformed all tested alternatives in a quantitative numerical experiment: a test case in which approximate numerical solutions are compared against a known analytical solution.
In a second qualitative test case, different schemes were used to enhance a fingerprint image. 
Our scheme appears here to close more efficiently the lines of the fingerprint, and to diffuse less orthogonally to the lines. 
This is precisely the purpose of the implemented PDE, coherence enhancing diffusion.
We also presented a 3-dimensional implementation as a proof of feasibility.

The construction of the stencils of the AD-LBR is both original and non-trivial. The computational load for this aspect of the algorithm is fortunately not dominant, thanks to the use of a tool from discrete geometry: lattice basis reduction.
The AD-LBR also allows to use larger time steps than most of its counterparts, in explicit discretizations of parabolic equations.

AD-LBR trivially extends to vector valued and matrix valued images, by applying it on each image component independently. (In other words, the coupling between image components lies in the construction of the common diffusion tensor $\Diff$, which AD-LBR regards as user input.) 
Future work will be devoted to the application of AD-LBR to the regularization of diffusion tensor fields, arising for instance from diffusion MRI, for which we expect it to be particularly appropriate: thanks to the scheme non-negativity, positive-definiteness is naturally preserved. 




\end{document}